\documentclass{amsart}
\usepackage{graphicx}
\usepackage[bookmarksnumbered,colorlinks,plainpages]{hyperref}
\usepackage{extarrows}
\usepackage{float}
\usepackage{amsthm}
\usepackage[all,2cell]{xy}
\usepackage{amsfonts}
\usepackage{amsmath}
\usepackage{amssymb}
\usepackage{xypic,leqno,amscd,amssymb,pstricks,latexsym,amsbsy,xypic,mathrsfs,verbatim}
\usepackage{multirow}
\usepackage{booktabs}
\usepackage{makecell}
\usepackage{cases}
\usepackage{appendix}
\usepackage{graphicx,colortbl}
\usepackage{tikz-cd}
\usepackage{etex}
\usepackage{mathtools}
\usepackage[bookmarksnumbered,colorlinks,plainpages]{hyperref}
\hypersetup{
	colorlinks=true,
	citecolor=blue
}
\newtheorem{theorem}{Theorem}[section]

\newtheorem{corollary}[theorem]{Corollary}
\newtheorem{lemma}[theorem]{Lemma}
\newtheorem{proposition}[theorem]{Proposition}
\newtheorem{definition-proposition}[theorem]{Definition-Proposition}

\theoremstyle{definition}
\newtheorem{definition}[theorem]{Definition}

\newtheorem{remark}[theorem]{Remark}
\newtheorem{example}[theorem]{Example}

\numberwithin{equation}{section}
\newtheorem{mainthm}{Theorem}

\newcommand{\OO}{{\mathcal O}}

\renewcommand{\L}{\mathbb{L}}
\newcommand{\X}{\mathbb{X}}

\newcommand{\Z}{\mathbb{Z}}

\renewcommand{\c}{\vec c}
\newcommand{\vdelta}{\vec{\delta}}
\newcommand{\y}{\vec y}
\newcommand{\x}{\vec x}

\newcommand{\w}{\vec{\omega}}
\newcommand{\z}{\vec z}

\newcommand{\cut}{\ar@{-}@[|(5)]}
\newcommand{\mr}{\hbox{\boldmath$\cdot$}}

\newcommand{\op}{\operatorname{op}\nolimits}
\newcommand{\Hom}{\operatorname{Hom}\nolimits}

\newcommand{\End}{\operatorname{End}\nolimits}
\newcommand{\Ext}{\operatorname{Ext}\nolimits}
\newcommand{\nExt}{\operatorname{nExt}\nolimits}

\renewcommand{\Im}{\operatorname{Im}\nolimits}
\renewcommand{\ker}{\operatorname{Ker}\nolimits}
\newcommand{\coker}{\operatorname{Cok}\nolimits}

\newcommand{\gl}{\operatorname{gl}.\operatorname{dim}\nolimits}

\newcommand{\RHom}{\mathbf{R}\strut\kern-.2em\operatorname{Hom}\nolimits}
\newcommand{\Hhom}[3]{\mathcal{H}om_{#1}(#2,#3)}

\DeclareMathOperator{\add}{\mathsf{add}}

\DeclareMathOperator{\thick}{\mathsf{thick}}
\DeclareMathOperator{\CM}{\mathsf{CM}}
\DeclareMathOperator{\ACM}{\mathsf{ACM}}
\DeclareMathOperator{\moduleCategory}{\mathsf{mod}} \renewcommand{\mod}{\moduleCategory}
\DeclareMathOperator{\coh}{\mathsf{coh}}

\DeclareMathOperator{\vect}{\mathsf{vect}}

\DeclareMathOperator{\proj}{\mathsf{proj}}

\tikzset{every picture/.style={line width=0.75pt}} 

\begin{document}

	\title[Ideal quotients of $n$-exact categories]{Ideal quotients of $n$-exact categories}
	
	\author[W. Weng] {Weikang Weng}

	\makeatletter \@namedef{subjclassname@2020}{\textup{2020} Mathematics Subject Classification} \makeatother
	
	
	\dedicatory{Dedicated to  Professor Yanan Lin on the occasion of his 70th birthday}
	\begin{abstract} 
	 	In this paper, we extend the notion of a removable subcategory
	 	to $n$-exact categories, and prove that an ideal quotient
	 	by a subcategory of projective-injectives inherits
	 	an $n$-exact structure from the original category
	 	if and only if the subcategory is removable.
	 	In the Frobenius case, the quotient remains Frobenius,
	 	and the stable $(n+2)$-angulated categories are equivalent. 
	 	We give a sufficient condition for a subcategory of projective-injectives to be  removable. 
	 	We also show that the ideal quotient construction is compatible
	 	with cluster tilting. Finally, we apply the construction to $2$-cluster tilting subcategories
	 	of the category of ACM bundles on a Geigle--Lenzing projective
	 	plane, obtaining quotients for both natural and relative
	 	$2$-exact structures.
	\end{abstract}
	
	\maketitle
	
	\section{Introduction}
	
	Exact categories in the sense of Quillen \cite{Q} provide a framework for homological
	algebra in additive categories equipped with a chosen class
	of kernel-cokernel pairs; see \cite{B}.
	The stable category of a Frobenius exact category, obtained by
	factoring out its projective-injective objects, carries a natural
	triangulated structure \cite{Ha}. A related question is when
	an ideal quotient  by a subcategory of
	projective-injective objects inherits the original exact structure. 
	Chen \cite[Theorem~3.1]{C} gave sufficient conditions for such a quotient
	of a Frobenius category to remain Frobenius.
	Demonet--Iyama \cite[Theorem~3.6]{DI} showed that the quotient
	inherits the exact structure precisely when the quotient functor sends
	inflations and deflations to categorical monomorphisms and
	epimorphisms, respectively. 
	
	Jasso \cite{Ja} introduced $n$-exact categories by formulating
	higher analogues of the axioms of exact categories,
	using $n$-exact sequences in place of short exact sequences.
	For $n=1$, these recover the usual exact categories. 
	The stable category of a Frobenius $n$-exact category carries
	a natural $(n+2)$-angulated structure in the sense of
	Geiss--Keller--Oppermann \cite{GKO}.  
	Herschend--Liu--Nakaoka \cite{HLN} introduced
	$n$-exangulated categories as higher analogues of extriangulated categories of Nakaoka--Palu \cite{NP}. They encompass both $n$-exact and $(n+2)$-angulated categories.
	Hu--Zhang--Zhou \cite[Theorem~3.1]{HZZ} established a criterion
	for an ideal quotient to inherit an $n$-exangulated structure.
	However, the resulting $n$-exangulated structure need not be $n$-exact, even when the original
	category is $n$-exact \cite[Example~3.4]{HZZ}.	
		
	In this paper, we give a direct treatment of ideal quotients
	of $n$-exact categories, including the Frobenius case,
	and establish related results. In particular, we apply the ideal quotient construction to Geigle--Lenzing projective planes. 
		
	Let $(\mathscr C,\mathscr X)$ be an $n$-exact category, and let $\mathscr P$ and $\mathscr I$ denote its subcategories of $\mathscr X$-projectives and $\mathscr X$-injectives, respectively. 	Let $\mathscr{F}$ be an additive subcategory of $\mathscr{C}$.	Denote by $\pi_\mathscr{F}\colon \mathscr{C}\rightarrow\mathscr{C}/[\mathscr{F}]$ the canonical functor. 
	The subcategory  $\mathscr{F}$  is called \emph{removable (for $\mathscr{X}$)} if $\mathscr{F}\subset \mathscr{P}\cap \mathscr{I}$ and  $\pi_\mathscr{F}$ sends $\mathscr{X}$-conflations in $\mathscr{C}$  to $n$-exact sequences in $\mathscr{C}/[\mathscr{F}]$. 					
		Let $\mathscr X_{\mathscr F}$ be the class of sequences of length
		$n+2$ in $\mathscr C/[\mathscr F]$ obtained from the images of
		$\mathscr X$-conflations under $\pi_{\mathscr F}$ by homotopy
		equivalences fixing both end terms. 		
		The following theorem characterizes when $\mathscr X_{\mathscr F}$
		defines an $n$-exact structure on $\mathscr C/[\mathscr F]$
		and describes the Frobenius case.
	
	\begin{mainthm} [Theorems~\ref{main thm1},~\ref{main thm2} and Remark~\ref{canonical equivalence}] \label{Thm A} 
		Let $(\mathscr C,\mathscr X)$ be an $n$-exact category and let $\mathscr F\subset\mathscr P\cap\mathscr I$ be an additive subcategory. The following hold.
		\begin{itemize}
			\item[(a)] The subcategory $\mathscr F$ is removable if and only if $(\mathscr C/[\mathscr F],\mathscr X_{\mathscr F})$ is an $n$-exact category.			
			In this case, $\pi_{\mathscr F}$ induces an isomorphism of abelian groups 
			\[
			\nExt^1_{\mathscr X}(C,A)\xrightarrow{\sim}
			\nExt^1_{\mathscr X_{\mathscr F}}(C,A)
			\]
			for any $A,C\in\mathscr C$. Here, $\nExt^1_{\mathscr X}(C,A)$ is the group of Yoneda equivalence classes of $\mathscr X$-conflations with end terms
			$A$ and $C$.  
			\item[(b)] If $(\mathscr C,\mathscr X)$ is Frobenius $n$-exact, then $\mathscr F$ is removable if and only $(\mathscr C/[\mathscr F],\mathscr X_{\mathscr F})$ is Frobenius $n$-exact. In this case, its subcategory of projective-injectives is precisely $\mathscr P/[\mathscr F]$. Moreover, $\pi_{\mathscr F}$ induces an
			$(n+2)$-angulated equivalence
			$$ \underline{\mathscr{C}}=
			\mathscr C/[\mathscr P]
			\xlongrightarrow{\sim}
			(\mathscr C/[\mathscr F])/(\mathscr P/[\mathscr F])=\underline{\mathscr C/[\mathscr F]}.
			$$	
		\end{itemize}
	\end{mainthm}
	 When $n=1$,  part~(a) recovers Demonet--Iyama's result \cite[Theorem~3.6]{DI}, and part~(b) gives its Frobenius form, recorded in \cite[Lemma 2.7]{CDR}; see Remark \ref{ramark:two exact structures}. 
		 Part~(a) also follows from the theory of $n$-exangulated
		 categories; see Remark~\ref{alternative quotient proof}.	 
		We give a direct proof using Herschend--Liu--Nakaoka's
		reformulation of Jasso's axioms for an $n$-exact category; see Definition \ref{reformulation of Jasso's axioms}. 	Proposition~\ref{redundant} shows that certain axioms
		in this formulation are redundant.
		After omitting them, the remaining axioms agree with Jasso's,
		except that closure under weak isomorphisms is replaced by
		$(\mathrm{EC}')$ and $(\mathrm{EI})$; see
		Remark~\ref{compare two axioms}.
		These conditions are convenient for the ideal quotient
		construction since closure under weak isomorphisms
		can be difficult to verify directly.
					
		For an additive subcategory $\mathscr{F}\subset \mathscr{P}\cap \mathscr{I}$, we give a sufficient condition for $\mathscr{F}$ to be a removable subcategory of $\mathscr{C}$; see Proposition \ref{A sufficient condition}. Together with  Theorem \ref{Thm A}(b), this recovers Chen's result \cite[Theorem 3.1]{C} when $n=1$. 
				
		A main source of $n$-exact categories is provided by
		$n$-cluster tilting subcategories of exact categories. We recall Jasso's construction
		\cite[Theorem~4.14]{Ja},
		with Kvamme's correction \cite[Theorems~2.6, 8.1]{Kv},
		in Theorem~\ref{n-exact from n-CT}.
		The following theorem shows that our ideal quotient construction
		is compatible with cluster tilting.
		
	\begin{mainthm}[Theorem \ref{Compatibility with cluster tilting}] \label{Thm B} 
			Let $\mathscr C$ be an $n$-cluster tilting subcategory of
			an exact category $(\mathscr A,\mathscr E)$, endowed with
			the induced $n$-exact structure $\mathscr X$ of
			Theorem~\ref{n-exact from n-CT}.
			If $\mathscr F$ is a removable subcategory of
			$(\mathscr A,\mathscr E)$, then the following hold.
			\begin{itemize}
				\item[(a)]
				$\mathscr F$ is a removable subcategory of $(\mathscr C,\mathscr X)$.
				\item[(b)]
				$\mathscr C/[\mathscr F]$ is an $n$-cluster tilting
				subcategory of $\mathscr A/[\mathscr F]$.
			\end{itemize}
			Moreover, the induced $n$-exact structure on
			$\mathscr C/[\mathscr F]$ coincides with
			$\mathscr X_{\mathscr F}$.
			In other words, the diagram of constructions (the vertical arrows represent the construction in
			Theorem~\ref{n-exact from n-CT})
			\[
			\begin{tikzcd}[column sep=2.4em,row sep=2.4em]
				(\mathscr A,\mathscr E)
				\arrow[r,"\pi_{\mathscr F}"]
				\arrow[d,Rightarrow]
				&
				(\mathscr A/[\mathscr F],\mathscr E_{\mathscr F})
				\arrow[d,Rightarrow]
				\\
				(\mathscr C,\mathscr X)
				\arrow[r,"{\pi_{\mathscr F}|_{\mathscr C}}"]
				&
				(\mathscr C/[\mathscr F],\mathscr X_{\mathscr F})
			\end{tikzcd}
			\]
			commutes. If moreover $(\mathscr A,\mathscr E)$ is Frobenius and
			$\mho_{\mathscr E}^{n}(\mathscr C)\subset\mathscr C$, then the top and bottom rows consist of Frobenius
			exact and Frobenius $n$-exact categories, respectively.
	\end{mainthm}
	
	We apply these results to a Geigle--Lenzing projective plane
	$\X$ of weight type $(p_1,\ldots,p_4)$, where $p_i\ge2$;
	see \cite{HIMO,CRW}. Let $\L$ be the abelian group on generators $\x_1,\ldots,\x_4,\c$ modulo the relations $p_1 \x_1= \dots=p_4 \x_4=\c$.  	
	Let $\coh\X$ be the abelian category of coherent sheaves on $\X$. Denote by $\vect \X$ the category of vector bundles on $\X$. Its subcategory $\ACM\X$ of arithmetically Cohen--Macaulay bundles on $\X$
	is Frobenius with the exact structure $\mathscr E$
	inherited from $\coh\X$. Its indecomposable projective-injective objects are precisely the line bundles $\OO(\x)$, $\x\in\L$. Each line bundle has a unique expression 
	$$L=\OO(\sum_{i=1}^{3}\lambda_i\x_i+a\x_4)$$
	such that $0\le \lambda_i <p_i$ for $1\le i \le 3$, and $a\in \Z$. We  call $L$ \emph{fading} if $\lambda_i=p_i-1$ for some $1\le i\le 3$, and \emph{persistent} otherwise. Let $\mathscr{F}$ denote the closure of the fading line bundles under finite direct sums and direct summands. The treatment here is inspired by \cite{KLM}.
	
	For a $2$-cluster  tilting subcategory $\mathscr{C}$ of $ \ACM \X$, let $\mathscr X$ and $\mathscr X^{\mathrm{vec}}$ denote the $2$-exact structures on $\mathscr C$ induced by
	the exact structures on $\ACM\X$ and $\operatorname{vect}\X$ inherited from $\operatorname{coh}\X$, respectively.
	Let $\mathscr X^{\mathrm{rel}}$ consist of $\mathscr X^{\mathrm{vec}}$-conflations, whose first differential
	is $\mathscr F$-monic and the last one is $\mathscr F$-epic. It follows that	 $(\mathscr C, \mathscr X^{\mathrm{rel}})$ is a $2$-exact category by Proposition \ref{rel exact}. 
	The following result illustrates the ideal quotient constructions in this setting. A key step of its proof is to show that each ACM bundle has Properties (L) and (R); see  Definition \ref{Properties (L) and (R)} and Proposition \ref{Property (L) and (R) for ACM bundles}. 
	
	\begin{mainthm}[Theorems \ref{main thm 3}, \ref{main thm 4} and  Proposition \ref{relative criterion}] \label{Thm C} 
		Let $\mathscr{C}$ be a $2$-cluster tilting subcategory of $\ACM \X$. The following hold. 
		\begin{itemize}
			\item[(a)] The subcategory $\mathscr F\subset\mathscr C$ is  removable for both $\mathscr X$ and $\mathscr X^{\mathrm{rel}}$. Moreover, both  $(\mathscr C/[\mathscr F],\mathscr X_{\mathscr F})$ and $(\mathscr C/[\mathscr F],\mathscr X_{\mathscr F}^{\mathrm{rel}})$ are $2$-exact categories.  	  If $\mho^2_{\mathscr{E}}(\mathscr{C})\subset\mathscr{C}$, then $(\mathscr C/[\mathscr F],\mathscr X_{\mathscr F})$ is a Frobenius $2$-exact category.
		\item[(b)] We have $\mathscr X\subset\mathscr X^{\mathrm{rel}}$ and equality holds if and only if
		$\min\{p_1,p_2,p_3\}=2$.
		\end{itemize}
	\end{mainthm}
	
	The paper is organized as follows.
	In Section~\ref{sec: Higher homological algebra}, we recall
	basic notions and facts from higher homological algebra.
	We also show that certain axioms in Herschend--Liu--Nakaoka's
	reformulation of Jasso's axioms for an $n$-exact category
	are redundant.
	In Section~\ref{sec: Ideal quotients of $n$-exact categories},
	we establish the ideal quotient theorem and show that
	the construction is compatible with cluster tilting,
	as stated in Theorems~\ref{Thm A} and~\ref{Thm B}.
	In Section~\ref{sec: Geigle--Lenzing projective planes},
	we apply these results to $2$-cluster tilting subcategories
	of $\ACM\X$ and prove Theorem~\ref{Thm C}.
	We conclude with explicit examples of Frobenius $2$-exact
	categories and their ideal quotient constructions.

	\section{Higher homological algebra}
	\label{sec: Higher homological algebra}	
	In this section, we recall some basic notions and facts on higher homological algebra needed in this paper. We show that certain axioms in Herschend--Liu--Nakaoka's reformulation of the definition
	of an $n$-exact category are redundant.  This formulation is convenient for our purposes.
	Throughout this paper, let $n\ge 1$ be a fixed integer. All subcategories
	are assumed to be full. For an additive category $\mathscr{C}$, we write $\mathscr{C}( - , - )$ for $\Hom_{\mathscr{C}}( - , -)$.
	
		\subsection{Exact categories}
	We use the terminology of \cite{B,Ke}.  A \emph{kernel-cokernel pair} $(i,p)$ in an additive category $\mathscr{A}$ is a pair of composable morphisms $X\xrightarrow{i}Y\xrightarrow{p}Z$ such that $i$ is a kernel of $p$ and $p$ is a cokernel of $i$. An \emph{exact structure} $\mathscr{E}$ on $\mathscr{A}$ is a chosen class of kernel-cokernel pairs, closed under isomorphisms,  which satisfies the axioms of  \cite[Definition 2.1]{B}. 	An \emph{exact category}  $(\mathscr{A},\mathscr{E})$ in the sense of Quillen  is an additive category $\mathscr{A}$ endowed with a specified exact structure $\mathscr{E}$. 
	The pair $(i,p)$ in $\mathscr{E}$ is called an  \emph{$\mathscr{E}$-conflation}, while $i$ is called an \emph{$\mathscr{E}$-inflation} and $p$ is  called an \emph{$\mathscr{E}$-deflation}.

	An additive category $\mathscr{A}$ is  \emph{weakly idempotent complete} if every split epimorphism has a kernel, or equivalently, if every split monomorphism has a cokernel; see \cite[Lemma 7.1]{B}.	
	If $(\mathscr{A},\mathscr{E})$ is a weakly idempotent complete exact category, then it follows that all split epimorphisms are $\mathscr{E}$-deflations and all split monomorphisms are $\mathscr{E}$-inflations; see \cite[Corollary 7.5]{B}.

	\subsection{$n$-exact categories} The notion of $n$-exact categories was introduced by Jasso \cite{Ja} as a natural generalization of exact categories. Let us recall the definitions and basic properties. We refer to \cite{HLN,Ja} for more information.
	
	Let  $\mathscr{C}$ be an additive category.	Denote by ${\bf{C}}_{\mathscr{C}}$  the category of complexes in $\mathscr{C}$. As its full subcategory, define ${\bf{C}}_{\mathscr{C}}^{n+2}$ to be the category of complexes in ${\bf{C}}_{\mathscr{C}}$ whose components are zero in the degrees outside of $\{0,\ldots,n+1\}$. Namely, an object in ${\bf{C}}_{\mathscr{C}}^{n+2}$ is a complex of the form
	\[{{X}^{\mr}}: \quad X^0 \xrightarrow{d_X^0} X^1 \xrightarrow{d_X^1}  \cdots \xrightarrow{d_X^{n-1}} X^{n} \xrightarrow{d_X^n} X^{n+1}\]

	The homotopy relation on morphism sets is defined in the usual way. Specifically, two morphisms $f^{\mr},g^{\mr}\in {\bf{C}}_{\mathscr{C}}^{n+2}({X}^{\mr},{Y}^{\mr})$ are \emph{homotopic}, written $f^{\mr}\sim g^{\mr}$, if there are  morphisms  $h^i\colon X^i\to Y^{i-1}$, $1\le i\le n+1$, such that
	\begin{eqnarray*}
		g^0-f^0&=&h^1 d_X^0,\\
		g^i-f^i&=&d_Y^{i-1}h^i+h^{i+1} d_X^i\quad (1\le i\le n),\\
		g^{n+1}-f^{n+1}&=&d_Y^nh^{n+1}.
	\end{eqnarray*}
	 A morphism $f^{\mr}\in {\bf{C}}_{\mathscr{C}}^{n+2}({X}^{\mr},{Y}^{\mr})$ is called a \emph{homotopy equivalence} if there exists a morphism $g^{\mr}\in {\bf{C}}_{\mathscr{C}}^{n+2}({Y}^{\mr},{X}^{\mr})$ such that  $g^{\mr}\circ f^{\mr}\sim 1_{{X}^{\mr}}$ and $f^{\mr}\circ g^{\mr}\sim 1_{{Y}^{\mr}}$. Such a morphism $g^{\mr}$ is called the  \emph{homotopy inverse} of $f^{\mr}$. A complex is \emph{contractible} if its identity
	 morphism is null-homotopic.

	\begin{definition}\label{DefMapCone} \cite[Definition 2.27]{HLN}
		For a morphism $f\in{\bf{C}}_{\mathscr{C}}^{n+2}({X}^{\mr},{Y}^{\mr})$ satisfying $f^0=1_{A}$ for some $A=X^0=Y^0$, its \emph{mapping cone} is defined to be the complex
		 \begin{equation}\label{mapping cone}
		{{M}}^{\mr}_f\colon\quad  X^1\overset{d_M^0}{\rightarrow}X^2\oplus Y^1\overset{d_M^1}{\rightarrow}X^3\oplus Y^2\overset{d_M^2}{\rightarrow}\cdots\overset{d_M^{n-1}}{\rightarrow}X^{n+1}\oplus Y^n\overset{d_M^n}{\rightarrow}Y^{n+1},
		 \end{equation}		 
		where
		$$d_M^0=\left[\begin{array}{c}-d_X^1\\ f^1\end{array}\right],~ d_M^i= \left[\begin{array}{cc}-d_X^{i+1}&0\\ f^{i+1}&d_Y^i\end{array}\right]~(1\le i\le n-1),~	d_M^n= \left[\begin{array}{cc}f^{n+1}&d_Y^n\end{array}\right].$$
		The \emph{mapping cocone} is defined dually, for morphism $g^{\mr}$ in ${\bf{C}}^{n+2}_{\mathscr{C}}$  satisfying $g^{n+1}=1$.
	\end{definition}
	
	Let $\mathscr{C}$ be an additive category and $f\colon A\rightarrow B$ a morphism in $\mathscr{C}$. A \emph{weak cokernel} of $f$ is a morphism
	$g\colon B\rightarrow C$ such that for any $X\in\mathscr{C}$ the induced sequence 
	$$\mathscr{C}(C,X)\xrightarrow{~\mathscr{C}(g,X)~}\mathscr{C}(B,X)\xrightarrow{~\mathscr{C}(f,X)~}\mathscr{C}(A,X)$$
	of abelian groups is exact. Equivalently, $g$ is a weak cokernel of $f$ if $gf=0$ and for each morphism
	$h\colon B\rightarrow X$ such that $hf=0$, there exists a (not necessarily unique) morphism
	$p\colon C\rightarrow X$ such that $h=pg$. These properties are subsumed in the following
	commutative diagram.
	\[
	\begin{tikzcd}
		A\rar{f}\drar[swap]{0}&B\rar{g}\dar{ h}
		&C\dlar[dotted]{ p}\\&X
	\end{tikzcd}
	\]
	Clearly, a weak cokernel $g$ of $f$ is a cokernel of $f$ if and only if $g$ is an epimorphism.
	The concept of a \emph{weak kernel} is defined dually.
	
		 Let $d_X^0\colon X^0 \to X^1$  be a morphism in $\mathscr{C}$. An \emph{$n$-cokernel} of $d_X^0$ is a sequence
		\[ (d_X^1,\dots,d_X^n):\quad X^1 \xrightarrow{d_X^1} X^2 \xrightarrow{d_X^2} \cdots \xrightarrow{d_X^{n-1}} X^{n} \xrightarrow{d_X^n} X^{n+1}\]
		of morphisms in $\mathscr{C}$ such that for every $X \in \mathscr{C}$ the induced sequence
		\[0 \rightarrow \mathscr{C}(X^{n+1}, X) \xrightarrow{} \mathscr{C}(X^{n}, X) \xrightarrow{} \cdots \xrightarrow{} \mathscr{C}(X^1, X) \xrightarrow{} \mathscr{C}(X^{0}, X)\]
		of abelian groups is exact. That is, the morphism $d_X^i$ is a weak cokernel of $d_X^{i-1}$ for all $1\le i \le n-1$ and $d_X^n$ is a cokernel of $d_X^{n-1}$. In this case, the sequence
		\[ X^0 \xrightarrow{d_X^0} X^1 \xrightarrow{d_X^1} X^2 \xrightarrow{d_X^2} \cdots \xrightarrow{d_X^{n-1}} X^{n} \xrightarrow{d_X^n} X^{n+1}\] is called \emph{right $n$-exact}. The notions of \emph{$n$-kernel} and \emph{left $n$-exact}  are defined dually.
		
		  A sequence
		 ${X}^{\mr} \in {\bf{C}}_{\mathscr{C}}^{n+2}$
		 is called an \emph{$n$-exact sequence} if it is
		 both right $n$-exact and left $n$-exact.  A morphism	$f^{\mr}\colon X^{\mr}\to Y^{\mr}$ of $n$-exact sequences in $\mathscr{C}$ is called a \emph{weak isomorphism} if $f^k$ and
		 $f^{k+1}$  are isomorphisms for some $k\in\{0,1,\dots,n+1\}$, where we set $n+2:=0$. 
		 	
	Let $\mathscr{X}$ be a class of $n$-exact sequences in $\mathscr{C}$. We call its members  \emph{$\mathscr{X}$-conflations}. Given an $\mathscr{X}$-conflation
	\[
	X^0\overset{d_X^0}{\rightarrow} X^1\xrightarrow{d_X^1}  X^2\xrightarrow{d_X^2} \cdots \xrightarrow{d_X^{n-1}} X^n\overset{d_X^n}{\to} X^{n+1},
	\]
	the morphism $d_X^0$ (resp. $d_X^n$) is  called an \emph{$\mathscr{X}$-inflation} (resp.  \emph{$\mathscr{X}$-deflation}).

	We frequently use the following reformulation of Jasso's
	definition for an $n$-exact category
	\cite[Definition~4.2]{Ja}, due to Herschend--Liu--Nakaoka
	\cite[Definition~4.21 and Proposition~4.23]{HLN}.
	This formulation is convenient for our purposes, since closure under weak isomorphisms can be difficult 
	to verify directly.

	\begin{definition} \label{reformulation of Jasso's axioms} Let $\mathscr{C}$ be an additive category, and let $\mathscr{X}$ be a class of $n$-exact sequences in $\mathscr{C}$. The pair $(\mathscr C,\mathscr X)$ is called an \emph{$n$-exact category} if the following conditions hold:
			\begin{itemize}
				\item [$(\rm EC')$] 
				
				Let ${X}^{\mr}, {Y}^{\mr}\in {\bf{C}}_{\mathscr{C}}^{n+2}$ with $A=X^0=Y^0$ and $C=X^{n+1}=Y^{n+1}$. If $f\colon{X}^{\mr}\to {Y}^{\mr}$ is a homotopy equivalence such that $f^{0}=1_A$ and $f^{n+1}=1_C$, then ${X}^{\mr} \in \mathscr{X}$  if and only if ${Y}^{\mr} \in \mathscr{X}$.
				\item [$(\rm E0)$] The sequence
				$0\rightarrow 0 \rightarrow \cdots\rightarrow 0\rightarrow 0$ is an $\mathscr{X}$-conflation.
				\item[{$(\rm E1)$}]  The class of $\mathscr{X}$-inflations is closed under composition.
					\item [$( \rm E1^{\op})$] The class of $\mathscr{X}$-deflations is closed under composition.
				\item [$(\rm E2')$] The dual of the following $(\rm E2'^{ op})$.
				\item [$(\rm {E2'}^{ op})$] 
				\begin{itemize}
					\item[{\rm (i)}] For any ${Y}^{\mr} \in \mathscr{X}$ and $f^{n+1}\colon X^{n+1}\to Y^{n+1}$, there is a commutative diagram as follows, such that the top sequence ${X}^{\mr} \in \mathscr{X}$.
					\begin{equation}\label{pullback}
						\begin{tikzcd}
							{X}^{\mr}\dar{f^{\mr}}	& X^0 \dar[equal]{} \rar[dotted]{d_X^0} & X^1 \rar[dotted]{d_X^1}\dar[dotted]{{}}& \cdots\rar[dotted]{d_X^{n-1}} & X^n\rar[dotted]{d_X^n} \dar[dotted]{} & X^{n+1} \dar{f^{n+1}}\\
							{Y}^{\mr}	& Y^0 \rar{d_Y^0} & Y^1\rar{d_Y^1} & \cdots\rar{d_Y^{n-1}}  & Y^n \rar{d_Y^n} & Y^{n+1}
						\end{tikzcd}
					\end{equation}
					\item[{\rm (ii)}] Let ${X}^{\mr},{Y}^{\mr}\in \mathscr{X}$ with $X^0=Y^0=A$. For any morphism $f^{\mr}\colon  {X}^{\mr} \to {Y}^{\mr}$ satisfying $f^0=1_{A}$,  the mapping cone ${M}^{\mr}_f\in\mathscr{X}$.
					\end{itemize}
					\item [$(\rm EI)$]  For any ${Y}^{\mr}\in \mathscr{X}$ and $n$-exact sequence  ${X}^{\mr}$ in ${\mathscr{C}}$,
					the following hold.
					\begin{itemize}
						\item[{\rm (i)}] If there is a morphism $f^{\mr}\colon{Y}^{\mr} \to {X}^{\mr}$ in which $f^0$ and $f^{n+1}$ are isomorphisms, then $f^{\mr} $ is a homotopy equivalence in ${\bf{C}}_{\mathscr{C}}^{n+2}$.
						\item[{\rm (ii)}] If there is a morphism $g^{\mr}\colon{X}^{\mr} \to {Y}^{\mr}$ in which $g^0$ and $g^{n+1}$ are isomorphisms, then $g^{\mr} $ is a homotopy equivalence in ${\bf{C}}_{\mathscr{C}}^{n+2}$.
					\end{itemize}
			\end{itemize}		
	\end{definition}

	The following observation  shows that the axioms $(\rm {E2'}^{ op})(ii)$ and its dual for $n$-exact categories are redundant.

	\begin{proposition} \label{redundant}
		Let $\mathscr{X}$ be a class of $n$-exact sequences in $\mathscr{C}$ satisfying $(\rm EC')$, $(\rm {E2'}^{})(i)$ and $(\rm {E2'}^{op})(i)$. Then $(\rm {E2'})(ii)$ and $(\rm {E2'}^{ op})(ii)$ hold. 
	\end{proposition}
	\begin{proof} 
		We only prove $(\rm {E2'}^{ op})(ii)$ since the other assertion can be shown dually. Let ${X}^{\mr},{Y}^{\mr}\in \mathscr{X}$ with $X^0=Y^0=A$, and
		  $f^{\mr}\colon  {X}^{\mr} \to {Y}^{\mr}$ be a morphism  with  $f^0=1_{A}$.
		 The proof is divided into the following two steps.
		
		\emph{Step $1$}: We show that the mapping cone 
			\[{M}^{\mr}_f\colon\quad X^1\xrightarrow{\begin{bmatrix}
					-d^1_X\\  f^1
			\end{bmatrix}} X^2\oplus Y^1\overset{d_M^1}{\rightarrow}\cdots\overset{d_M^{n-1}}{\rightarrow}X^{n+1}\oplus Y^n\xrightarrow{\begin{bmatrix}
			f^{n+1} & d_Y^n
		\end{bmatrix}} Y^{n+1} \]
	 is an $n$-exact sequence in $\mathscr{C}$. Here, 
		 $$d_M^i= \left[\begin{array}{cc}-d_X^{i+1}&0\\ f^{i+1}&d_Y^i\end{array}\right]\quad(1\le i\le n-1).$$
		 
		 We first verify that $d_M^{n-1}$ is a weak kernel of $d_M^n$. Let $$\begin{bmatrix}
		 	u & v
		 \end{bmatrix}^{\top}\colon Z^{n+1} \to X^{n+1}\oplus Y^n$$ be a morphism such that 
		 \begin{equation}\label{equ2.3.0}
		 	f^{n+1}u+d^n_Yv=0.
		 \end{equation}
 		Applying $(\rm {E2'}^{op})(i)$ to ${X}^{\mr}$ along $u$, we obtain a morphism 	 		
 		$g^{\mr}\colon Z^{\mr}\to {X}^{\mr}$ such that $Z^{\mr}\in  \mathscr{X}$, $g^0=1_A$ and $g^{n+1}=u$. Put $h^{\mr}=f^{\mr}g^{\mr}$.  Starting with $t^{n+1}=-v$, we construct a null-homotopy of $h^{\mr}$ by descending induction. Indeed, equality \eqref{equ2.3.0} gives $h^{n+1}=d_Y^nt^{n+1}$, and then by composing  $d^n_Z$ to  both sides, we obtain 
 		\begin{equation}\label{equ2.3.1} 
 			d_Y^n(h^n-t^{n+1}d_Z^n)=0.
 		\end{equation} Assume that  we have constructed the equality of form (\ref{equ2.3.1}) for $n=k$. Since $d^{k-1}_Y$ is a 
 		weak kernel of $d^{k}_Y$,  there exists a morphism $t^k\colon Z^k\to Y^{k-1}$ such that 
 		$$h^k=d_Y^{k-1}t^k+t^{k+1}d_Z^k.$$ Composing $d^{k-1}_Z$ to  both sides, we obtain $d_Y^{k-1}(h^{k-1}-t^{k}d_Z^{k-1})=0$ and the induction step follows. Thus we have $d_Y^0(h^0-t^1d_Z^0)=0$. Since $d_Y^0$ is a  monomorphism, we have $1_A=h^0=t^1d_Z^0.$ Thus $d_Z^0$ is a split monomorphism. 
 		By \cite[Proposition 2.6]{Ja}, the $n$-exact sequence $Z^{\mr}$ is contractible. Choose a morphism $s\colon Z^{n+1} \to Z^n$ with $d_Z^ns=1_{Z^{n+1}}$, and put $a=g^ns$. Then $$u=g^{n+1}d_Z^ns=d_X^ng^ns=d_X^na.$$
		It follows from (\ref{equ2.3.0}) that $d_Y^n(f^na+v)=0$. Hence there exists a morphism $b \colon Z^{n+1}\to Y^{n-1}$ such that $f^na+v=d_Y^{n-1}b$. If $n\ge 2$, this yields 
		\begin{equation}\label{formula1}
			\begin{bmatrix}
				u\\  v
			\end{bmatrix}=d_M^{n-1}\begin{bmatrix}
				-a\\  b
			\end{bmatrix}.
		\end{equation}
		If $n=1$, then we have 
		\begin{equation}\label{formula2}
			\begin{bmatrix}
			u\\  v
		\end{bmatrix}=d_M^{0}(-a+d_X^0b)
		\end{equation}	
		since $f^1d_X^0=d_Y^0$. Thus $d_M^{n-1}$ is a weak kernel of $d_M^{n}$. 
		
		Next we show that $d_M^{i-1}$ is a weak kernel of $d_M^i$ for any $1\le i <n$. Let $\begin{bmatrix}
			u & v
		\end{bmatrix}^{\top}\colon V \to X^{i+1}\oplus Y^i$ be a morphism such that $d_M^i\begin{bmatrix}
		u & v
		\end{bmatrix}^{\top}=0$, i.e. $$d_X^{i+1}u=0 \quad \text{and} \quad f^{i+1}u+d_Y^iv=0.$$ 
		Since $d_X^i$ is a weak kernel of $d_X^{i+1}$, we have  $u=d_X^ia$ for some $a\colon V\to X^i$. Thus 
		$d_Y^i(f^ia+v)=0$, and moreover $f^ia+v=d_Y^{i-1}b$ holds again for some $b\colon V\to Y^{i-1}$. 
		For $1\le i< n$, formula \eqref{formula1} holds with $n$
		replaced by $i$ when $i\ge2$, while \eqref{formula2} applies
		when $i=1$. In either case, we obtain a factorization through
		$d_M^{i-1}$.
		
		 Let $u\colon T \to X^1$ be a morphism such that $d_M^0u=0$. This implies that $u=d_X^0a$ for some  morphism $a\colon T \to X^0$ and thus $d_Y^0a=f^1u=0$. Since $d_Y^0$ is a monomorphism, we have that $a=0$ and so $u=0$. Thus $d_M^{0}$ is a monomorphism. Hence we have shown that ${M}^{\mr}_f$ is left $n$-exact.
		
		  Now we show right $n$-exactness of ${M}^{\mr}_f$. Take $1\le i\le n$ and a morphism
		$\begin{bmatrix}
			u & v
		\end{bmatrix}\colon  X^{i+1}\oplus Y^i\to V$  such that $\begin{bmatrix}
		u & v
		\end{bmatrix}d_M^{i-1}=0$. Equivalently, $$ud_X^i=vf^i \quad\text{and}\quad vd_Y^{i-1}=0;$$
		For $i=1$, the second equality follows by composing the first with $d_X^0$. Since $d_Y^i$ is a weak cokernel of $d_Y^{i-1}$, there exists a morphism $w\colon Y^{i+1}\to V$ such that $v=wd_Y^i$. Thus we have $(u-wf^{i+1})d_X^i=0$. 		
		If $i<n$, then $u-wf^{i+1}=rd_X^{i+1}$ for some morphism  $r\colon X^{i+2}\to V$. Thus $$\begin{bmatrix}
			u & v
		\end{bmatrix}=\begin{bmatrix}
		-r & w
		\end{bmatrix}d_M^i.$$ 		
		If $i=n$, then the epimorphism $d_X^n$ can be cancelled on the right, and hence
		$u=wf^{n+1}$. Moreover, we have $\begin{bmatrix}
			u & v
		\end{bmatrix}=wd_M^n$. Finally, $d_M^n$ is an epimorphism since $d_Y^n$ is an
		epimorphism. This proves that $M^{\mr}_f$ is
		$n$-exact.

		\emph{Step $2$}: We show that ${M}^{\mr}_f\in \mathscr{X}$. 		
			The preceding argument uses only $(\rm {E2'}^{op})(i)$. Dually, one can show that $(\rm {E2'})(i)$ implies that the mapping cocone of a
		morphism of $\mathscr{X}$-conflations with the identity in the degree $n+1$ is an $n$-exact sequence. Applying $(\rm {E2'})(i)$ to ${Y}^{\mr}$ along $d_X^0\colon A\to X^1$, we obtain  a morphism 	 		
		$p^{\mr}\colon Y^{\mr}\to {W}^{\mr}$ such that $W^{\mr}\in  \mathscr{X}$, $p^0=d_X^0$ and $p^{n+1}=1_{Y^{n+1}}$. 
		Its mapping cocone $N_p^{\mr}$ is therefore $n$-exact. 		
		Consider  the 
		morphism  $y^{\mr}\colon Y^{\mr}\to {M}^{\mr}_f$ given by
		\[
		y^0=d_X^0,\quad
		y^i=\begin{bmatrix}0\\1_{Y^i}\end{bmatrix}\ (1\leq i\leq n),
		\quad y^{n+1}=1_{Y^{n+1}}.
		\] 		
		
		To show ${N}^{\mr}_y$ is $n$-exact, it suffices to show that
		 there is an isomorphism 
		\begin{equation}\label{N_y an isomorphism}
			N_y^{\mr}\simeq
			X^{\mr}\oplus\bigoplus_{1\le i\le n} \Delta^i(Y^i),
		\end{equation}
		of complexes, where for an object $E\in\mathscr{C}$,  $\Delta^i(E)$ denotes $E\xrightarrow{1_E}E$ in degrees
		$i,i+1$ and zero elsewhere.	In fact, the complexes $\Delta^i(E)$ are contractible for $1\le i\le n$. 
			
		For $n=1$, the maps
		\[
		u^0=1_A,\quad
		u^1=\begin{bmatrix}1&f^1\\0&1\end{bmatrix},
		\quad
		u^2=\begin{bmatrix}-1&0\\0&1\end{bmatrix}
		\]
		give an isomorphism
		$N_y^{\mr}\simeq X^{\mr}\oplus \Delta^1(Y^1)$. 
		Hence $N_y^{\mr}$ is exact.
		
		For $n\ge2$, the mapping cocone $N_y^{\mr}$ is given by
		\begin{align*}{N}^{\mr}_y\colon\quad A&\xrightarrow{\begin{bmatrix}
				-d^0_Y\\  d_X^0
		\end{bmatrix}} Y^1\oplus X^1 \xrightarrow{\begin{bmatrix}
		  -d_Y^1 & 0 \\  0 & -d_X^1 \\ 1_{Y^1} & f^1
	\end{bmatrix}} Y^2\oplus X^2 \oplus Y^1 \xrightarrow{d_{N_y}^{2}} \cdots \\
	&\xrightarrow{d_{N_y}^{n-1}}  Y^n\oplus X^n \oplus Y^{n-1} \xrightarrow{\begin{bmatrix}
			0 & -d_X^n & 0 \\  1_{Y^n} & f^n & d_Y^{n-1} \\
	\end{bmatrix}} X^{n+1} \oplus Y^{n}, \end{align*} where
		\[
	d_{N_y}^{i}=
	\begin{bmatrix}
		-d^{i}_Y &0 &0\\0&-d^i_X & 0 \\ 1_{Y^i} & f^i & d_Y^{i-1}
	\end{bmatrix} \quad (2\le i \le n-1) . \]					
		Denote the complex on the right of \eqref{N_y an isomorphism} by $U^{\mr}$, which is given by 
			\begin{align*}U^{\mr}\colon\quad A&\xrightarrow{\begin{bmatrix}
					0 \\  d_X^0
			\end{bmatrix}} Y^1\oplus X^1 \xrightarrow{\begin{bmatrix}
			0&0 \\ 0&d_X^1\\1_{Y^1}&0  
		\end{bmatrix}} Y^2\oplus X^2 \oplus Y^1 \xrightarrow{d_U^2} \cdots \\
			&\xrightarrow{d_U^{n-1}}  Y^n\oplus X^n \oplus Y^{n-1} \xrightarrow{\begin{bmatrix}0&d_X^n&0\\1_{Y^n}&0&0\end{bmatrix}} X^{n+1} \oplus Y^{n}, \end{align*}	where			
		$$ 
		d_U^i=\begin{bmatrix}
			0&0&0\\0&d_X^i&0\\1_{Y^i}&0&0
		\end{bmatrix}\quad (2\le i\le n-1).
		$$
		Define a morphism  $u^{\mr}\colon N_y^{\mr}\to U^{\mr}$ given by
		$$
		u^0=1_A,\quad
		u^1=\begin{bmatrix}1&f^1\\0&1\end{bmatrix},$$$$
		u^i=
		\begin{bmatrix}
			1&f^i&d_Y^{i-1}\\
			0&(-1)^{i-1}&0\\
			0&0&1
		\end{bmatrix}\ (2\leq i\leq n),\quad
		u^{n+1}=
		\begin{bmatrix}(-1)^n&0\\0&1\end{bmatrix}.
		$$
		 A direct calculation gives the following commutative diagram.
				\[
				\begin{tikzcd}[ampersand replacement=\&]
						 A\dar[equal]\rar{d_{N_y}^0}
						\& (N_y)^1 \dar{{u^1}}\rar{d_{N_y}^1}
						\& (N_y)^2 \dar{{u^2}}\rar{}
						\& \cdots \rar{}
						\& (N_y)^n\dar{{u^n}}\rar{d_{N_y}^n}
						\& X^{n+1}\oplus Y^n\dar{u^{n+1}}\\			
						 A\rar{d_U^0}
						\& U^1\rar{d_U^1}
						\& U^2\rar{}
						\& \cdots\rar{}
						\& U^n\rar{d_U^n}
						\& X^{n+1}\oplus Y^n
					\end{tikzcd}
				\]	
		Each $u^i$ is invertible and thus $u^{\mr}$ is the desired isomorphism. 

		Since $N_y^{\mr}$ and $N_p^{\mr}$ are $n$-exact, the dual
		of \cite[Lemma~4.17]{HLN}, applied to $ y^{\mr}\colon Y^{\mr}\to {M}^{\mr}_f$ and $ p^{\mr}\colon Y^{\mr}\to {W}^{\mr}$
		in both orders, gives morphisms
		\[
		r^{\mr}\colon M_f^{\mr}\longrightarrow W^{\mr}
		\quad \text{and}\quad
		s^{\mr}\colon W^{\mr}\longrightarrow M_f^{\mr}
		\]
		with
		$r^0=s^0=1_{X^1}$
		and 
		$r^{n+1}=s^{n+1}=1_{Y^{n+1}}$. Since $M_f^{\mr}$ and $W^{\mr}$ are $n$-exact,
		by \cite[Proposition~4.10(2)]{HLN}, we have
		\[
		s^{\mr}r^{\mr}\sim1_{M_f^{\mr}}
		\quad\text{and}\quad
		r^{\mr}s^{\mr}\sim1_{W^{\mr}}.
		\]
		Thus $r^{\mr}$ and $s^{\mr}$ are mutually homotopy inverses fixing the two endpoints. Finally,
		since $W^{\mr}\in\mathscr X$, axiom $(\mathrm{EC'})$ yields
		$M_f^{\mr}\in\mathscr X$.	
	\end{proof}
	
	\begin{remark}\label{compare two axioms} For an $n$-exact category in the sense of Jasso \cite[Definition~4.2]{Ja},
		 axiom \cite[Definition 4.2$(\mathrm{E2}^{\op})$]{Ja} 
		 may be replaced by
		$(\mathrm{E2'}^{\op})(\mathrm{i})$.
		Indeed,  by the dual of
		\cite[Proposition~4.8(i)$\Rightarrow$(iv)]{Ja}, \cite[Definition 4.2$(\mathrm{E2}^{\op})$]{Ja} implies
		$(\mathrm{E2'}^{\op})(\mathrm{i})$.
		Conversely, let $f^{\mr}\colon X^{\mr}\to Y^{\mr}$ be
		a morphism given by $(\mathrm{E2'}^{\op})(\mathrm{i})$.
		Step~1 in the proof of Proposition~\ref{redundant}
		shows that its mapping cone  $M_f^{\mr}$ is $n$-exact, and thus
		 the last $n$ squares of $f^{\mr}$ form an
		$n$-pullback diagram. This proves \cite[Definition 4.2$(\mathrm{E2}^{\op})$]{Ja}.
		Dually, \cite[Definition 4.2$(\mathrm{E2})$]{Ja} 
		may be replaced by $(\mathrm{E2'})(\mathrm{i})$.	
		
		After omitting the redundant axioms
		$(\mathrm{E2'})(\mathrm{ii})$ and
		$(\mathrm{E2'}^{\mathrm{op}})(\mathrm{ii})$,
		Definition~\ref{reformulation of Jasso's axioms} agrees with Jasso's definition
		\cite[Definition~4.2]{Ja}, modified as above, except that
		closure under weak isomorphisms is replaced by
		$(\mathrm{EC'})$ and $(\mathrm{EI})$.
		We shall therefore verify these two conditions in what follows.
	\end{remark}

		\begin{proposition}[Strong obscure axiom] \label{Strong obscure axiom} Let $(\mathscr{C},\mathscr{X})$ be an $n$-exact category, and suppose that $\mathscr{C}$  is weakly idempotent complete or $n\ge 2$. Set $f\in \mathscr{C}(X,Y)$ and $g\in \mathscr{C}(Y,Z)$. Then
			the following hold.
			\begin{itemize}
				\item [(a)] If $gf$ is an $\mathscr{X}$-inflation, then so is $f$.
				\item [(b)]	If $gf$ is an $\mathscr{X}$-deflation, then so is $g$.
			\end{itemize}
		\end{proposition}
	\begin{proof} By \cite[Proposition~4.34]{HLN}, every $n$-exact category is
		$n$-exangulated. Thus the assertion follows from
		\cite[Theorem~3.5]{Kla}. For the convenience of the reader, we include  a direct proof. We only show (b) as (a) can be proved similarly.
		For $n=1$, by assumption, $\mathscr{C}$  is weakly idempotent complete and this is \cite[Proposition 7.6]{B}. So we assume $n\ge 2$ and  that $gf$ is an $\mathscr{X}$-deflation. By $( \rm E1^{\op})$,  the factorization
		\[
		\begin{bmatrix}
			0 & g
		\end{bmatrix}
		=\begin{bmatrix}
			1_Z & g
		\end{bmatrix}
		\begin{bmatrix}
			gf&0\\
			0&1_Y
		\end{bmatrix}
		\begin{bmatrix}
			1_X&0\\
			-f&1_Y
		\end{bmatrix}
		\]
		shows that $\begin{bmatrix}
			0 & g
		\end{bmatrix}\colon X\oplus Y\to Z$ is an $\mathscr{X}$-deflation. Indeed, the first factor is a split
		epimorphism with kernel $\begin{bmatrix}
			-g & 1_Y
		\end{bmatrix}^\top$, the second is a direct sum of $\mathscr{X}$-deflations, and the third is an isomorphism. Choose an $\mathscr{X}$-conflation
		\[
		A_n\xrightarrow{d_n}\cdots \xrightarrow{d_3}A_2
		\xrightarrow{d_2}A_1
		\xrightarrow{\begin{bmatrix}
				a\\  b
		\end{bmatrix}}X\oplus Y
		\xrightarrow{\begin{bmatrix}
				0 & g
		\end{bmatrix}}Z.
		\]
		Since $\begin{bmatrix}
			a & b
		\end{bmatrix}^\top$ is a weak kernel of $\begin{bmatrix}
		0 & g
		\end{bmatrix}$, the composition $$X\xrightarrow{\begin{bmatrix}
			1 \\ 0
			\end{bmatrix}} X \oplus Y \xrightarrow{\begin{bmatrix}
			0 & g
		\end{bmatrix}} Z$$ vanishes and thus
		there exists a morphism
		$s\colon X\to A_1$ such that
		\[
		as=1_X\quad\text{and}\quad bs=0.
		\]
		
		We claim that $g$ has the following $n$-kernel:
		\[
		\begin{cases}
			A_n\longrightarrow\cdots\longrightarrow A_3
			\xrightarrow{\begin{bsmallmatrix}d_3\\0\end{bsmallmatrix}}
			A_2\oplus X
			\xrightarrow{\begin{bsmallmatrix}d_2&s\end{bsmallmatrix}}
			A_1\xrightarrow{b}Y,&n\geq3,\\[4pt]
			A_2\oplus X
			\xrightarrow{\begin{bsmallmatrix}d_2&s\end{bsmallmatrix}}
			A_1\xrightarrow{b}Y,&n=2.
		\end{cases}
		\]	
		 Since $\begin{bmatrix}
		 	a & b
		 \end{bmatrix}^\top$ is a weak kernel of $\begin{bmatrix}
		 0 & g
		 \end{bmatrix}$,  
		we have that $b$ is a weak kernel of $g$. Next, let $u\colon T\to A_1$ be a morphism
		such that $bu=0$. Then
		\[
		\begin{bmatrix}
			a\\  b
		\end{bmatrix}(u-sau)=0.
		\]
		Hence $u-sau=d_2v$ for some morphism $v\colon T\to A_2$, so $u$ factors
		through $\begin{bmatrix}
			d_2 & s
		\end{bmatrix}$. Finally, let $ \begin{bmatrix}
		v&   w
		\end{bmatrix}^\top\colon T'\to A_2\oplus X $ be a morphism such that
		$d_2v+sw=0$. Composing both sides with $a$,	we have	
		$w=a(d_2v+sw)=0$ and thus $d_2v=0$.
		If $n\geq3$, then $v$  factors through $d_3$; if $n=2$,
		the monomorphism $d_2$ gives $v=0$ and thus $\begin{bmatrix}
			d_2 & s
		\end{bmatrix}$ is a monomorphism. The remaining exactness conditions 
		follow from the chosen $\mathscr X$-conflation.		
		Thus $g$ has an $n$-kernel. Since $gf$ is an $\mathscr{X}$-deflation,  $g$ is an $\mathscr{X}$-deflation by the dual of 
		\cite[Proposition~4.11]{Ja}.
	\end{proof}
	
		\begin{definition} \cite{Ja} Let $(\mathscr{C},\mathscr{X})$ be an $n$-exact category.
		\begin{itemize}
			\item [(a)] An object $P\in\mathscr{C}$ is called  
			\emph{$\mathscr{X}$-projective} if for each $\mathscr{X}$-deflation $f\colon M\to
			N$, the induced morphism
			$${\mathscr{C}}(P,f)\colon {\mathscr{C}}(P,M)\to{\mathscr{C}}(P,N)$$	
			is an epimorphism. The notion of 	\emph{$\mathscr{X}$-injective} is defined dually. 			
			Denote by $\mathscr{P}$ (resp. $\mathscr{I}$)   the subcategory of $\mathscr{X}$-projectives (resp. $\mathscr{X}$-injectives). 
			
			\item [(b)] We say that $\mathscr{C}$ \emph{has enough $\mathscr{X}$-projectives} if for any $M \in \mathscr{C}$, there exists an $\mathscr{X}$-conflation
			$
			N\to { P^1\xrightarrow{}   \cdots \xrightarrow{} P^n\to M},
			$
			such that $P^1,\dots,P^n\in \mathscr{P}$.
			The notion of  \emph{having  enough $\mathscr{X}$-injectives} is defined dually.
			\item [(c)] An $n$-exact category is \emph{Frobenius} if it has enough $\mathscr{X}$-projectives, enough $\mathscr{X}$-injectives, and the class of $\mathscr{X}$-projectives coincides with that of $\mathscr{X}$-injectives.
		\end{itemize}
	\end{definition}
	
		Let $\mathscr{C}$ be an additive category, and $\mathscr{F}$ be a subcategory of $\mathscr{C}$. Recall that a 
	morphism $f\colon A\to B$ in $\mathscr{C}$ is called  \emph{$\mathscr{F}$-monic} if the induced morphism
	$\mathscr{C}(B,F)\xrightarrow{}{\mathscr{C}}(A,F)$ is an epimorphism for all $F\in\mathscr{F}$.
	The notion of \emph{$\mathscr{F}$-epic} is defined dually.

	\begin{lemma} \label{F-monic and F-epic}
		Let $(\mathscr{C},\mathscr{X})$ be an $n$-exact category, and $\mathscr{F}\subset\mathscr{C}$ a subcategory. Let $f^{\mr}\colon X^{\mr} \to Y^{\mr}$ be a morphism of $\mathscr{X}$-conflations. The following hold.
		\begin{itemize}
			\item [(a)] If $f^0=1$ and $d_Y^n$ is $\mathscr{F}$-epic, then $d_X^n$ is $\mathscr{F}$-epic.
			\item [(b)]  If $f^{n+1}=1$ and $d_X^0$ is $\mathscr{F}$-monic, then $d_Y^0$ is $\mathscr{F}$-monic. 
		\end{itemize}
	\end{lemma}
	\begin{proof}  We prove (a). Let $M_f^{\mr}$ be the mapping cone of $f^{\mr}$; see Definition \ref{DefMapCone}.
		Let $u\colon F\to X^{n+1}$ be a morphism with $F\in \mathscr{F}$. Since $d_Y^n$ is $\mathscr{F}$-epic, there exists $v\colon F\to  Y^{n}$ such that $f^{n+1}u=d_Y^{n}v$, that is, ${\begin{bmatrix}
				u & -v
		\end{bmatrix}}^{\top}$ is killed by $d_M^n$. Hence ${\begin{bmatrix}
				u & -v
		\end{bmatrix}}^{\top}$ factors through $d_M^{n-1}$. Comparing first components gives $u=-d_X^nt$ for some  $t\colon F\to X^n$. Hence $d_X^n$ is $\mathscr F$-epic, proving (a). The proof of (b) is dual.	
	\end{proof}
	
	Let $(\mathscr C,\mathscr X)$ be an $n$-exact category, and let
	$\mathscr F$ be a subcategory of $\mathscr C$. With a view to
	constructing an $n$-exact substructure in which every object of
	$\mathscr F$ is projective and injective, let
	$\mathscr X^{\mathrm{rel}}$ consist of all
	$\mathscr X$-conflations
	\[
	X^{\mr}\colon\quad
	X^0\xrightarrow{d_X^0}X^1\rightarrow\cdots
	\rightarrow X^n\xrightarrow{d_X^n}X^{n+1}
	\]
	such that $d_X^0$ is $\mathscr F$-monic and $d_X^n$ is
	$\mathscr F$-epic; compare \cite[Lemmas 4.2 and 4.3]{LZ}.
	
	\begin{proposition}\label{rel exact}
			Let $(\mathscr{C},\mathscr{X})$ be an $n$-exact category, and $\mathscr{F}$ be a subcategory of $\mathscr{C}$. 
			Then the pair $(\mathscr C,\mathscr X^{\mathrm{rel}})$ is an $n$-exact category. 
	\end{proposition}
	\begin{proof}
		We verify the axioms in Definition~\ref{reformulation of Jasso's axioms}. 		
		To show $(\mathrm{EC}')$,
		let $f\colon{X}^{\mr}\to {Y}^{\mr}$ be a homotopy equivalence in ${\bf{C}}_{\mathscr{C}}^{n+2}$ such that $f^{0}=1$ and $f^{n+1}=1$. 
		We show that ${X}^{\mr} \in \mathscr{X}^{\mathrm{rel}}$  if and only if ${Y}^{\mr} \in \mathscr{X}^{\mathrm{rel}}$. We may assume ${Y}^{\mr} \in \mathscr{X}^{\mathrm{rel}}$, since the other case can be shown dually. Applying $(\mathrm{EC}')$ for $(\mathscr{C},\mathscr{X})$, we have ${X}^{\mr} \in \mathscr{X}$. By Lemma \ref{F-monic and F-epic}(a) for $f^{\mr}$, 		$d_X^n$ is $\mathscr{F}$-epic since $d_Y^n$ is $\mathscr{F}$-epic.  We can choose a homotopy inverse $g^{\mr}\colon Y^{\mr}\to X^{\mr}$	of $f^{\mr}$ such that $g^0=1$ and $g^{n+1}=1$ by \cite[Claim 2.8]{HLN}.  Applying Lemma \ref{F-monic and F-epic}(b) to $g^{\mr}$, 	$d_X^0$ is $\mathscr{F}$-monic since $d_Y^0$ is $\mathscr{F}$-monic. Thus ${X}^{\mr} \in \mathscr{X}^{\mathrm{rel}}$.		Hence $(\mathrm{EC}')$ holds.		
		Axioms $(\mathrm E0)$ and $(\mathrm{EI})$ are inherited from $\mathscr X$.
		
		Let $V^{\mr}\in\mathscr X^{\mathrm{rel}}$ and
		$c\colon C\to V^{n+1}$. By
		$(\mathrm{E2'}^{\mathrm{op}})(\mathrm i)$ for $\mathscr X$,
		there is a morphism
		$\alpha^{\mr}\colon U^{\mr}\rightarrow V^{\mr}$
		of $\mathscr X$-conflations such that
		$\alpha^0=1$ and $\alpha^{n+1}=c$. Since
		$d_V^0=\alpha^1d_U^0$ and  $d_V^0$ is $\mathscr F$-monic,  $d_U^0$ is
		$\mathscr F$-monic. By
		Lemma~\ref{F-monic and F-epic}(a), $d_U^n$ is
		$\mathscr F$-epic since $d_V^n$ is
		$\mathscr F$-epic. Thus
		$U^{\mr}\in\mathscr X^{\mathrm{rel}}$, proving
		$(\mathrm{E2'}^{\mathrm{op}})(\mathrm i)$. Its dual follows
		similarly. Proposition~\ref{redundant} then gives axioms
		$(\mathrm{E2'})(\mathrm{ii})$ and $(\mathrm{E2'}^{\mathrm{op}})(\mathrm{ii})$.
		
		It remains to verify $(\mathrm E1)^{\mathrm{op}}$.
		Let $\delta\colon A\to B$ and $\gamma\colon B\to C$ be
		$\mathscr X^{\mathrm{rel}}$-deflations. Take
		$X^{\mr},Y^{\mr}\in\mathscr X^{\mathrm{rel}}$ with
		$d_X^n=\delta$ and $d_Y^n=\gamma$. By
		$(\mathrm E1)^{\mathrm{op}}$ for $\mathscr X$, choose
		$Z^{\mr}\in\mathscr X$ with $d_Z^n=\gamma\delta$.
		Since $\mathscr F$-epic morphisms are closed under
		composition, $d_Z^n$ is $\mathscr F$-epic.		
		By the weak kernel properties of $Y^{\mr}$, the morphisms
		$\alpha^n=\delta$ and $\alpha^{n+1}=1_C$ extend to a morphism
		$\alpha^{\mr}\colon Z^{\mr}\longrightarrow Y^{\mr}.$
		By $(\mathrm{E2'})(\mathrm{ii})$ for $\mathscr X$, its
		mapping cocone $N_\alpha^{\mr}$ belongs to $\mathscr X$, which is given by
			\[
			N_\alpha^{\mr}\colon\quad
			Z^0
			\xrightarrow{\begin{bmatrix}
				-d_Z^0\\ \alpha^0
			\end{bmatrix}}
			Z^1\oplus Y^0
			\xrightarrow{d_N^1}
			\cdots
			\xrightarrow{d_N^{n-1}}
			A\oplus Y^{n-1}
			\xrightarrow{\begin{bmatrix}
			\alpha^n&d_Y^{n-1}
		\end{bmatrix}}
			B.
			\]		
		The morphisms $\beta^n=\begin{bmatrix}1_A &0\end{bmatrix}^{\top}$
		and $\beta^{n+1}=1_B$ extend to a morphism
		$\beta^{\mr}\colon X^{\mr}\rightarrow N_\alpha^{\mr}.$
		Since $d_X^0$ is $\mathscr F$-monic,
		 $d_N^0$ is $\mathscr F$-monic by Lemma~\ref{F-monic and F-epic}(b).		
		Let $h\colon Z^0\to F$ with $F\in\mathscr F$. We have
		$h=-u d_Z^0+v\alpha^0$
		for some morphisms $u\colon Z^1\to F$ and $v\colon Y^0\to F$.
		Since $d_Y^0$ is $\mathscr F$-monic, 
		$v=w d_Y^0$ for some $w\colon Y^1\to F$. Thus
		$d_Y^0\alpha^0=\alpha^1d_Z^0$ yields
		$h=(-u+w\alpha^1)d_Z^0.$
		Thus $d_Z^0$ is $\mathscr F$-monic, and hence
		$Z^{\mr}\in\mathscr X^{\mathrm{rel}}$.
		Therefore $\gamma\delta$ is an
		$\mathscr X^{\mathrm{rel}}$-deflation, proving
		$(\mathrm E1)^{\mathrm{op}}$. The proof of $(\mathrm E1)$ is dual. 
	\end{proof}

		\subsection{$n$-cluster tilting subcategories} 				
		Let $\mathscr{A}$ be an additive category, and $\mathscr{C}$ be a subcategory of $\mathscr{A}$.	
	A morphism $f\colon A\to B$ in $\mathscr{A}$ is called a
	\emph{left $\mathscr C$-approximation} of $A$ if
	$B\in\mathscr C$ and $f$ is $\mathscr C$-monic.
	The subcategory $\mathscr C$ is called
	\emph{covariantly finite} if every object of $\mathscr A$
	admits a left $\mathscr C$-approximation.
	The notions of a \emph{right $\mathscr C$-approximation}
	and a \emph{contravariantly finite} subcategory are
	defined dually.
	We say that $\mathscr C$ is \emph{functorially finite} if it is
	both covariantly and contravariantly finite.
	
	\begin{definition}\cite[Definition 4.13]{Ja} A subcategory $\mathscr{C}$  of an exact category  $(\mathscr{A},\mathscr{E})$ is called  an \emph{$n$-cluster tilting subcategory} if the following conditions are satisfied.	
		\begin{itemize}
			\item [(a)] $\mathscr{C}$ is a functorially finite subcategory of $\mathscr{A}$ such that	
			\begin{align*}
				\mathscr{C}&=\{X\in\mathscr{A}\mid \Ext_{\mathscr{A}}^i(C,X)=0\text{ for any } C\in \mathscr{C} \text{ and } 1\le i<n \}\\
				&=\{X\in\mathscr{A}\mid \Ext_{\mathscr{A}}^i(X,C)=0\text{ for any } C\in \mathscr{C} \text{ and } 1\le i<n \}.
			\end{align*}
			\item [(b)]  $\mathscr{C}$ is both a generating and a cogenerating subcategory of $\mathscr{A}$, that is, for any object $A\in \mathscr{A}$, there exists an $\mathscr{E}$-deflation $C \to A$ with $C \in \mathscr{C}$ and an  $\mathscr{E}$-inflation  $ A \to C'$ with  $C' \in \mathscr{C}$.
		\end{itemize} 
	\end{definition} 	
	If $n=1$, then $\mathscr{C}=\mathscr{A}$ is the unique $1$-cluster tilting subcategory of $\mathscr{A}$. 
		
		Let $(\mathscr{A},\mathscr{E})$ be an exact category. Recall that a morphism $X\to Y$ in
	$\mathscr{A}$ is \emph{proper} if it has a factorization $X\to Z\to Y$
	such that $X\to Z$ is an $\mathscr{E}$-deflation and $Z\to Y$ an $\mathscr{E}$-inflation.  A sequence 
	\[
	\begin{tikzcd}[row sep=small,column sep=small]
		\cdots\ar{rr}&&X^{k-1}\ar{rr}\drar&&X^k\ar{rr}\drar&&X^{k+1}\ar{rr}&&\cdots\\
		&&&Y^{k-1}\urar&&Y^k\urar
	\end{tikzcd}
	\]
	of proper morphisms is an \emph{$\mathscr{E}$-acyclic complex} if for each $k\in\Z$ the sequence
	$Y^{k-1}\to X^k\to Y^k$ is an $\mathscr{E}$-conflation. An \emph{$\mathscr{E}$-acyclic complex} of the form
	$$\delta\colon X^0 \xrightarrow{} X^1  \xrightarrow{} \cdots \xrightarrow{} X^{n} \xrightarrow{} X^{n+1}$$ is called an \emph{$\mathscr{E}$-acyclic complex of length $n+2$}. 
	
	The following result provides a source of examples of $n$-exact categories and Frobenius $n$-exact categories. This is based on Jasso's construction \cite[Theorems~4.14,~5.16]{Ja}, with Kvamme's correction
	\cite[Theorems~2.6, 8.1]{Kv}.
	By taking the closure under homotopy equivalences,
	the construction works without assuming weak idempotent
	completeness. Here,  for a Frobenius exact category $(\mathscr{A},\mathscr{E})$ and $E\in \mathscr{A}$, an \emph{$n$-th cosyzygy} $\mho^n_{\mathscr{E}}(E)$ of $E$ is defined by an $\mathscr{E}$-acyclic complex
	$$E\to I^1\to \dots \to I^n\to \mho^n_{\mathscr{E}}(E)$$
	where for all $k\in\{1,\dots,n\}$ the object $I^k$ is $\mathscr{E}$-injective.
	\begin{theorem}\label{n-exact from n-CT}  
	Let $(\mathscr A,\mathscr E)$ be an exact category, and let
	$\mathscr C$ be an $n$-cluster tilting subcategory of $\mathscr A$.
	Then $\mathscr C$ carries an $n$-exact structure $\mathscr X$,
	where a complex
	\[
	X^0 \rightarrow X^1 \rightarrow \cdots
	\rightarrow X^n \rightarrow X^{n+1}
	\]
	in $\mathscr C$ is an $\mathscr X$-conflation if and only if
	it is homotopy equivalent to an $\mathscr E$-acyclic complex
	in $\mathscr A$.
		If moreover $(\mathscr{A},\mathscr{E})$ is Frobenius and $\mho^n_{\mathscr{E}}(\mathscr{C})\subset\mathscr{C}$, then  $(\mathscr{C},\mathscr{X})$ is a Frobenius $n$-exact category. Its projective-injectives are precisely the $\mathscr E$-projective-injectives.
	\end{theorem}
	\begin{proof}
		The first assertion follows from
		\cite[Theorem~8.1]{Kv}.		
		For the second assertion, we adapt the proof of
		\cite[Theorem~5.16(i)]{Ja}.
		We assume that the exact category $(\mathscr A,\mathscr E)$ is Frobenius and
		$\mho_{\mathscr E}^{n}(\mathscr C)\subset\mathscr C$.
		Let $\mathscr I$ denote the subcategory of
		$\mathscr E$-injectives.
		The definition of $n$-cluster tilting subcategory implies 
				$\mathscr I\subset\mathscr C$.
		
		Let $Z^{\mr}\in\mathscr X$, and choose an
		$\mathscr E$-acyclic complex $X^{\mr}$ homotopy equivalent
		to $Z^{\mr}$. For each $I\in\mathscr I$, applying
		$\mathscr A(-,I)$ yields a homotopy equivalence between		
		$\mathscr A(X^{\mr},I)$ and $\mathscr A(Z^{\mr},I)$.		
		The complex $\mathscr A(X^{\mr},I)$ is exact since $I$ is
		$\mathscr E$-injective.
		Since homotopy equivalent complexes of abelian groups
		have isomorphic cohomology, the complex $\mathscr A(Z^{\mr},I)$ is
		also exact.
		Thus every object of $\mathscr I$ is
		$\mathscr X$-injective. Finally, for $C\in \mathscr{C}$, choose an $\mathscr{E}$-acyclic complex
		\begin{equation} \label{n cosyzygy seq.}
			C\to I^1\to \cdots \to I^n\to \mho^n_{\mathscr{E}}({C}).
		\end{equation}
	Each term belongs to $\mathscr{C}$ since $\mho^n_{\mathscr{E}}(\mathscr{C})\subset\mathscr{C}$. This means that \eqref{n cosyzygy seq.} is an $\mathscr{X}$-conflation, proving that $(\mathscr{C},\mathscr{X})$ has enough	$\mathscr X$-injectives.		
	We claim that  $\mathscr{X}$-injectives are precisely $\mathscr{E}$-injectives. If $C$ is $\mathscr{X}$-injective, then the first morphism in \eqref{n cosyzygy seq.} splits, so $C$ is a direct summand of $I^1$ and thus is $\mathscr{E}$-injective as claimed. 
			
		  Dually, $(\mathscr{C},\mathscr{X})$ has enough $\mathscr X$-projectives (and they are the 	
		  $\mathscr E$-projectives). Since $(\mathscr{A},\mathscr{E})$ is Frobenius, $\mathscr{X}$-projectives and 
		$\mathscr{X}$-injectives coincide. Hence $(\mathscr{C},\mathscr{X})$ is a Frobenius $n$-exact category, whose projective-injective objects are precisely the $\mathscr E$-projective-injective objects.		
	\end{proof}
	
	\section{Ideal quotients of $n$-exact categories}
		\label{sec: Ideal quotients of $n$-exact categories}	
		Let $(\mathscr{C}, \mathscr{X})$ be an $n$-exact category. 
		Let $\mathscr{F}$ be an additive subcategory of $\mathscr{C}$.
		 	For objects 
		$X, Y \in \mathscr{C}$, we write $[\mathscr{F}](X, Y)$ for the subgroup of 
		$\mathscr{C}(X, Y)$ consisting of morphisms factoring through some object 
		in $\mathscr{F}$. The \emph{factor category} $\mathscr{C}/[\mathscr{F}]$ of 
		$\mathscr{C}$ modulo $\mathscr{F}$ is defined as follows: it has the same objects 
		as $\mathscr{C}$;  the morphism set between $X$ and $Y$ is the quotient group 
		$\mathscr{C}(X, Y)/[\mathscr{F}](X, Y)$; the composition is induced from 
		that in $\mathscr{C}$. Note that 
		$\mathscr{C}/[\mathscr{F}]$ inherits the structure of an additive category. 	
			Denote by $$\pi_\mathscr{F}\colon \mathscr{C}\rightarrow\mathscr{C}/[\mathscr{F}]$$ the canonical functor.  For any morphism $f\colon X \to Y $ in $\mathscr{C}$, we denote by $\overline{f}$ the image of $f$ under $\pi_\mathscr{F}$. Denote by $\mathscr{P}$ (resp. $\mathscr{I}$)
			the  subcategory consisting of $\mathscr{X}$-projectives (resp. $\mathscr{X}$-injectives) in $\mathscr{C}$. 
			An additive subcategory  $\mathscr{F}\subset \mathscr{C}$  is called \emph{removable (for $\mathscr{X}$)} if $\mathscr{F}\subset \mathscr{P}\cap \mathscr{I}$ and the canonical functor $\pi_\mathscr{F}$ sends $\mathscr{X}$-conflations in $\mathscr{C}$  to $n$-exact sequences in $\mathscr{C}/[\mathscr{F}]$; compare \cite{CDR}. 
						
			If moreover $(\mathscr{C}, \mathscr{X})$ is a Frobenius $n$-exact category, the case $\mathscr{F}=\mathscr{P}=\mathscr{I}$ is of particular interest, since the associated quotient category $\mathscr{C}/[\mathscr{P}]$, namely the  \emph{stable category} of $\mathscr{C}$, denoted by $\underline{\mathscr{C}}$,  is equipped with a natural $(n+2)$-angulated structure.
			 In this case,  the canonical functor $\pi_\mathscr{P}\colon \mathscr{C}\rightarrow\underline{\mathscr{C}}$  sends $\mathscr{X}$-conflations to standard $(n+2)$-angles. For details, we refer to \cite[Section 5.2]{Ja}.
			
			Assume $\mathscr{F}\subset \mathscr{P}\cap \mathscr{I}$. We are interested in the following question: when the quotient category $\mathscr{C}/[\mathscr{F}]$ is still an  $n$-exact category, and when this structure is Frobenius.  In this section,	we establish the ideal quotient theorem for this question. We also give a sufficient condition for $\mathscr{F}$ to be a removable subcategory of $\mathscr{C}$, and show that
			the ideal quotient construction is compatible with cluster tilting.

			\subsection{Ideal quotients of  $n$-exact categories}

				Let $\mathscr{X}_\mathscr{F}$ be the class of sequences $\overline{{Y}^{\mr}}=(\overline{d^0_Y},\dots,\overline{d^{n}_Y})$ in $\mathscr{C}/[\mathscr{F}]$ admitting a homotopy equivalence
			\begin{equation}\label{new exact structure}
				\begin{tikzcd}
					\overline{Y^{\mr}}\dar{\overline{y^{\mr}}}	& A \dar[equal]{} \rar{\overline{d_Y^0}} & Y^1 \rar{\overline{d_Y^1}}\dar{\overline{y^1}}& \cdots\rar{\overline{d^{n-1}_Y}} & Y^n\rar{\overline{d_Y^n}} \dar{\overline{y^n}} & C \dar[equal]{}\\
					\overline{{Z}^{\mr}}	& A \rar{\overline{d^0_Z}} & Z^1\rar{\overline{d^1_Z}} & \cdots\rar{\overline{d^{n-1}_Z}}  & Z^n \rar{\overline{d^n_Z}} & C
				\end{tikzcd}
			\end{equation}
			 where $\overline{{Z}^{\mr}}=(\overline{d^0_Z},\dots,\overline{d^{n}_Z})$ is the image of some ${{Z}^{\mr}}=(d^0_Z,\dots, d^n_Z)\in\mathscr{X}$, and  $\overline{y^0}=1_{A}$ and  $\overline{y^{n+1}}=1_{C}$. Note that by \cite[Claim 2.8]{HLN}, there exists a homotopy inverse $\overline{z^{\mr}}\colon \overline{{Z}^{\mr}}\to \overline{Y^{\mr}}$ of $\overline{y^{\mr}}$ such that $\overline{z^0}=1_{A}$ and $\overline{z^{n+1}}=1_{C}$.

			\begin{proposition} \label{closed under homotopy equivalences}
			Assume that $\mathscr{F}$ is a removable subcategory of $\mathscr{C}$. Then  any sequence in $\mathscr{X}_\mathscr{F}$ is $n$-exact  in $\mathscr{C}/[\mathscr{F}]$, and the pair $(\mathscr{C}/[\mathscr{F}],\mathscr{X}_\mathscr{F})$ satisfies $(\rm EC')$, $(\rm {E2'}^{})$ and $(\rm {E2'}^{op})$. Moreover,  			
			$\mathscr{X}_\mathscr{F}$ is closed under homotopy equivalences in ${\bf{C}}_{\mathscr{C}/[\mathscr{F}]}^{n+2}$.
		\end{proposition}
		\begin{proof} Since $\mathscr{F}$ is a removable subcategory of $\mathscr{C}$, the image of an $\mathscr{X}$-conflation under $\pi_\mathscr{F}$ is an $n$-exact sequence in $\mathscr{C}/[\mathscr{F}]$. It further follows from \cite[Proposition 2.5]{Ja}  that $n$-exactness  is preserved under homotopy equivalences. Thus
			any sequence in $\mathscr{X}_\mathscr{F}$ is an $n$-exact sequence in $\mathscr{C}/[\mathscr{F}]$.
			Moreover, the definition of $\mathscr X_{\mathscr F}$
			and the transitivity of homotopy equivalences fixing
			both end terms imply $(\mathrm{EC'})$.	

			To show $(\rm {E2'}^{})$ and $(\rm {E2'}^{op})$,
			by Proposition~\ref{redundant}, it suffices to check
			$(\mathrm{E2'})(\rm i)$ and
			$(\mathrm{E2'}^{\mathrm{op}})(\rm  i)$.
			We prove the latter since the former follows dually.
			 Let $\overline{Y^{\mr}}\in \mathscr{X}_\mathscr{F}$ be of the form (\ref{new exact structure}). Let $\theta\colon B\to C$ be a morphism in
			 $\mathscr C/[\mathscr F]$. Since $\pi_{\mathscr F}$ is full, we may choose a lift
			 $f^{n+1}\colon B\to C$ in $\mathscr C$ such that
			 $\overline{f^{n+1}}=\theta$.
			Then we apply $(\rm {E2'}^{op})(i)$ for $(\mathscr{C}, \mathscr{X})$ to get the following commutative diagram
			in $\mathscr{C}$ such that ${X}^{\mr}\in\mathscr{X}$.
			\[
			\begin{tikzcd}[ampersand replacement=\&]
				{X^{\mr}}\dar{{f^{\mr}}}	\&	A \dar[equal] \rar[dotted]{} 
				\& X^1 \rar[dotted]{} \dar[dotted]{} 
				\& \cdots\rar[dotted]{} 
				\& X^n \rar[dotted]{} \dar[dotted] 
				\& B \dar{f^{n+1}} \\
				{Z^{\mr}}	\&	A \rar{d_Z^0} 
				\& Z^1 \rar{d_Z^1} 
				\& \cdots\rar{d_Z^{n-1}}
				\& Z^n  \rar{d_Z^n} 
				\& C
			\end{tikzcd}
			\]
			 Moreover, we get $\overline{X^{\mr}}\in\mathscr{X}_\mathscr{F}$. By \cite[Claim 2.8]{HLN}, we can choose a homotopy inverse $\overline{h^{\mr}}$ of $\overline{y^{\mr}}$   such that
			$\overline{h^0}=1_{A}$ and $\overline{h^{n+1}}=1_{C}$.  Then we obtain the desired morphism $\overline{h^{\mr}}\circ \overline{f^{\mr}}\colon\overline{X^{\mr}}\to \overline{Y^{\mr}}$.  
			Thus $(\rm {E2'}^{op})(i)$ holds for $(\mathscr{C}/[\mathscr{F}], \mathscr{X}_\mathscr{F})$.

			The last assertion follows directly from \cite[Lemma 4.22(4)]{HLN}. 
						\end{proof}	
		
			\begin{remark} \label{ramark:two exact structures}
				In the preceding proposition, for $n=1$, let $\mathscr{X}'_{\mathscr{F}}$ denote the class of composable
				pairs in $\mathscr{C}/[\mathscr{F}]$ that are isomorphic to
				images of $\mathscr{X}$-conflations. In this case,
				$\mathscr{X}'_{\mathscr{F}}=\mathscr{X}_{\mathscr{F}}$. Indeed, 
				since every isomorphism of complexes is a homotopy equivalence,
				$\mathscr{X}'_{\mathscr{F}}\subset\mathscr{X}_{\mathscr{F}}$ follows from Proposition~\ref{closed under homotopy equivalences}. 
				
				For the reverse inclusion, we show that  a homotopy equivalence fixing both end terms between
				 kernel-cokernel pairs is an isomorphism. 
				 	Let
				 	\[
				 	X^{\mr}:\ A\xrightarrow{a}B\xrightarrow{b}C
				 	\quad \text{and} \quad
				 	Y^{\mr}:\ A\xrightarrow{a'}B'\xrightarrow{b'}C
				 	\]
				 	be two kernel-cokernel pairs, and let
				 	$f^{\mr}=(1_A,f^1,1_C)\colon X^{\mr}\to Y^{\mr}$ be a homotopy equivalence.
				 	By \cite[Claim 2.8]{HLN}, we can choose a homotopy inverse
				 	 $g^{\mr}=(1_A,g^1,1_C)\colon Y^{\mr}\to X^{\mr}$ of $f^{\mr}$.				 	
				 Put $t=g^1f^1-1_B$. Then we have
				 $bt=0$ and $ta=0$.
				 Since $a=\ker b$, we have $t=ah$ for some 
				 $h\colon B\to A$. Since $a$ is a monomorphism, $aha=ta=0$ implies $ha=0$.
				 Since $b=\operatorname{coker}a$, we have $h=ub$ for
				 some $u:C\to A$. Hence
				 $t=aub$ and $t^2=au(ba)ub=0.$
				 Thus $g^1f^1=1_B+t$ is invertible, whose inverse is $1_B-t$.
				 Similarly, $f^1g^1$ is invertible.
				 Hence $f^1$ is an isomorphism.
				 It follows that $f^{\mr}$ is an isomorphism of complexes. This proves $\mathscr{X}'_{\mathscr{F}}=\mathscr{X}_{\mathscr{F}}$.
			\end{remark}

		\begin{lemma}\label{a lemma 2}
			Assume that $\mathscr{F}$ is a removable subcategory of $\mathscr{C}$. If $\mathscr{C}$ is weakly idempotent complete or $n\ge 2$, then the following assertions hold for each morphism $\delta\colon X\to Y$ in $\mathscr{C}$.
			\begin{itemize}						
				\item[(a)] $\overline \delta$ is an
				$\mathscr X_{\mathscr F}$-inflation if and only if
				$\begin{bmatrix}
					\delta & a
				\end{bmatrix}^\top\colon X\to Y\oplus F$ is an
				$\mathscr X$-inflation for some morphism $a\colon X\to F$ with $F\in\mathscr F$.
				\item[(b)] $\overline \delta$ is an
				$\mathscr X_{\mathscr F}$-deflation if and only if
				$\begin{bmatrix}
					\delta & b
				\end{bmatrix}\colon X\oplus F\to Y$ is an
				$\mathscr X$-deflation for some morphism $b\colon F\to Y$ with $F\in\mathscr F$.
			\end{itemize}
		\end{lemma}	 		 
		\begin{proof}
			Next we only show (b), since (a) can be proved similarly.
		
		(b) Let ${\delta}:{X}\to {Y}$ be a morphism in $\mathscr{C}$ such that $\overline{\delta}$ is an $\mathscr{X}_\mathscr{F}$-deflation.
		Then there exists a  morphism ${\alpha}\colon X'\to X$ in $\mathscr{C}$ such that 	$\overline{\delta'}=   \overline{\delta}\circ\overline{\alpha}$, where $\delta'\colon X'\to Y$ is an $\mathscr{X}$-deflation. Then
		$\delta'-\delta\alpha$ has a factorization $X'\xrightarrow{a} F \xrightarrow{b} Y$ with $F\in \mathscr{F}$. 
		This shows that $\delta'$ has a factorization
		$$X'\xrightarrow{\begin{bmatrix}
				\alpha\\
				a
		\end{bmatrix}}X\oplus F \xrightarrow{\begin{bmatrix}
				\delta & b
		\end{bmatrix}} Y.$$
		Using Proposition \ref{Strong obscure axiom} for $(\mathscr{C}, \mathscr{X})$, we have that $\theta:=\begin{bmatrix}
			\delta & b
		\end{bmatrix}\colon X\oplus F\to Y$ is an $\mathscr{X}$-deflation.  Conversely, completing $\begin{bmatrix}
			\delta & b
		\end{bmatrix}$ to an  $\mathscr{X}$-conflation and then applying $\pi_\mathscr{F}$ yields an $\mathscr X_{\mathscr F}$-conflation whose last morphism is $\bar\delta$.
				\end{proof}

		For $A,C\in\mathscr C$, let $\nExt^1_{\mathscr X}(C,A)$
		denote the collection of classes $[X^{\mr}]$ of $n$-exact sequences
		\[
		X^{\mr}\colon\quad
		A\rightarrow X^1\rightarrow\cdots
		\rightarrow X^n\rightarrow C
		\]
		in $\mathscr X$, modulo homotopy equivalences fixing both end
		terms. By $(\mathrm{EI})$, every morphism fixing both end terms
		is a homotopy equivalence, so this agrees with Yoneda equivalence.
		We always assume that
		$\nExt^1_{\mathscr X}(C,A)$ is a set for all $A,C\in\mathscr C$. By \cite[Definition~4.28 and Proposition~4.32]{HLN}, 
		the Baer sum then endows $\nExt^1_{\mathscr X}(C,A)$ with an
		abelian group structure.
			
			 The following result is a higher analogue of \cite[Theorem~3.6]{DI}.
		 	\begin{theorem} \label{main thm1} Let $(\mathscr C,\mathscr X)$ be an $n$-exact category and let
		 		$\mathscr F\subset\mathscr P\cap\mathscr I$ be an additive subcategory.
		 	Then the following conditions are equivalent.
		 	\begin{itemize}
		 		\item [(a)] $\mathscr{F}$ is a removable subcategory of $\mathscr{C}$.
		 		\item [(b)]  $(\mathscr{C}/[\mathscr{F}],{\mathscr{X}_\mathscr{F}})$ is an $n$-exact  category.			
		 	\end{itemize}
		 	In this case,  the induced map $$\nExt^1_{\mathscr X}(C,A)\to \nExt^1_{\mathscr{X}_\mathscr{F}}(C,A)$$ is an isomorphism of abelian
		 	groups for any $A,C\in \mathscr{C}$.
		 \end{theorem}
		 
		\begin{proof} The implication (b)$\Rightarrow$(a) is trivial. Let us prove the direction	(a)$\Rightarrow$(b). 	For $n=1$,  this follows from   \cite[Theorem 3.6]{DI} and Remark \ref{ramark:two exact structures}. So we assume $n\ge 2$. 	Together with Proposition \ref{closed under homotopy equivalences},  it remains to check the following axioms:
			
			$( \rm E0)$: This is trivial.
			
			$( \rm E1)$: This is the dual of the following $( \rm E1^{\op})$.
			
			$( \rm E1^{\op})$: Let   $\overline{\delta}\colon {{X}}\to {{Y}}$	and $\overline{\gamma} \colon {{Y}}\to {{Z}}$ be two $\mathscr{X}_\mathscr{F}$-deflations. By Lemma \ref{a lemma 2}(b), we can lift them to 
			$\mathscr X$-deflations 	
			\[
			\begin{bmatrix}
				\delta&
				a
			\end{bmatrix}\colon X\oplus F\to Y 
			\quad \text{and} \quad \begin{bmatrix}
				\gamma&
				b
			\end{bmatrix}
			\colon Y\oplus G\to Z,
			\]
			where $F,G\in\mathscr F$. By \cite[Proposition~4.6]{Ja}, the morphism
			$\begin{bmatrix}\delta&a\end{bmatrix}\oplus 1_G$
			is again an $\mathscr X$-deflation.
			Since $\mathscr X$-deflations are closed under composition, the composition
			\[
			\begin{bmatrix}\gamma\delta&\gamma a&b\end{bmatrix}
			=
			\begin{bmatrix}\gamma&b\end{bmatrix}
			\begin{bmatrix}
				\delta&a&0\\
				0&0&1_G
			\end{bmatrix}
			\colon X\oplus F\oplus G\longrightarrow Z
			\]
			is an $\mathscr X$-deflation.  Lemma~\ref{a lemma 2}(b)  therefore implies
			that $\overline{\gamma}\circ \overline{\delta}$ is an
			$\mathscr X_{\mathscr F}$-deflation.

			$( \rm EI)$: We only show $( \rm EI)(i)$, as $( \rm EI)(ii)$ can be shown similarly. 
			For $\overline{{Y}^{\mr}}\in \mathscr{X}_\mathscr{F}$, let  $\overline{f^{\mr}} \colon \overline{{Y}^{\mr}} \to \overline{{X}^{\mr}}$ be a morphism to an $n$-exact sequence $\overline{{X}^{\mr}}$ in $\mathscr{C}/[\mathscr{F}]$, in which $\overline{f^0}$ and $\overline{f^{n+1}}$ are isomorphisms. Transporting the end terms of $\overline{{X}^{\mr}}$ along these isomorphisms, we may assume that both complexes $\overline{{Y}^{\mr}}$ and $\overline{{X}^{\mr}}$ have end terms $A,C$ and $\overline{f^0}=1_A$ and $\overline{f^{n+1}}=1_C$ from the beginning. Choose a homotopy equivalence $\overline{w^{\mr}}\colon\overline{W^{\mr}}\to \overline{Y^{\mr}}$ fixing the end terms, where $W^{\mr}\in \mathscr{X}$. Thus we have
			 the following commutative diagram with $n$-exact rows.
			\begin{equation*}
				\begin{tikzcd}
					\overline{W^{\mr}}\dar{\overline{w^{\mr}}}	& A \dar[equal]{} \rar{\overline{d_W^0}} & W^1 \rar{\overline{d_W^1}}\dar{\overline{w^1}}& \cdots\rar{\overline{d^{n-1}_W}} & W^n\rar{\overline{d_W^n}} \dar{\overline{w^n}} & C \dar[equal]{}\\
					\overline{Y^{\mr}}\dar{\overline{f^{\mr}}}	& A \dar[equal]{} \rar{\overline{d_Y^0}} & Y^1 \rar{\overline{d_Y^1}}\dar{\overline{f^1}}& \cdots\rar{\overline{d^{n-1}_Y}} & Y^n\rar{\overline{d_Y^n}} \dar{\overline{f^n}} & C \dar[equal]{}\\ 
					\overline{{X}^{\mr}}	& A \rar{\overline{d^0_X}} & X^1\rar{\overline{d^1_X}} & \cdots\rar{\overline{d^{n-1}_X}}  & X^n \rar{\overline{d^n_X}} & C
				\end{tikzcd}
			\end{equation*}	
			
			By \cite[Lemma 4.22(1)]{HLN}, it suffices to show $\overline{{X}^{\mr}}\in \mathscr{X}_\mathscr{F}$, and thus 
			$\overline{f^{\mr}}:\overline{Y^{\mr}} \to \overline{{X}^{\mr}}$ is a homotopy equivalence in $\mathscr{C}/[\mathscr{F}]$.						
			The equality $ \overline{d_W^n}=\overline{d^n_X}\circ \overline{f^n} \circ \overline{w^n} $ implies  that
			$  {d_W^n}={d^n_X} {f^n}  {w^n}+ba$  for some morphisms $a\colon W^n \to F$ and $b\colon F \to C$ with $F\in \mathscr{F}$. Thus we have
			$$
			d_W^n=\begin{bmatrix}d_X^n&b\end{bmatrix}
			\begin{bmatrix}f^nw^n\\a\end{bmatrix}.
			$$	
			By Proposition \ref{Strong obscure axiom} for the $n$-exact category $(\mathscr{C}, \mathscr{X})$, $\theta=\begin{bmatrix}d_X^n&b\end{bmatrix}\colon X^n\oplus F\to C$ is an $\mathscr{X}$-deflation. Complete $\theta$ to ${Z^{\mr}}\in \mathscr{X}$ of the form 
			\[ Z^0 \xrightarrow{d_Z^0} Z^1 \xrightarrow{d_Z^1}  \cdots \xrightarrow{d_Z^{n-2}} Z^{n-1} \xrightarrow{d_Z^{n-1}} X^{n}\oplus F \xrightarrow{ {\theta}}C\]
			Moreover, $\overline{Z^{\mr}}\in \mathscr{X}_\mathscr{F}$. The weak kernel properties of $\overline{Z^{\mr}}$ extend
			 the  commutative square  
			\begin{equation*}
				\begin{tikzcd}
					\overline{{X}^{\mr}} &	A \rar{\overline{d^0_X}} & X^1\rar{\overline{d^1_X}} & \cdots \rar{} & X^{n-1}\rar{\overline{d^{n-1}_X}}  & X^n \rar{\overline{d^n_X}}\dar[equal] & C\dar[equal] \\
					\overline{Z^{\mr}} &		 Z^0  \rar{\overline{d_Z^0}} & Z^1 \rar{\overline{d_Z^1}}& \cdots \rar{} & Z^{n-1}\rar{\overline{d^{n-1}_Z}} & X^n\rar{\overline{d^n_X}}  & C 
				\end{tikzcd}
			\end{equation*}	
			to a morphism from $\overline{{X}^{\mr}}$ to $\overline{Z^{\mr}}$.  This is a
			homotopy equivalence  by \cite[Proposition 2.7]{Ja}. Thus $\overline{Z^{\mr}}\in \mathscr{X}_\mathscr{F}$ implies $\overline{{X}^{\mr}}\in \mathscr{X}_\mathscr{F}$ by Proposition \ref{closed under homotopy equivalences}. This proves (EI)(i).	
			
			Finally, for any $n\geq1$, the induced map
			\[
			\Phi\colon\nExt^1_{\mathscr X}(C,A)\rightarrow
			\nExt^1_{\mathscr X_{\mathscr F}}(C,A)
			\]
			is surjective by the definition of $\mathscr X_{\mathscr F}$.		
			Since the quotient functor $\pi_{\mathscr{F}}$ preserves direct sums and morphisms
			of conflations,  $\Phi$ is a group homomorphism by \cite[Lemma~4.30]{HLN}.
			Let $[X^{\mr}]\in\ker\Phi$.
			By \cite[Remark~4.33]{HLN} and the dual of
			\cite[Proposition~2.6]{Ja}, 			
			$\overline{d_X^n}$ is a split epimorphism.
			Consequently, there are morphisms $s\colon C\to X^n$, $a\colon C\to F$ and $b\colon F\to C$ in $\mathscr{C}$ with $F\in\mathscr F$ such that
			\[
			1_C-d_X^ns=ba.
			\]
			Since $F$ is $\mathscr X$-projective, there is a morphism
			$t\colon F\to X^n$ such that $b=d_X^nt$.
			Thus $d_X^n(s+ta)=1_C$.
			The same splitting criterion gives $[X^{\mr}]=0$.
			Hence $\Phi$ is injective, completing the proof.	
				\end{proof}	 
			
				\begin{remark} Assume that $\mathscr{F}$ is a removable subcategory of $\mathscr{C}$. 
					The $n$-exact structure $\mathscr{X}_\mathscr{F}$ coincides with each of the
					following classes:
					\begin{itemize}
						\item [(a)] The class of $n$-exact sequences in $\mathscr{C}/[\mathscr{F}]$ that are homotopy equivalent to the images of $\mathscr{X}$-conflations. 
						\item [(b)] The class of $n$-exact sequences in $\mathscr{C}/[\mathscr{F}]$ that  are weakly isomorphic
								 to the images of $\mathscr{X}$-conflations. 
						\end{itemize}
						Indeed, (a) follows from
						Proposition~\ref{closed under homotopy equivalences} and the
						definition of $\mathscr X_{\mathscr F}$.
						For (b), Theorem~\ref{main thm1} shows that
						$\mathscr X_{\mathscr F}$ is an $n$-exact structure and
						hence is closed under weak isomorphisms.
						The reverse inclusion follows since the homotopy equivalences
						defining $\mathscr X_{\mathscr F}$ fix both end terms
						and hence are weak isomorphisms.
						
						When $n=1$, $\mathscr X_{\mathscr F}$ is also the class of
						composable pairs in $\mathscr C/[\mathscr F]$ that are
						isomorphic to images of $\mathscr X$-conflations;
						see Remark~\ref{ramark:two exact structures}.					
				\end{remark}
			
					\begin{remark}  \label{alternative quotient proof}
						Assume that $\mathscr F$ is removable.	The $n$-exactness of
							$(\mathscr C/[\mathscr F],\mathscr X_{\mathscr F})$
							can also be deduced from the theory of $n$-exangulated categories.
							Indeed, by \cite[Proposition~4.34]{HLN},
							$(\mathscr C,\mathscr X)$ carries a canonical
							$n$-exangulated structure $(\mathscr C,\mathbb E,\mathfrak s)$, 
							where $\mathbb E\colon \mathscr C^{\rm op}\times\mathscr C\to Ab$ is a  
							biadditive functor, and $\mathfrak s$ is an exact realization of $\mathbb E$, satisfying certain conditions. We refer to \cite[Definition 2.32]{HLN} for more details. 
							Since $\mathscr F$ is removable, every $\mathfrak s$-conflation
							has an $n$-exact image in $\mathscr C/[\mathscr F]$ and hence
							induces a weak kernel-cokernel sequence.
							Thus \cite[Theorem~3.1]{HZZ} yields an $n$-exangulated structure
							$(\mathscr C/[\mathscr F],\overline{\mathbb E},
							\overline{\mathfrak s})$. 
							Consequently, its $\overline{\mathfrak s}$-conflations are
							precisely the sequences in $\mathscr X_{\mathscr F}$.
							These are $n$-exact, since $n$-exactness is preserved under
							homotopy equivalences. In particular, all
							$\overline{\mathfrak s}$-inflations are monomorphisms and all
							$\overline{\mathfrak s}$-deflations are epimorphisms.
							By \cite[Proposition~5.8]{Kla},
							$(\mathscr C/[\mathscr F],\mathscr X_{\mathscr F})$
							is therefore an $n$-exact category. 
							
							Moreover, this construction also yields the extension-group
							isomorphism. Indeed, for any $A,C\in\mathscr C$, we have
							$
							\overline{\mathbb E}(C,A)
							=\mathbb E(C,A)
							=\nExt^1_{\mathscr X}(C,A).
							$
							By \cite[Lemma~5.9]{Kla}, the quotient realization induces
							a natural isomorphism of abelian groups
							$
							\overline{\mathbb E}(C,A)
							\xrightarrow{\sim}
							\nExt^1_{\mathscr X_{\mathscr F}}(C,A).
							$
							Under the above identification, this sends
							$[X^{\mr}]$ to $[\pi_{\mathscr F}(X^{\mr})]$,
							so it is precisely the map induced by $\pi_{\mathscr F}$.				
						\end{remark}

				The following theorem gives the Frobenius version of
				Theorem~\ref{main thm1}.
	
				\begin{theorem} \label{main thm2} Assume that $(\mathscr{C}, \mathscr{X})$ is a Frobenius $n$-exact category.  Let $\mathscr F\subset\mathscr P$ be an additive subcategory.
				Then the following conditions are equivalent.
				\begin{itemize}
					\item [(a)] $\mathscr{F}$ is a removable subcategory of $\mathscr{C}$.
					\item [(b)]  $(\mathscr{C}/[\mathscr{F}],{\mathscr{X}_\mathscr{F}})$ is a Frobenius $n$-exact  category, whose subcategory of $\mathscr{X}_\mathscr{F}$-projectives is precisely $\mathscr P/[\mathscr F]$.			
				\end{itemize}
			\end{theorem}
			\begin{proof} (b)$\Rightarrow$(a) is trivial. Let us prove the direction	(a)$\Rightarrow$(b). 	For $n=1$,  this follows from   \cite[Lemma 2.7]{CDR}. So we assume $n\ge 2$. 
				By Theorem \ref{main thm1}, the pair $(\mathscr{C}/[\mathscr{F}], \mathscr{X}_\mathscr{F})$ is an $n$-exact category. 												
				By Lemma \ref{a lemma 2}, we know that any $\mathscr{X}_\mathscr{F}$-inflation (resp. $\mathscr{X}_\mathscr{F}$-deflation) is  isomorphic to the image of an $\mathscr{X}$-inflation (resp. $\mathscr{X}$-deflation). Then it follows that objects in $\mathscr{P}/[\mathscr{F}]$ are both $\mathscr{X}_\mathscr{F}$-projectives and $\mathscr{X}_\mathscr{F}$-injectives in the $n$-exact category $(\mathscr{C}/[\mathscr{F}], \mathscr{X}_\mathscr{F})$. 
				
				Suppose that $Q$ is $\mathscr{X}_\mathscr{F}$-projective in $\mathscr{C}/[\mathscr{F}]$. Since $\mathscr{C}$ has enough $\mathscr{X}$-projectives, we take an $\mathscr{X}$-conflation
				$$A\xrightarrow{d^0_P} P^1 \to \dots \to P^n \xrightarrow{d^n_P} Q$$
				with $P^1,\dots,P^n\in \mathscr{P}$. Its image is an $\mathscr{X}_\mathscr{F}$-conflation. Since $Q$ is $\mathscr{X}_\mathscr{F}$-projective in $\mathscr{C}/[\mathscr{F}]$, there is a morphism ${q}\colon Q\to P^n$ such that $\overline{d^n_P}\circ \overline{q}=1_Q$. Then $d^n_Pq-1_Q$ has a factorization
				$$Q\xrightarrow{a}F \xrightarrow{b} Q.$$ 
				Since $F$ is $\mathscr{X}$-projective in $\mathscr{C}$ and $d^n_P$ is an 
				$\mathscr{X}$-deflation, there exists $c\colon F\to P^n$ such that $b=d^n_Pc$. Consequently, we have $1_Q=d^n_P(q-ca)$. This implies that $d^n_P$ is a split epimorphism, and thus $Q\in \mathscr{P}$. The injective case is dual. Hence the objects of $\mathscr{P}/[\mathscr{F}]$ are precisely the $\mathscr{X}_\mathscr{F}$-projectives in the $n$-exact category $(\mathscr{C}/[\mathscr{F}], \mathscr{X}_\mathscr{F})$, which coincide with the $\mathscr{X}_\mathscr{F}$-injectives.
				
				For any $C\in \mathscr{C}$, take two $\mathscr{X}$-conflations
				$$A\xrightarrow{} P^1 \to \dots \to P^n \xrightarrow{} C\quad \text{and}\quad C\xrightarrow{} I^1 \to \dots \to I^n \xrightarrow{} B$$
				with all $P^i,I^i\in \mathscr{P}$. Their images are $\mathscr{X}_\mathscr{F}$-conflations. From these, we know that $(\mathscr{C}/[\mathscr{F}], \mathscr{X}_\mathscr{F})$ has enough $\mathscr{X}_\mathscr{F}$-projectives and enough $\mathscr{X}_\mathscr{F}$-injectives. Therefore the $n$-exact category $(\mathscr{C}/[\mathscr{F}], \mathscr{X}_\mathscr{F})$ is Frobenius.  		
			\end{proof}

				\begin{remark} \label{canonical equivalence}
				Assume that $(\mathscr C,\mathscr X)$ is Frobenius and
				$\mathscr F$ is removable.
				The quotient functor $\pi_{\mathscr F}$ induces a canonical
				equivalence of $(n+2)$-angulated categories
				$$ \underline{\pi_\mathscr{F}}\colon \quad \underline{\mathscr{C}}=
				\mathscr C/[\mathscr P]
				\xlongrightarrow{\sim}
				(\mathscr C/[\mathscr F])/(\mathscr P/[\mathscr F])=\underline{\mathscr C/[\mathscr F]},
				$$	
				which follows from the construction in \cite[Section~5.2]{Ja}. Indeed, since $\mathscr F\subset\mathscr P$, both stable categories have morphism groups
				$\mathscr C(A,B)/[\mathscr P](A,B)$.
				Hence $\underline{\pi_\mathscr{F}}$ is an additive equivalence.				
				For each $A\in\mathscr C$, choose an $\mathscr X$-conflation
				\[
				I(A)\colon\quad
				A\rightarrow I^1\rightarrow\cdots
				\rightarrow I^n\rightarrow SA,
				\]
				with  $I^i\in\mathscr P$. Use $I(A)$ and $\pi_{\mathscr F}(I(A))$ to construct the
				suspension functors $\Sigma\colon \underline{\mathscr{C}} \to  \underline{\mathscr{C}}$ and $\Sigma_{\mathscr F}\colon \underline{\mathscr C/[\mathscr F]} \to \underline{\mathscr C/[\mathscr F]}$,
				respectively, as in \cite[Section~5.2]{Ja}.
				These compatible choices give a natural isomorphism
				$
				\underline{\pi_{\mathscr F}}\Sigma
				\simeq\Sigma_{\mathscr F}\underline{\pi_{\mathscr F}}.
				$
				Since $\pi_{\mathscr F}$ preserves conflations and
				projective-injective objects, it sends the diagrams defining
				standard $(n+2)$-angles to such diagrams in the quotient.
				Thus
				$\underline{\pi_{\mathscr F}}$ preserves standard
				$(n+2)$-angles.
				The transported angulation is therefore contained in the
				quotient angulation, and the two coincide by
				\cite[Proposition~5.2]{Ja}.
			\end{remark}
			
	\subsection{A sufficient condition} 
	In  this subsection, 
	we give a sufficient condition for $\mathscr{F}$ to be a removable subcategory of $\mathscr{C}$.
	\begin{proposition} \label{A sufficient condition}
		Assume that $\mathscr F\subset\mathscr P\cap\mathscr I$ satisfies the following conditions.
		\begin{itemize}
			\item[$(a)$] For any object $Z\in\mathscr{C}$, there exists a sequence
			$$Z\stackrel{f^0}\longrightarrow F\stackrel{f^1}\longrightarrow P^1\stackrel{f^2}\longrightarrow \dots \stackrel{}\longrightarrow P^{n-1}  \stackrel{f^{n}}\longrightarrow P^{n},$$ where  $f^0$ is a left $\mathscr{F}$-approximation of $Z$, and for each $1\le i \le n$,    $f^i$ is a weak cokernel of $f^{i-1}$ with $P^{i}\in \mathscr{P}$; 
			\item[$(b)$] For any object $Z\in\mathscr{C}$, there exists a sequence
			$$I_n\stackrel{f_{n}}\longrightarrow I_{n-1}\stackrel{f_{n-1}}\longrightarrow  \dots  \longrightarrow I_1 \stackrel{f_{1}}\longrightarrow F \stackrel{f_0}\longrightarrow Z,$$ where  $f_{0}$ is a right $\mathscr{F}$-approximation of $Z$, and for each $1\le i \le n$, $f_{i}$ is a weak kernel of $f_{i-1}$ with $I_{i}\in \mathscr{I}$.
		\end{itemize}
		Then $\mathscr{F}$ is a removable subcategory of $\mathscr{C}$.		
	\end{proposition}
	\begin{proof}  
		It suffices to show that for any  $\mathscr{X}$-conflation 
		$$ X^{\mr}\colon\quad X^0 \stackrel{d^0_X}\rightarrow X^1  \stackrel{d^1_X}\longrightarrow \cdots \longrightarrow X^{n} \stackrel{d^n_X}\rightarrow X^{n+1}$$
		in $ \mathscr{X}$, its image under $\pi_\mathscr{F}$ is an $n$-exact sequence in $\mathscr{C}/[\mathscr{F}]$, that is, the sequence $(\overline{d^0_X},\dots,\overline{d^{n-1}_X})$ is an $n$-kernel of $\overline{d^{n}_X}$,  and  $(\overline{d^1_X},\dots,\overline{d^n_X})$ is an $n$-cokernel of $\overline{d^0_X}$. We only show the former one by using (a) since the  latter assertion  can be shown  dually by using (b). The proof is divided into the following two steps
		
		\emph{Step $1$}: To prove that $\overline{d^0_X}$ is a monomorphism, it amounts to showing that any morphism $g^0\colon Z\to X^0$ in $\mathscr{C}$ satisfying $d^0_Xg^0\in [\mathscr{F}](Z,X^1)$ necessarily belongs to $ [\mathscr{F}](Z,X^0)$. Consider the sequence in (a).  Since
		$d^0_Xg^0 \colon Z\rightarrow X^1$ factors through an object in
		$\mathscr{F}$ and $f^0\colon Z\rightarrow F$ is a left
		$\mathscr{F}$-approximation, there exists a morphism $g^1\colon
		F\rightarrow X^1$ such that $d^0_Xg^0=g^1f^0$. Since $f^i$ is a weak cokernel of $f^{i-1}$ for $1\le i\le n$, 
		$(g^0,g^1)$ can be extended to a morphism of complexes:
		\begin{equation*}
			\begin{tikzcd}
				Z \dar{g^0} \rar{f^0} & F \rar{f^1}\dar[dotted]{g^1}& P^1 \rar{f^2}\dar[dotted]{g^{2}}&\cdots\rar &  P^{n-1}\rar{f^n} \dar[dotted]{g^{n}} & P^{n} \dar[dotted]{g^{n+1}}\\
				X^0 \rar[tail]{d^0_X} & X^1\rar{d^1_X} & X^2\rar{d^2_X} & \cdots\rar  & X^n \rar[two heads]{d^n_X} & X^{n+1}
			\end{tikzcd}
		\end{equation*}
	  Let  $f^{n+1}\colon P^{n}\to 0$ and $h^{n+1}\colon 0\to X^{n+1}$ be the zero morphisms.   Write $P^0=F$. 
		Next we construct morphisms $h^{k}\colon P^{k}\to X^{k}$  $(0\le k \le n)$  satisfying 
		\begin{align} \label{g^{k}}
			g^{k+1}=h^{k+1}f^{k+1}+d_X^{k}h^{k}
		\end{align}
		by descending induction on $k$. 
		Since $P^{n}$ is projective and $d_X^n\colon X^n\to X^{n+1}$ is an $\mathscr{X}$-deflation, we may lift $g^{n+1}$ to a morphism $h^{n}\colon P^{n}\to X^n$ such that $g^{n+1}=d_X^nh^{n}$.
		Assume that for some $1\le k \le n$, we have constructed morphisms $h^{k},h^{k+1},\dots,h^{n}$ satisfying (\ref{g^{k}}). 
		Composing $f^{k}$ to  both sides of  (\ref{g^{k}}), we have $d_X^{k}h^{k}f^{k}=g^{k+1}f^{k}=d_X^{k}g^{k}$, that is, $d_X^{k}(g^{k}-h^{k}f^{k})=0$. Since $d_X^{k-1}$ is a weak kernel of $d_X^{k}$, there exists a morphism $$h^{k-1}\colon P^{k-1}\to X^{k-1}$$ such that $d_X^{k-1}h^{k-1}=g^{k}-h^{k}f^{k}$. This completes the descending induction. Hence we have $d_X^{0}h^{0}=g^1-h^{1}f^1$. Composing $f^{0}$ to  both sides, we have $$d_X^0h^0f^0=g^1f^0=d_X^0g^0.$$ Note that $d_X^0$ is a monomorphism and $F\in \mathscr{F}$. Then we have $g^0=h^0f^0$ and $g^0$ belongs to  $ [\mathscr{F}](Z,X^0)$. Hence $\overline{d^0_X}$ is a monomorphism.
		
		\emph{Step $2$}: We show that  $\overline{d_X^{\ell-1}}$ is a weak kernel of $\overline{d_X^\ell}$ for $1\le \ell \le n$. Put $P^0=F$. 
		Let $v^0\colon Z\to X^{\ell}$ be a morphism in $\mathscr{C}$ such that $d_X^{\ell}v^0\in  [\mathscr{F}](Z,X^{\ell+1})$. We  need to show that $\overline{v^0}$ factors through $\overline{d_X^{\ell-1}}$. Since $f^0\colon Z\rightarrow F$ is a left $\mathscr{F}$-approximation, there exists a morphism $v^1\colon F\to X^{\ell+1}$ such that $v^1f^0=d_X^\ell v^0$. Using  the factorization property of weak cokernels, we obtain the following commutative diagram.
		\begin{equation*}
			\begin{tikzcd}
				Z\rar{f^{0}} \dar{v^{0}}& F\rar{f^1} \dar[dotted]{v^{1}} & P^1 \dar[dotted]{v^{2}} \rar{f^2}& \cdots \rar &  P^{n-\ell-1}\rar{f^{n-\ell}} \dar[dotted]{v^{n-\ell}} & P^{n-\ell} \dar[dotted]{v^{n+1-\ell}}\\
				X^{\ell} \rar{d_X^{\ell}} & X^{\ell+1} \rar{d_X^{\ell+1}} & X^{\ell+2} \rar{d_X^{\ell+2}}& \cdots\rar & X^n \rar[two heads]{d_X^n} & X^{n+1}
			\end{tikzcd}
		\end{equation*}
		Since $P^{n-\ell}$ is projective and $d_X^n\colon X^n\to X^{n+1}$ is an $\mathscr{X}$-deflation, we may lift $v^{n+1-\ell}$ to a morphism $u^{n-\ell}:P^{n-\ell}\to X^{n}$ such that $v^{n+1-\ell}=d_X^nu^{n-\ell}$. Then we have $$d_X^n(v^{n-\ell}-u^{n-\ell}f^{n-\ell})=d_X^nv^{n-\ell}-v^{n+1-\ell}f^{n-\ell}=0.$$ Thus there exists a morphism $u^{n-\ell-1}\colon P^{n-\ell-1}\to X^{n-1}$ such that $$v^{n-\ell}=u^{n-\ell}f^{n-\ell}+d_X^{n-1}u^{n-\ell-1}.$$ By a descending induction as before, one can show that there are morphisms $u^{-1}\colon Z\to X^{\ell-1}$ and $u^{0}\colon F\to X^{\ell}$ such that
		$v^0-u^0f^0=d_X^{\ell-1}u^{-1}$.  By applying $\pi_\mathscr{F}$, we get $\overline{v^0}=\overline{d_X^{\ell-1}}\circ \overline{u^{-1}}$ since $F\in \mathscr{F}$. This completes the proof.
	\end{proof}
	
	For $n=1$, Proposition~\ref{A sufficient condition}
	together with Theorem~\ref{main thm2} recovers
	Chen's result \cite[Theorem~3.1]{C}.
	
	\subsection{Compatibility with cluster tilting} In this subsection, 
	we show that the ideal quotient construction 
	is compatible with cluster tilting.

	\begin{theorem} \label{Compatibility with cluster tilting}
		Let $(\mathscr A,\mathscr E)$ be an exact category,
		and let $\mathscr C$ be an $n$-cluster tilting subcategory
		of $\mathscr A$.
		Let $(\mathscr C,\mathscr X)$ be the induced 
		$n$-exact category of Theorem~\ref{n-exact from n-CT}.
		If $\mathscr F$ is a removable subcategory of $\mathscr A$,
		then the following hold.
		\begin{itemize}
			\item[(a)]
			$\mathscr F$ is a removable subcategory of
			$(\mathscr C,\mathscr X)$.
			\item[(b)]
			$\mathscr C/[\mathscr F]$ is an $n$-cluster tilting
			subcategory of $\mathscr A/[\mathscr F]$.
		\end{itemize}
		Moreover, the $n$-exact structure induced on
		$\mathscr C/[\mathscr F]$ by
		Theorem~\ref{n-exact from n-CT}
		coincides with $\mathscr X_{\mathscr F}$ of Theorem \ref{main thm1}.
		In other words, the diagram of constructions
		\[
		\begin{tikzcd}[column sep=2.4em,row sep=2.4em]
			(\mathscr A,\mathscr E)
			\arrow[r,"\pi_{\mathscr F}"]
			\arrow[d,Rightarrow]
			&
			(\mathscr A/[\mathscr F],\mathscr E_{\mathscr F})
			\arrow[d,Rightarrow]
			\\
			(\mathscr C,\mathscr X)
			\arrow[r,"{\pi_{\mathscr F}|_{\mathscr C}}"]
			&
			(\mathscr C/[\mathscr F],\mathscr X_{\mathscr F})
		\end{tikzcd}
		\]
		 commutes, where the vertical arrows represent the construction in
		Theorem~\ref{n-exact from n-CT}.
		
		If, in addition,  $(\mathscr A,\mathscr E)$ is Frobenius and $\mho^n_{\mathscr{E}}(\mathscr{C})\subset\mathscr{C}$, then 
		$\mho_{\mathscr E_{\mathscr F}}^{n}
		(\mathscr C/[\mathscr F])
		\subset\mathscr C/[\mathscr F].
		$ In this case, the two exact categories in the top row are
		Frobenius, and the two $n$-exact categories
		in the bottom row are Frobenius $n$-exact categories.
		
	\end{theorem}
	\begin{proof} 
		Applying Theorem~\ref{main thm1} with $n=1$ to the  exact category $(\mathscr A,\mathscr E)$, we obtain
		an exact category $(\mathscr A/[\mathscr F],\mathscr E_{\mathscr F})$.  Since every object of $\mathscr{F}$ is both $\mathscr{E}$-projective and  $\mathscr{E}$-injective, the defining orthogonality conditions for $\mathscr{C}$ give $\mathscr{F}\subset \mathscr{C}$. The proof is divided into the following three steps.

		\emph{Step $1$}: First we show (b). The quotient functor $\pi_{\mathscr{F}}$ induces natural morphisms
		\begin{equation}\label{phi_i}
			\phi^i=\phi^i_{U,V}\colon 
			\Ext_{\mathscr E}^i(U,V)
			\to
			\Ext_{\mathscr E_{\mathscr F}}^i( U,V)			
		\end{equation}
		for any $U,V\in\mathscr A$ and $i\ge 1$. Here, $\Ext_{\mathscr E}^i(U,V)$ denotes the $i$-th Yoneda extension group of two objects $U,V$ in $\mathscr A$ with respect to
		$\mathscr E$; see for example \cite[Section 2.1]{Kv}. 
		We claim that these morphisms are isomorphisms
		for any $U,V\in\mathscr A$ by induction on $i$. 
		For $i=1$, this follows from the last assertion of
		Theorem~\ref{main thm1}, applied with $n=1$. 		
		Assume $i\ge 2$. Any $i$-fold extension (= $\mathscr{E}$-acyclic complex of length $i+2$) is the Yoneda composition of an $\mathscr E$-conflation and $(i-1)$-fold extension. Since $\mathscr F$ is removable subcategory of $\mathscr A$, $\pi_{\mathscr{F}}$ sends an $\mathscr E$-conflation to an  $\mathscr E_\mathscr F$-conflation. By the induction hypothesis, surjectivity of $\phi^1$ and $\phi^{i-1}$ gives surjectivity of $\phi^{i}$.
	
		 For injectivity, let $i\geq 2$ and take
		$\xi\in\Ext^i_{\mathscr E}(U,V)$ with $\phi^i(\xi)=0$.
		Write
		$\xi=\partial_e(\eta),$ where
		$e: 0\rightarrow V\rightarrow B
		\xrightarrow{q}K\rightarrow0\in \mathscr{E},
		$
		$\eta\in\Ext^{i-1}_{\mathscr E}(U,K)$ and
		$\partial_e$ is the connecting homomorphism given by
		Yoneda composition with $e$.	
		Naturality and compatibility with Yoneda composition give
		the following commutative diagram with exact rows.
		\begin{equation*}
		\begin{tikzcd}
			\Ext^{i-1}_{\mathscr E}(U, B)
			\rar{{q}_*}\dar{\phi^{i-1}_{U,B}}&
			\Ext^{i-1}_{\mathscr E}(U,K)
			\rar{\partial_{{e}}}\dar{\phi^{i-1}_{U,K}}&
			\Ext^i_{\mathscr E}(U,V)\dar{\phi^{i}}\\
			\Ext^{i-1}_{\mathscr E_{\mathscr F}}(U, B)
			\rar{\overline{q}_*}&
			\Ext^{i-1}_{\mathscr E_{\mathscr F}}(U,K)
			\rar{\partial_{\overline{e}}}&
			\Ext^i_{\mathscr E_{\mathscr F}}(U,V)
		\end{tikzcd}
	\end{equation*}
		Since
		$\partial_{\overline e}\phi^{i-1}(\eta)=\phi^i(\xi)=0$,
		we have
		$\phi^{i-1}(\eta)=\overline q_*(\zeta)$ for some
		$\zeta\in\Ext^{i-1}_{\mathscr E_{\mathscr F}}(U,B)$.
		By induction, the first two vertical maps are isomorphisms.
		Thus $\zeta=\phi^{i-1}(\eta_0)$ for some $\eta_0\in\Ext^{i-1}_{\mathscr E}(U,B)$.
	 Commutativity of the left square then gives
		$$
		\phi^{i-1}(q_*\eta_0)
		=\overline q_*(\zeta)
		=\phi^{i-1}(\eta),
		$$
		so $\eta=q_*\eta_0$. Thus we have
		$\xi=\partial_e(q_*\eta_0)=0$.
		This completes the induction.

		Since $\pi_{\mathscr F}$ is full, left and right $\mathscr C$-approximations
		descend to left and right $(\mathscr C/[\mathscr F])$-approximations, respectively. Moreover, the exactness of $\pi_{\mathscr F}$ ensures that the generating and cogenerating conditions also descend. Finally, \eqref{phi_i} transfers the defining
		orthogonality conditions for $\mathscr C$ to
		$\mathscr C/[\mathscr F]$.
		Hence $\mathscr C/[\mathscr F]$ is an $n$-cluster tilting
		subcategory of $\mathscr A/[\mathscr F]$. Thus (b) holds.

		\emph{Step $2$}: We  prove (a) and identify the induced structure. 	Every object of $\mathscr F$ is both $\mathscr X$-projective
		and $\mathscr X$-injective. Indeed, for $F\in\mathscr F$,
		 applying
		$\mathscr A(F,-)$ and $\mathscr A(-,F)$ to an
		$\mathscr E$-acyclic complex yields exact sequences, and exactness of these sequences is preserved
		under homotopy equivalences.
		
		Let $\mathscr Y$ be the class of complexes in
		${\bf C}^{n+2}_{\mathscr C/[\mathscr F]}$ that are homotopy
		equivalent to an $\mathscr E_{\mathscr F}$-acyclic complex
		with all terms in $\mathscr C/[\mathscr F]$. By Theorem~\ref{n-exact from n-CT} and its proof, we obtain an
		$n$-exact category $(\mathscr C/[\mathscr F],\mathscr Y)$. Let $X^{\mr}\in\mathscr X$, and choose an
		$\mathscr E$-acyclic complex $Z^{\mr}$ with terms in $\mathscr C$ and a homotopy equivalence
		between $X^{\mr}$ and $Z^{\mr}$. Since $\pi_{\mathscr F}$ is exact, we know that $\overline{Z^{\mr}}$ is $\mathscr E_{\mathscr F}$-acyclic, and thus $\overline{Z^{\mr}}\in \mathscr Y$. Since $\pi_{\mathscr F}$ preserves homotopy equivalences, 
		$\overline{X^{\mr}}$ is homotopy equivalent to $\overline{Z^{\mr}}$, and thus $\overline{X^{\mr}}\in \mathscr Y$ 
		is $n$-exact in $\mathscr C/[\mathscr F]$. This proves (a). 
		
		Now	we show that $\mathscr Y=\mathscr X_{\mathscr F}$.
		We have shown that $\overline{X^{\mr}}\in\mathscr Y$ for any
		$X^{\mr}\in\mathscr X$.
		Since $\mathscr Y$ is closed under homotopy equivalences,
		the definition of $\mathscr X_{\mathscr F}$ gives
		$\mathscr X_{\mathscr F}\subset\mathscr Y$.	For the reverse inclusion, both structures are closed under homotopy equivalences, so it suffices to consider an $\mathscr E_{\mathscr F}$-acyclic complex
				\[
				\overline{Y^{\mr}}\colon\quad
				A\xrightarrow{}Y^1\rightarrow\cdots
				\rightarrow Y^n\rightarrow C
				\]
				with all terms in $\mathscr C/[\mathscr F]$. Let $\eta\in\Ext^n_{\mathscr E_{\mathscr F}}(C,A)$ be its
				Yoneda class. By Step~$1$, choose
				$\xi\in\Ext^n_{\mathscr E}(C,A)$ with $\phi^n(\xi)=\eta$.
				By \cite[Proposition~3.17]{Kv},  $\mathscr C$
				is closed under $n$-extensions in $\mathscr A$.
				Thus $\xi$ has an $\mathscr E$-acyclic representative
				\[
				U^{\mr}\colon\quad A\rightarrow U^1\rightarrow\cdots
				\rightarrow U^n\rightarrow C
				\]
				with all terms in $\mathscr C$.
				Then $U^{\mr}\in\mathscr X$, and $\overline{U^{\mr}}$ and
				$\overline{Y^{\mr}}$ represent the same Yoneda class.
				Applying \cite[Lemma~3.15]{Kv} to the $n$-cluster tilting
				subcategory 
				$\mathscr C/[\mathscr F]$ gives a morphism
				$\overline{Y^{\mr}}\rightarrow\overline{U^{\mr}}$ which is the identity
				on both end terms. By $(\mathrm{EI})$ for
				$(\mathscr C/[\mathscr F],\mathscr Y)$, this is a homotopy
				equivalence. Then $\overline{Y^{\mr}}\in\mathscr X_{\mathscr F}$,
				and so $\mathscr Y=\mathscr X_{\mathscr F}$.				
		
		\emph{Step $3$}: Assume in addition that $(\mathscr A,\mathscr E)$ is Frobenius and $\mho^n_{\mathscr{E}}(\mathscr{C})\subset\mathscr{C}$. Let $\mathscr{P}$ be the subcategory of $\mathscr E$-projective-injectives.
		By Theorem~\ref{main thm2} with $n=1$,
		$(\mathscr A/[\mathscr F],\mathscr E_{\mathscr F})$ is a Frobenius exact category with projective-injectives
		$\mathscr P/[\mathscr F]$.
		For $C\in\mathscr C$, choose an
		$\mathscr E$-acyclic complex
		\[
		C\longrightarrow I^1\longrightarrow\cdots
		\longrightarrow I^n\longrightarrow\mho_{\mathscr E}^n(C).
		\] such that for each $k\in\{1,\dots,n\}$, $I^k\in \mathscr{P}$. 
		Applying $\pi_{\mathscr F}$ gives an
		$\mathscr E_{\mathscr F}$-acyclic complex with 
		middle terms belonging to $\mathscr P/[\mathscr F]$. Hence we may choose
		\[
		\mho_{\mathscr E_{\mathscr F}}^{n}(C)
		=\pi_{\mathscr F}(\mho_{\mathscr E}^{n}(C))\in\mathscr C/[\mathscr F].
		\]
		Theorem~\ref{n-exact from n-CT} gives the two
		Frobenius $n$-exact categories in the bottom row.
		\end{proof}

	\section{Geigle--Lenzing projective planes}
	\label{sec: Geigle--Lenzing projective planes}	
		In this section, we apply the ideal quotient construction to
		$2$-cluster tilting subcategories of $\ACM\X$, where $\X$
		is a Geigle--Lenzing projective plane.
		After recalling the basic notions and facts, we introduce
		Properties (L) and (R) and prove that every ACM bundle
		on $\X$ satisfies them.
		These properties yield quotient $2$-exact categories
		for both the natural and relative $2$-exact structures.
	
	\subsection{Definitions and basic properties} In this subsection, we recall some notions and facts about Geigle--Lenzing projective planes and arithmetically Cohen--Macaulay  bundles on them \cite{HIMO,CRW}. 
		Let $\mathbf{k}$ be an algebraically closed field. Denote by $D$ the $\mathbf{k}$-dual, that is  $D(-):=\Hom_{\mathbf{k}}(-,\mathbf{k})$.  Fix a quadruple $(p_1,\dots, p_4)$ with integers $p_i\ge 2$, called \emph{weights}. We consider the $\mathbf{k}$-algebra
		$$R=\mathbf{k}[X_1, \dots, X_4]/(\sum_{i=1}^{4}X_i^{p_i}).$$ Let $\L$ be the abelian group on generators $\x_1,\ldots,\x_4,\c$ subject to the relations $$p_1 \x_1= \dots=p_4 \x_4=\c.$$ The element $\c$ is called the \emph{canonical element} of $\L$. Each element $\x$ in $\L$ can be uniquely written in a \emph{normal form} as 
		\begin{align}\label{equ:nor}
			\x=\sum_{i=1}^4\lambda_i\x_i+\lambda\c
		\end{align}
		with $ 0\le \lambda_i < p_i$ and $\lambda\in\Z$. We can regard $R$ as an $\L$-graded $\mathbf{k}$-algebra by setting $\deg X_i:=\x_i$ for any $i$, hence $R=\bigoplus_{\x\in \L } R_{\x}$, where $R_{\x}$ denotes the homogeneous component of degree $\vec{x}$. Note that $R$ has a unique $\L$-graded maximal ideal $R_+=\bigoplus_{\x\in\L\setminus\{0\}}R_{\x}$.  We call $R$ an \emph{$\L$-graded Brieskorn--Pham} (BP) \emph{singularity of weight type $(p_1,\dots,p_4)$}.
		In the terminology of \cite{HIMO}, the pair $(R,\L)$ is called a
		\emph{Geigle--Lenzing hypersurface of dimension three}.  Note that $R$ has a unique $\L$-graded maximal ideal $R_+=\bigoplus_{\x\in\L\setminus\{0\}}R_{\x}$.
		
		Let $\L_+$ be the submonoid of $\L$ generated by $\x_1,\ldots,\x_4,\c$. Then we equip $\L$ with the structure of a partially ordered set: $\x\le \y$ if and only if $\y-\x\in\L_+$. Each element $\x$ of $\L$ satisfies exactly one of the two conditions
		\begin{align}\label{two possibilities of x}
			\x\geq 0 \text{\quad or\quad}  \x\leq 2\c+\vec{\omega},
		\end{align}
		where  $\w:=\c-\sum_{i=1}^4 \x_i$ is called the \emph{dualizing element} of $\L$. We have $R_{\x} \neq 0$ if and only if $\x \geq 0$ if and only if $\lambda \geq 0$ in its normal form in (\ref{equ:nor}). 
		
			Let $\mod^{\L}R$ denote the category of finitely generated
			$\L$-graded $R$-modules, and  $\CM^{\L}R$ denote its 
			subcategory of 	$\L$-graded maximal Cohen--Macaulay $R$-modules. 
		
		The category of \emph{coherent sheaves} on the
		\emph{Geigle--Lenzing} (\emph{GL}) \emph{projective plane} $\X$ is defined as the quotient category	
		\[ \coh \X :=\mod^{\L}R/\mod_0^{\L}R,\] 
		where $\mod^{\L}_0R$ denotes the  Serre subcategory of $\mod^{\L}R$ consisting of finite dimensional $R$-modules. The category $\coh\X$ is a Hom-finite abelian category of global dimension $2$; see \cite[Theorem 5.6]{HIMO} for $d=2$.		
		Denote by $\pi\colon \mod^{\L}R\to\coh\X$ the natural functor, also called \emph{sheafification}. The object $\OO:=\pi(R)$ is called the \emph{structure sheaf} of $\X$.  Each line bundle has the form $\OO(\x)$ for a uniquely determined $\x$ in $\L$, and  for any $\x,\y \in \L$ and $i\in\Z$,  
		\begin{equation}\label{extension spaces between line bundles}
			\Ext_{\X}^i(\OO(\x),\OO(\y))=\left\{
			\begin{array}{ll}
				R_{\y-\x}&\mbox{if $i=0$,}\\
				D(R_{\x-\y+\w})&\mbox{if $i=2$,}\\
				0&\mbox{otherwise.}
			\end{array}\right.
		\end{equation}
		The category $\coh\X$ satisfies Auslander--Reiten--Serre duality: for any $X,Y\in\coh\X$ and $0\le i\le2$, there are
		bifunctorial isomorphisms
		\begin{equation*}
			\Ext^{2-i}_{\X}(X,Y)\simeq D\Ext^i_{\X}(Y,X(\w)).
		\end{equation*}
		
		Let $\vect\X$ denote the full subcategory of $\coh\X$ consisting
		of vector bundles, and let $\mathscr{L}$ be the full subcategory
		of finite direct sums of line bundles $\OO(\x)$, with $\x\in\L$.
		The subcategory $\vect\X$ is extension-closed in $\coh\X$
		and hence inherits an exact structure from $\coh\X$.
		
			\begin{lemma}\label{kernels of epimorphisms in vect X}
			$\vect\X$ is closed under kernels of
			epimorphisms in $\coh\X$.
		\end{lemma}
		
		\begin{proof}
			Let $f\colon V\to W$ be an epimorphism in $\coh\X$,
			where $V,W\in\vect\X$.
			By \cite[Lemmas~5.9 and~5.18]{HIMO}, we choose
			representatives $M,N\in\mod^{\L}R$ with $\pi(M)\simeq V$ and $\pi(N)\simeq W$, together with a morphism 
			$\widetilde f\colon M\to N$ such that  $\pi(\widetilde f)=f$ under these identifications. 			
			Put $T=\ker\widetilde f$.
			Since $\pi$ is exact and $f$ is an epimorphism, we have
			$\pi(T)\simeq\ker f$ and
			$\coker\widetilde f\in\mod_0^{\L}R$.  For any $\mathfrak p\in\operatorname{Spec}^{\mathbb L}R
			\setminus\{R_+\}$, we have $(\coker\widetilde f)_{(\mathfrak p)}=0$ by
			\cite[Lemma~4.7]{HIMO}. Hence localization  gives an
			exact sequence
			\[
			0\to T_{(\mathfrak p)}
			\to M_{(\mathfrak p)}
			\xrightarrow{\widetilde f_{(\mathfrak p)}}
			N_{(\mathfrak p)}\to0.
			\]
			By \cite[Proposition~5.15(c)]{HIMO}, $M$ and $N$
			are locally free on the punctured spectrum, that is,
			$
			M_{(\mathfrak p)}, N_{(\mathfrak p)}
			\in\proj^{\L}R_{(\mathfrak p)}
			$
			for any $\mathfrak p\ne R_+$.
		Thus this sequence splits, and hence $T_{(\mathfrak p)}\in\proj^{\L}R_{(\mathfrak p)}$ for any $\mathfrak p\ne R_+$, that is, $T$ is locally free on the punctured
		spectrum. Applying \cite[Proposition~5.15(c)]{HIMO} again gives
		$\ker f\simeq\pi(T)\in\vect\X$.
		\end{proof}

		The sheafification  $\pi:\mod^{\L}R\to\coh\X$ restricts to a fully faithful 
		functor $\CM^{\L}R\to\vect\X$ and further induces an equivalence $$\CM^{\L}R \xrightarrow{\sim} \pi(\CM^{\L} R)=:\ACM\X\subset\vect\X.$$
		The objects of $\ACM\X$ are called
		\emph{arithmetically Cohen--Macaulay (ACM) bundles}.
		The following gives a description of $\ACM \X$ in terms of vector bundles by \cite{HIMO}.
		\begin{align}\label{ACM bundle eq.}
			\ACM\X
			&=\{X\in\vect\X\mid
			\Ext^1_{\X}(P,X)=0
			\text{ for any }P\in\mathscr{L}\}\\
			&=\{X\in\vect\X\mid
			\Ext^1_{\X}(X,P)=0
			\text{ for any }P\in\mathscr{L}\}.
			\nonumber
		\end{align}
		Notice that all line bundles serve as the indecomposable projective-injective objects in $\ACM \X$ by (\ref{ACM bundle eq.}). Moreover, $\pi$ restricts to an equivalence $\proj^{\L}R\xrightarrow{\sim}\mathscr{L}$ with the $\L$-graded  indecomposable projective module $R(\x)$ corresponding to the line bundle $\OO(\x)$. Therefore, this turns $\ACM \X$ into a Frobenius category  with the exact structure $\mathscr{E}$ inherited  from $\coh\X$, whose  $\mathscr{E}$-projectives are precisely $\mathscr{L}$. Its stable category $\underline{\ACM}\,\X$ is thus triangle equivalent to $\underline{\CM}^{\L} R$.

			Recall that the \emph{dominant element} of $\L$ is
			\[
			\vdelta := 2\c+2\w=\sum_{i=1}^{4}(p_i-2)\x_i.
			\]
			Note that an element $\x\in\L$ (in the normal form) satisfies $0\leq\x\leq\vdelta$ if and only if
			$\x=\sum_{i=1}^{4}\lambda_i\x_i$ with
			$ 0\le\lambda_i\le p_i-2$ for $1\le i\le4$. 
			
			Let $L$ be a line bundle and $\x\in \L$ with $0\leq\x\leq\vdelta$. Set
			\begin{align*}
				L\langle \x \rangle:=\bigoplus_{1\le i \le 4}L(\x-(1+\lambda_i)\x_i).
			\end{align*}			
			Consider the short exact sequence
			\begin{align*}
				\eta_{1}:\quad  0 \to \ker\gamma \xrightarrow{} L\langle \x \rangle \xrightarrow{\gamma} L(\x) \to 0
			\end{align*}
			in $\coh \X$, where $\gamma:=(X_i^{\lambda_i+1})_{1\le i\le 4}$. 
			By applying $\Hom(-,L(\w))$ to $\eta_1$, we have
			$\Ext^1_{\X}(\ker\gamma,L(\w))\simeq\mathbf{k}$, and thus 
		there exists a nonsplit short exact sequence
		\begin{align*}
			\eta_2: \quad
			0\to L(\w)\xrightarrow{\alpha}E\to \ker\gamma \to0.
		\end{align*} Splicing these two sequences $\eta_{1}$ with $\eta_{2}$ we finally obtain the exact sequence
			\begin{align}
				\eta:& \quad	0 \to L(\w) \xrightarrow{\alpha}  E 
				\xrightarrow{\beta} 
				L\langle \x \rangle \xrightarrow{\gamma} L(\x) \to 0, \label{2-extension bundle exact} 
			\end{align}
			in $\coh \X$. In this case,
			the term $E_L \langle \x \rangle:=E$, determined up to isomorphism, is called the \emph{$2$-extension bundle} given by the data $(L,\x)$. If $L=\OO$, then we just write $E\langle \x \rangle$. For $\x=0$, the term $E_L:=E_L\langle 0 \rangle$ is called the \emph{$2$-Auslander bundle} associated with $L$. Note that $\ker \gamma\in \vect \X$ by Lemma \ref{kernels of epimorphisms in vect X}. 		 Each $2$-extension bundle is an indecomposable ACM bundle \cite{CRW}. 
			Recall from \cite[Proposition 3.8(b)]{CRW} that for any  $0\leq\x\leq\vdelta$ and $\z \in \L$, we have $E_L\langle\x \rangle(\z)\simeq E_{L(\z)}\langle\x \rangle$.

			\begin{proposition}\cite[Proposition 4.6]{CRW}\label{susp.ext.} 
			Let $L$ be a line bundle and $\x=\sum_{i=1}^4\lambda_i\x_i$ with $0\leq\x\leq\vdelta$. 
			Then for any $1\le i\le 4$, we have 
			\begin{align}\label{susp.}
				E_{L} \langle \x \rangle[1] = E_L \langle (p_i-2-\lambda_i)\x_i+\sum_{j\neq i}\lambda_j\x_j \rangle ((1+\lambda_i)\x_i).
			\end{align}
			In particular, we have $E_{L} \langle \x \rangle[2]=E_{L} \langle \x \rangle(\c)$ and $E_{L} \langle \x \rangle=E_{L} \langle \vdelta-\x \rangle(\x-\w-\c)$.  			
		\end{proposition}

		\subsection{Properties (L) and (R)} In this subsection, we show that each ACM bundle on $\X$
		satisfies Properties (L) and (R). These properties will be used in the next subsection. 
		Each line bundle $L$ has a unique expression
		$$L=\OO(\sum_{i=1}^{3}\lambda_i\x_i+a\x_4)$$
		such that $0\le \lambda_i <p_i$ for $1\le i \le 3$, and $a\in \Z$. We call $L$ \emph{persistent} if $\lambda_i\le p_i-2$ for $1\le i \le 3$, and \emph{fading} otherwise. Let $\mathscr{F}\subset\mathscr{L}$ denote the closure of the fading line bundles under finite direct sums and direct summands.

		\begin{definition} \label{Properties (L) and (R)}
			Let $X\in \ACM \X$. 
			\begin{itemize}
				\item[(a)]
			We say that $X$ \emph{has Property (L)} if there exists an exact sequence
			\begin{equation}\label{L equ.}\tag{L}
				0\to X\xrightarrow{f^0}F^1\xrightarrow{f^1}F^2
				\xrightarrow{f^2}F^3\to0
			\end{equation}
			in $\coh \X$,
			where each $F^{i}\in \mathscr{F}$ and $f^0$ is a left $\mathscr{F}$-approximation of $X$. 
			\item[(b)] Dually, 	we say that $X$ \emph{has Property (R)} if there exists an exact sequence
			\begin{equation}\label{R equ.}\tag{R}
				0\to F_3\xrightarrow{f_2}F_2\xrightarrow{f_1}F_1
				\xrightarrow{f_0}X\to0
			\end{equation}
			in $\coh \X$, where each $F_{i}\in \mathscr{F}$ and $f_0$ is a right $\mathscr{F}$-approximation of $X$. 
			\end{itemize} 
		\end{definition}
	
			\begin{lemma} \label{lemma 4.3}
			Let $X\in\ACM\X$ satisfy Property~\textup{(L)}, and choose a
			sequence \eqref{L equ.}. Put $K=\operatorname{Im}f^1$. Then the sequence \eqref{L equ.} is right $2$-exact in $\ACM\X$ and  $\Ext^1_{\X}(K,F)=0$ for any $F\in\mathscr F$.	
					\end{lemma}
				
				\begin{proof}		To show the first assertion, it suffices to show that 	$f^1$ is a weak
					cokernel of $f^0$.	
 			Let $F^1\xrightarrow{p}K\xrightarrow{i}F^2$ be the canonical factorization of $f^1$. 	
			For each $T\in\ACM\X$, we have
			$\Ext^1_{\X}(F^3,T)=0$ by \eqref{ACM bundle eq.}. Applying
			$\Hom_{\X}(-,T)$ to the exact sequence
			\[
			0\to K\xrightarrow{i}F^2\xrightarrow{f^2}F^3\to0
			\]
			shows that every morphism $K\to T$ extends to $F^2$.
			Thus, if $h:F^1\to T$ satisfies $hf^0=0$, then
			$h$ factors through $p$, and hence through $ip=f^1$. Thus $f^1$ is a weak
			cokernel of $f^0$.
			
			For any $F\in\mathscr F$, applying $\Hom(-,F)$ to the 
			exact sequence
			$0\to X\xrightarrow{f^0}F^1\xrightarrow{p} K\to0,$			
			we obtain an exact sequence
						$$\Hom_{\X}(F^1,F)\xrightarrow{(f^0)^\ast} \Hom_{\X}(X,F) \to \Ext^1_{\X}(K,F) \to \Ext^1_{\X}(F^1,F)=0.$$
					Since $f^0$ is a left $\mathscr F$-approximation, $(f^0)^\ast$ is surjective, yielding  $\Ext^1_{\X}(K,F)$=0.
			\end{proof}

	\begin{lemma} \label{lem:lines} Each line bundle has Property (L).
	\end{lemma}	
	\begin{proof}
		Let $L$ be a line bundle. If $L\in \mathscr{F}$, then one can take the exact sequence $$0 \to L \xrightarrow{1} L \to 0\to 0\to 0.$$ In this case, the assertion is trivial. Otherwise we let $L=\OO(\x)$, where  $\x=\sum_{i=1}^{3}\lambda_i\x_i+a\x_4$ with $0\le \lambda_i\le p_i-2$ and $a\in \Z$.   Set $q_i=p_i-1-\lambda_i$ for $1\le i\le 3$. Since $(X_i)_{1\le i \le 3}$ is an $R$-regular sequence, so is $(X_i^{q_i})_{ 1\le i \le 3}$. Hence the corresponding Koszul complex
		\begin{align} \label{Koszul complex of R}
			0\to R(-\sum_{ 1\le i \le 3} q_i\x_i)\to  \bigoplus_{1\le i<j \le 3}  R(-q_i\x_i-q_j\x_j) \to  \bigoplus_{ 1\le i \le 3}  R(-q_i\x_i)\to R \to 0
		\end{align}
		of $R$ is exact except in the rightmost position whose cohomology is $R/(X_i^{q_i}\mid 1\le i \le 3)$. Since this belongs to $\mod_0^{\L} R$, the image of (\ref{Koszul complex of R}) in $\coh \X$ is exact, and then applying degree shift by $\sum_{i=1}^3 (p_i-1)\x_i+a\x_4$, we obtain an exact sequence 	
		\begin{align} \label{Koszul complex L}
			0 \to L \xrightarrow{f^0} \bigoplus_{ 1\le i \le 3} L(q_i\x_i) \xrightarrow{} \bigoplus_{ 1\le i<j \le 3} 
		L(q_i\x_i+q_j\x_j)\xrightarrow{} L(\sum_{ 1\le i \le 3} q_i\x_i)\to 0.
		\end{align} The last three terms belong to $\mathscr{F}$. 	Now let $L'=\OO(\y)\in \mathscr{F}$ and $h:L\to L'$ be a morphism. Without loss of generality, we may assume $h\neq 0$. This implies that $\y\ge \x$ and thus $\y \ge \x+q_j\x_j$ for some $1\le j \le 3$ since $\OO(\y)\in \mathscr{F}$. Then we have 
		$$h\in \Hom(L,L')=X_j^{q_j}\Hom(L(q_j\x_j),L').$$
		In particular $h$ factors through $f^0$. This proves the assertion.
	\end{proof}
		
		 Let $\mathscr{H}$ be the subcategory of $\ACM \X$ consisting of objects with Property  (L).
		
	\begin{lemma} \label{H closed}
		 The subcategory $\mathscr{H}$ contains $\mathscr{L}$ and is closed under stable isomorphisms, extensions, finite direct sums and direct summands. Its image $\underline{\mathscr{H}}$ in $\underline{\ACM}\,\X$ is closed under isomorphisms, extensions, finite direct sums and direct summands.	 	
	\end{lemma}	
		\begin{proof} The proof is divided into the following four steps.
			
			\emph{Step $1$}: We show  that $\mathscr H$ contains $\mathscr L$ and is closed under extensions and finite direct sums.	 			
			Let $
			0\to E'\xrightarrow{i}E\xrightarrow{p}E''
			\rightarrow0$
			be an exact sequence in $\ACM \X$ with $E',E''\in\mathscr{H}$.
			Choose sequences \eqref{L equ.} for $E'$ and $E''$, and
			write their two short exact pieces as
			\[
			0\to E^{\ast}\xrightarrow{a^{\ast}}
			F_0^{\ast}\to K^{\ast}\to0
			\quad \text{and} \quad
			0\to K^{\ast}\xrightarrow{b^\ast} F_1^{\ast}
			\to F_2^{\ast}\to0			
			\] in $\coh \X$, where $\ast\in\{',\,''\}$.
			 Since $\Ext_{\X}^1(E'',F_0')=0$ by \eqref{ACM bundle eq.},  there is a morphism $u\colon E\to F_0'$ such that $ui=a'$. The morphism
			\[
			a=\begin{bmatrix}u\\a''p\end{bmatrix}:
			E\longrightarrow F_0'\oplus F_0''
			\]
			is a monomorphism. Its cokernel $K$ fits into the commutative diagram
			\[
			\begin{tikzcd}[column sep=2.4em,row sep=2em]
				0\arrow[r]&E'\arrow[r,"i"]\arrow[d,"a'"']
				&E\arrow[r,"p"]\arrow[d,"a"]
				&E''\arrow[r]\arrow[d,"a''"]&0\\
				0\arrow[r]&F_0'\arrow[r,"\begin{bsmallmatrix}
					1\\
					0
				\end{bsmallmatrix}"]\arrow[d]
				&F_0'\oplus F_0''\arrow[r,"{\begin{bsmallmatrix}0 \ 1\end{bsmallmatrix}}"]\arrow[d]
				&F_0''\arrow[r]\arrow[d]&0\\
				0\arrow[r]&K'\arrow[r,dotted]&K\arrow[r,dotted]&K''\arrow[r]&0
			\end{tikzcd}
			\]
			with exact rows and columns.
			Moreover, $a$ is a left $\mathscr{F}$-approximation. Indeed, for any
			$h\colon E\to F$ with $F\in\mathscr{F}$, there exists a morphism $t'\colon F_0' \to F$ such that $hi=t'a'$ 
			since $a'$ is a left $\mathscr{F}$-approximation.
			Then $(h-t'u)i=0$ implies that
			$h-t'u$ factors through $p$. Thus we have $h-t'u=sp$ for some morphism $s\colon E''\to F$.
			Since $a''\colon E''\to F_0'' $ is an $\mathscr{F}$-approximation, there is a morphism $t''\colon F_0''\to F$ such that $s=t''a''$, which gives $$h=t'u+t''a''p=\begin{bmatrix}
				t' & t''
			\end{bmatrix}a.$$
			
				Similarly,	by Lemma \ref{lemma 4.3}, we have
			$\Ext^1_{\X}(K'',F_1')=0$.  Hence $b'$ extends  to a morphism $v\colon K\to F_1'$ with $vi'=b'$. The morphism
			\[
			b=\begin{bmatrix}v\\b''p'\end{bmatrix}:
			K\longrightarrow F_1'\oplus F_1''
			\]
			is a monomorphism. Its cokernel $\coker b$ fits	into the  commutative diagram
			\[
			\begin{tikzcd}[column sep=2.4em,row sep=2em]
				0\arrow[r]&K'\arrow[r,"i'"]\arrow[d,"b'"']
				&K\arrow[r,"p'"] \arrow[d,"b"]
				&K''\arrow[r]\arrow[d,"b''"]&0\\
				0\arrow[r]&F_1'\arrow[r,"\begin{bsmallmatrix}
					1\\
					0
				\end{bsmallmatrix}"]\arrow[d]
				&F_1'\oplus F_1''\arrow[r,"{\begin{bsmallmatrix}0 \ 1\end{bsmallmatrix}}"]\arrow[d]
				&F_1''\arrow[r]\arrow[d]&0\\
				0\arrow[r]&F_2'\arrow[r,dotted]&\coker b\arrow[r,dotted]&F_2''\arrow[r]&0
			\end{tikzcd}
			\]	with exact rows and columns. The bottom row splits since
			$\Ext^1_{\X}(F_2'',F_2')=0$
			by \eqref{ACM bundle eq.}. This
			yields an exact sequence
			$
			0\to K\to F_1'\oplus F_1''\to F_2'\oplus F_2''\to0.
			$ 
			Therefore, we have an exact sequence
			\[
			0\to E\xrightarrow{a}F_0'\oplus F_0''
			\to F_1'\oplus F_1''\to F_2'\oplus F_2''\to0			
			\] in $\coh \X$, where $a$ is a left $\mathscr{F}$-approximation of $E$ and each $F_i',F_i''\in \mathscr{F}$. 
			Hence $E\in\mathscr{H}$. Finite direct sums are included in this case. 	Thus by Lemma~\ref{lem:lines}, $\mathscr{H}\supset\mathscr{L}$.
			
		\emph{Step $2$}: We show that $\mathscr{H}$ is closed under  direct summands.				
		Suppose that $E\oplus E'\in\mathscr{H}$, and choose
		\begin{equation}\label{eq:sum-core}
			0\to E\oplus E'\xrightarrow{f}F_0\to F_1\to F_2\to0
		\end{equation}
		as in \eqref{L equ.}. 
		Notice that if $a\colon Z\to A$ is a minimal left $\mathscr F$-approximation,
		then any left $\mathscr F$-approximation $f\colon Z\to F$ can,
		after a suitable identification $F\simeq A\oplus Q$, be written as
		$
		f=\begin{bmatrix}a & 0\end{bmatrix}^\top
		\colon Z\longrightarrow A\oplus Q.
		$
		Moreover, finite direct sums of minimal left
		$\mathscr F$-approximations are again minimal left
		$\mathscr F$-approximations. 
		
		The restrictions of $f$ to the diect summands $E$ and $E'$ are left
		$\mathscr F$-approximations, so we may choose minimal ones
		$a\colon E\to A$ and $a'\colon E'\to A'$.
		Their direct sum is again a minimal left $\mathscr F$-approximation.  Thus we may identify
		\[
		F_0=A\oplus A'\oplus Q
		\quad \text{and} \quad
		f=\begin{bmatrix}a&0\\0&a'\\0&0\end{bmatrix},
		\] where $Q\in \mathscr{F}$. 
		Since $f$ is a monomorphism, so are $a$ and $a'$. Put $K=\coker a$ and
		$K'=\coker a'$. Thus $\coker f=K\oplus K'\oplus Q$. Then \eqref{eq:sum-core} induces an exact sequence
		\[
		0\rightarrow K\oplus K'\oplus Q
		\xrightarrow{j}F_1\rightarrow F_2\rightarrow0.
		\]
		Since $\Ext^1_{\X}(F_2,F)=0$ for $F\in\mathscr{F}$, the map $j$ is
		a left $\mathscr{F}$-approximation. Choose minimal left $\mathscr F$-approximations
		$b\colon K\to B$ and $b'\colon K'\to B'$ in $\coh\X$.
		Since $b\oplus b'\oplus1_Q$ is a minimal left $\mathscr F$-approximation, the same argument gives
		\[
		F_1=B\oplus B'\oplus Q\oplus Q'
		\quad \text{and} \quad
		j=\begin{bmatrix}
			b&0&0\\0&b'&0\\0&0&1_Q\\0&0&0
		\end{bmatrix},
		\] where $Q'\in \mathscr{F}$. 
		Thus $b$ and $b'$ are monomorphisms and
		$F_2\simeq C\oplus C'\oplus Q'$, where
		$C=\coker b$ and $C'=\coker b'$. 
		Since $\mathscr{F}$ is closed under direct summands, we have $C,C'\in\mathscr{F}$.
		Connecting the corresponding short exact sequences yields
		\[
		0\to E\xrightarrow{a}A\to B\to C\to0
		\quad \text{and} \quad
		0\to E'\xrightarrow{a'}A'\to B'\to C'\to0,
		\]
		whose first morphisms are left $\mathscr{F}$-approximations.
		Consequently $E,E'\in\mathscr{H}$.
		
		\emph{Step $3$}: We show that $\underline{\mathscr{H}}$ is
		closed under retracts. 
		Suppose that $\underline E$ is a retract of $\underline G$,
		where $E\in \ACM \X$ and $G\in\mathscr{H}$. Choose representatives $i\colon E\to G$ and
		$p\colon G\to E$ such that $\underline{p}\circ\underline{i}=1_{\underline{E}}$. This means that
		$pi-1_E=ba$  for some factorization
		$E\xrightarrow{a}P\xrightarrow{b}E$ through an object $P\in\mathscr{L}$.
		Then
		\[\begin{bmatrix}p&-b\end{bmatrix}
		\begin{bmatrix}i\\a\end{bmatrix}=1_E,
		\]
		and thus  $E$ is a direct summand of $G\oplus P$. Since 
		$G,P\in\mathscr H$, by Step 1 and Step 2, we have $E\in\mathscr H$ and thus $\underline{E}\in \underline{\mathscr{H}}$.
		This proves that $\underline{\mathscr H}$ is closed under
		retracts. The same argument shows that $\mathscr H$ is closed
		under stable isomorphisms. In particular,
		$\underline{\mathscr H}$ is closed under isomorphisms and
		direct summands.
		
		\emph{Step $4$}: We show that $\underline{\mathscr{H}}$ is closed under  extensions and finite direct sums.	
		
		Let $
		\underline E'\longrightarrow\underline E\longrightarrow
		\underline E''\longrightarrow \underline E'[1]
		$
		be a triangle in $\underline{\ACM} \, \X$ with $\underline E',	\underline E''\in \underline{\mathscr H}$. It is isomorphic to a triangle induced by an 
		$\mathscr E$-conflation
		\[
		0\rightarrow H'\rightarrow H\rightarrow H''\rightarrow0
		\]
		in $\ACM\X$. Since $\underline{\mathscr H}$ is closed under isomorphisms, we have
		$\underline H',\underline H''\in\underline{\mathscr H}.$
		By the stable isomorphism closure of $\mathscr H$ proved in Step~3,
		it follows that
		$H',H''\in\mathscr H.$
		By Step~1, we have $H\in\mathscr H$. Since
		$\underline E\simeq\underline H$ and
		$\underline{\mathscr H}$ is closed under isomorphisms, we obtain
		$\underline E\in\underline{\mathscr H}.$
		Thus $\underline{\mathscr H}$ is extension-closed.
		Since the stable quotient functor is additive, closure under finite direct sums follows from Step~1.
		\end{proof}
		
		We now show that Properties (L) and (R) hold for each ACM bundle
		
		\begin{proposition} \label{Property (L) and (R) for ACM bundles}
			Each ACM bundle has Properties (L) and (R). 
		\end{proposition}
		\begin{proof}
			Set $\vec t=\sum_{i=1}^{3}(p_i-2)\x_i$. For any $a\in \Z$ and $\x=\sum_{i=1}^{4}\lambda_i\x_i$ with $0\le \lambda_i\le p_i-2$, we let 
			$$L=L_{(\x,a)}=\OO(\vec t-\sum_{i=1}^{3}\lambda_i\x_i-\w+a\x_4) \quad \text{and}\quad V=E_L\langle \x \rangle.$$
			The proof is divided into the following three steps.
			
			\emph{Step $1$}:	We  show that $V\in \mathscr{H}$.			
			 We recall that  $0\to P \xrightarrow{\alpha} V \xrightarrow{\beta} F \xrightarrow{\delta} F'\to 0$ is the
 		defining sequence for the $2$-extension bundle $V$, 			
			where $$P=L(\w),\quad F=\bigoplus_{1\le i \le 4}L(\x-(1+\lambda_i)\x_i)\quad\text{and} \quad F'=L(\x).$$
			Note that $P$ is persistent and $F,F'\in \mathscr{F}$. Choose the Koszul sequence \eqref{Koszul complex L} for $P$ constructed in the proof of Lemma \ref{lem:lines},
			$$0\to P\xrightarrow{f^0} U^0\xrightarrow{f^1} U^1 \xrightarrow{f^2} U^2 \to 0,$$
			where each $U^i\in \mathscr{F}$ and $f^0$ is a left $\mathscr{F}$-approximation. Using Auslander--Reiten--Serre duality, we have
			\begin{equation} \label{prop 4.6 equ}
				\Ext_{\X}^2(F',U^0)\simeq D\Hom_{\X}(U^0,P(\x))=\bigoplus_{1\le i\le 3}D\Hom_{\X}(P((\lambda_i+1)\x_i),P(\x)),
			\end{equation}
			which equals zero since the coefficient of $\c$ in the normal form of $\x-(\lambda_i+1)\x_i$ is $-1$ for each $1\le i\le 3$. 
			
			Let $K=\ker\delta$. Applying $\Hom_{\X}(-,U^0)$ to
			$0\to K\to F\to F'\to0$, and using
			\eqref{ACM bundle eq.} and \eqref{prop 4.6 equ},
			gives $\Ext^1_{\X}(K,U^0)=0$. Hence $f^0$ extends to
			$u\colon V\to U^0$. Since $f^1u\alpha=0$, 
			 $f^1u$ factors through $K$, and then the
			resulting morphism $K\to U^1$ extends to $v\colon F\to U^1$ since
			$\Ext^1_{\X}(F',U^1)=0$ by (\ref{extension spaces between line bundles}). The cokernel property of $\delta$
			then gives $w\colon F'\to U^2$. Thus we obtain a commutative diagram
			\begin{equation*}
				\begin{tikzcd}
				0\arrow[r]&	P\arrow[r,"\alpha"]\arrow[d,equal]
					&V\arrow[r,"\beta"]\arrow[d,"u",dotted]
					&F\arrow[r,"\delta"]\arrow[d,"v",dotted]
					&F'\arrow[d,"w",dotted]\arrow[r]&0\\
				0\arrow[r]&	P\arrow[r,"f^0"]
					&U^0\arrow[r,"f^1"]
					&U^1\arrow[r,"f^2"]
					&U^2\arrow[r]&0
				\end{tikzcd}
			\end{equation*}
			with exact rows. Its
			mapping cone is  acyclic. Cancelling the contractible
			summand associated with $1_P$ yields the exact sequence
			$$
					0\longrightarrow V
					\xrightarrow{\begin{bmatrix}
							-\beta\\
							u
					\end{bmatrix}}F\oplus U^0
					\xrightarrow{{\begin{bmatrix}
								-\delta  & 0\\
								v & f^1
								\end{bmatrix}}}
					 F'\oplus U^1
					\xrightarrow{{\begin{bmatrix}
								w  & f^2
					\end{bmatrix}}}
					U^2\longrightarrow0,
			$$
			in which the last three terms belong to $\mathscr{F}$. 
			
		We show that $\begin{bmatrix}-\beta & u\end{bmatrix}^\top$
		is a left $\mathscr F$-approximation.
		Let $h\colon V\to G$ with $G\in\mathscr F$.
		Since $f^0$ is a left $\mathscr F$-approximation,
		there exists a morphism $s\colon U^0\to G$ such that
		$h\alpha=sf^0=su\alpha$. 
		Then $(h-su)\alpha=0$, and thus $h-su$ factors through $K$, and the
		resulting morphism $K\to G$ extends to $t\colon F\to G$ such that $h-su=t\beta$ since
		$\Ext^1_{\X}(F',G)=0$ by (\ref{extension spaces between line bundles}). Thus
		$$
		h=\begin{bmatrix}-t&s\end{bmatrix}
		\begin{bmatrix}-\beta\\u\end{bmatrix}.
		$$		
			This proves that $\begin{bmatrix}-\beta & u\end{bmatrix}^\top$
			is a left $\mathscr F$-approximation and thus  $E_{L}\langle\x\rangle\in\mathscr{H}$.
			
		\emph{Step $2$}:	We show that $\mathscr{H}=\ACM \X$. 			
		 	By \cite[Theorem~6.1]{CRW}, the object
			\[
			T=\bigoplus_{0\leq\y\leq\vdelta}E\langle \y \rangle
			\]
		 	 is  tilting  in
			$\underline{\ACM}\,\X$. In particular,
			$\thick(T)=\underline{\ACM}\,\X$. 	By Proposition \ref{susp.ext.}, we have
			\[
			E_L\langle\x\rangle
			\simeq E\langle\vdelta-\x\rangle
			((a+\lambda_4+2)\x_4).
			\]
			Thus Step~1 shows that $\mathscr H$ contains all
			$E\langle\y\rangle(b\x_4)$ with
			$0\leq\y\leq\vdelta$ and $b\in\mathbb Z$.
			Applying \eqref{susp.} for $i=4$ gives
			$T[m]\in\underline{\mathscr H}$ for any $m\in\mathbb Z$.
			Every object of $\thick(T)$ is a direct summand of a finite
			iterated extension of finite direct sums of suspensions of $T$.
			Since $\underline{\mathscr H}$ is closed under finite direct sums, extensions and direct
			summands by Lemma~\ref{H closed}, it follows that $\underline{\mathscr H}=\thick(T)=\underline{\ACM}\,\X$.		
			Finally, since $\mathscr H$ is closed under stable isomorphisms
			by Lemma~\ref{H closed}, we have $\mathscr H=\ACM\X$.
			
			\emph{Step $3$}: We deduce Property (R) by duality. 			
		 Recall that \emph{vector bundle duality} $(-)^{\vee}=\Hhom{\X}{-}{\OO}\colon \vect\X \to \vect\X$, defined by the sheaf Hom functor, sends the line bundle $\OO(\x)$ to $\OO(-\x)$ for any $\x \in \L$. Denote
			\[
			\theta(X):=X^\vee(\vec t).
			\]
			It preserves ACM bundles. Indeed, since $\theta$ is an exact duality of $\vect \X$ and $\theta^2\simeq \rm Id$,  for any $X\in\ACM\X$ and
			$P\in\mathscr L$, there is an isomorphism
			\[
			\Ext^1_{\X}(\theta(X),P)
			\simeq \Ext^1_{\X}(\theta(P),X),
			\]
			 which equals zero since $\theta(P)\in \mathscr{L}$ and \eqref{ACM bundle eq.}.  Thus $\theta(X)\in \ACM \X$ by \eqref{ACM bundle eq.}. 		 
			 Clearly, the duality $\theta$ preserves persistent line bundles. 
			 It also preserves fading line bundles. Indeed, let $\z=\sum_{i=1}^{3}z_i\x_i+z\x_4$ in the normal form such that $\OO(\z)\in \mathscr{F}$. The duality $\theta$ acts on the coefficient of $\x_i$ by sending			
			$z_i$ to $p_i-2-z_i\pmod{p_i}$. 			
			Since $p_i-1$ is fixed by this operation, the fading condition is
			preserved. Hence $\theta(\mathscr{F})=\mathscr{F}$.

			By Step 2, we apply Property~(L) to $\theta(X)$ and  obtain an exact sequence
			\[
			0\to \theta(X)\xrightarrow{f^0}F^1
			\xrightarrow{f^1}F^2\xrightarrow{f^2}F^3\to0.
			\]
			The kernel of $f^2$ is a vector bundle, so duality preserves
			exactness and yields
			\begin{equation} \label{duality Property (R)}
				0\to \theta(F^3)\to \theta(F^2)
				\to \theta(F^1)\xrightarrow{\theta(f^0)}
				X\to0.
			\end{equation}
			All three terms on the left belong to $\mathscr F$. If $h\colon F\to X$
			with $F\in\mathscr F$, then $\theta(h)$ factors through
			$f^0$; dualizing again shows that $h$ factors through
			$\theta(f^0)$. Thus the latter is a right
			$\mathscr F$-approximation. Finally, a contravariant equivalence sends
			weak cokernels to weak kernels. Hence the sequence \eqref{duality Property (R)} gives
			Property~(R) for $X$.							
		\end{proof}
		
		\begin{corollary} \label{F removable of ACM X}
				$\mathscr{F}$ is a removable subcategory of $(\ACM \X,\mathscr{E})$.
		\end{corollary}
		\begin{proof} Clearly, $\mathscr{F}\subset \mathscr{L}$. 
			For each $X\in \ACM \X$, $X$ has Properties (L) and (R).  Truncating the sequences in Properties (L) and (R), we have sequences	$$X\xrightarrow{f^0}F^1\xrightarrow{f^1}F^2\quad\text{and}\quad F_2\xrightarrow{f_1}F_1\xrightarrow{f_0}X.$$
			By Lemma \ref{lemma 4.3} and its dual,  these sequences yield conditions (a) and (b) of Proposition \ref{A sufficient condition} for $n=1$, respectively. Thus $\mathscr{F}$ is a removable subcategory of $\ACM \X$.
		\end{proof}
		
		\subsection{$2$-exact categories from GL projective planes}

		\begin{proposition}\label{ACM X and vect X} 
			The following classes of subcategories of $\coh \X$  coincide:
			\begin{itemize}
				\item[(a)]	$2$-cluster tilting subcategories of $\ACM \X$;	
				\item[(b)]	$2$-cluster tilting subcategories of  $\vect \X$  containing $\mathscr{L}$.
			\end{itemize}
		\end{proposition}
		\begin{proof}
			This follows from \cite[Theorem 5.23]{HIMO}.
		\end{proof}
		
	As extension-closed subcategories of $\coh\X$, the categories
	$\ACM\X$ and $\vect\X$ inherit exact structures, denoted by
	$\mathscr E$ and $\mathscr E^{\mathrm{vec}}$, respectively.		
	Let $\mathscr C$ be a $2$-cluster tilting subcategory of
	$\ACM\X$. By Proposition \ref{ACM X and vect X}, $\mathscr C$ is also a
	$2$-cluster tilting subcategory of $\vect\X$.
	By Theorem~\ref{n-exact from n-CT}, $\mathscr C$ carries two induced $2$-exact structures $\mathscr X$ and $\mathscr X^{\mathrm{vec}}$. They consist, respectively, of the complexes in
	$\mathbf C_{\mathscr C}^{4}$ that are homotopy equivalent to
	$\mathscr E$-acyclic and $\mathscr E^{\mathrm{vec}}$-acyclic
	complexes of length four with all terms in $\mathscr C$.
	Since every $\mathscr E^{\mathrm{}}$-acyclic complex is $\mathscr E^{\mathrm{vec}}$-acyclic, we have $ \mathscr X^{\mathrm{}} \subset\mathscr X^{\mathrm{vec}}$. But the converse is not true in general, see Example \ref{Example 1}.

	\begin{theorem} \label{main thm 3} Let $\mathscr{C}$ be a $2$-cluster tilting subcategory of $\ACM \X$. Then
			$\mathscr{F}$ is a removable subcategory of $(\mathscr C,\mathscr X)$. Moreover, 
		  $(\mathscr{C}/[\mathscr{F}],\mathscr{X}_{\mathscr{F}})$ is a $2$-exact category, where  $\mathscr{X}_{\mathscr{F}}$ is defined in \eqref{new exact structure}. 		  
		  If moreover, $\mho^2_{\mathscr{E}}(\mathscr{C})\subset\mathscr{C}$, then 	$(\mathscr{C}/[\mathscr{F}],\mathscr{X}_{\mathscr{F}})$ is a Frobenius $2$-exact category.
	\end{theorem}
	\begin{proof} By Corollary  \ref{F removable of ACM X}, $\mathscr{F}$ is a removable subcategory of $(\ACM \X,\mathscr E)$. 
		Together with Theorem \ref{Compatibility with cluster tilting}(a),  $\mathscr{F}$ is a removable subcategory of $(\mathscr C,\mathscr X)$. By Theorem  \ref{main thm1}, the pair $(\mathscr{C}/[\mathscr{F}],\mathscr{X}_{\mathscr{F}})$ is a $2$-exact category. 
		
		Assume $\mho_{\mathscr E}^{2}(\mathscr C)\subset\mathscr C$. Then
		 $(\mathscr{C},\mathscr{X})$ is a Frobenius $2$-exact category by Theorem \ref{n-exact from n-CT}, and  so also is $(\mathscr{C}/[\mathscr{F}],\mathscr{X}_{\mathscr{F}})$ by Theorem \ref{main thm2}.
	\end{proof}	
	
		Objects of $\mathscr F$ need not be $\mathscr X^{\mathrm{vec}}$-projective-injective objects, and so $\mathscr F$ need not be a removable subcategory of $\mathscr C$ with respect to $\mathscr X^{\mathrm{vec}}$. We  consider the following relative construction. Let $$\mathscr X^{\mathrm{rel}}=\{X^{\mr}\in\mathscr X^{\mathrm{vec}}
		\mid d_X^0\text{ is }\mathscr F\text{-monic and }
		d_X^2\text{ is }\mathscr F\text{-epic}\}.$$

		\begin{theorem} \label{main thm 4}
			Let $\mathscr{C}$ be a $2$-cluster tilting subcategory of $\ACM \X$.
			Then $(\mathscr C,\mathscr X^{\mathrm{rel}})$ is a $2$-exact category, and $\mathscr{F}$ is a removable subcategory of $(\mathscr C,\mathscr X^{\mathrm{rel}})$. Moreover,
			$(\mathscr{C}/[\mathscr{F}],\mathscr X^{\mathrm{rel}}_{\mathscr{F}})$ is a  $2$-exact category, where $\mathscr X^{\mathrm{rel}}_{\mathscr{F}}$ is defined in \eqref{new exact structure}. 							 
		\end{theorem}
		\begin{proof}
			By Proposition \ref{rel exact} for the $2$-exact category $(\mathscr{C},\mathscr X^{\mathrm{vec}})$, the pair $(\mathscr C, \mathscr X^{\mathrm{rel}})$ is a $2$-exact category, and  every object of $\mathscr{F}$ is $\mathscr X^{\mathrm{rel}}$-projective-injective. For $X\in\mathscr C$, by Proposition~\ref{Property (L) and (R) for ACM bundles}, there are exact sequences in $\coh\X$
			\[
			0\rightarrow X\xrightarrow{f^0}F^1
			\xrightarrow{f^1}F^2\xrightarrow{f^2}F^3\rightarrow0,
			\]
			\[
			0\longrightarrow F_3\xrightarrow{f_2}F_2
			\xrightarrow{f_1}F_1\xrightarrow{f_0}X\rightarrow0,
			\]
			where all $F^i,F_i$ lie in $\mathscr F$, $f^0$ is a left
			$\mathscr F$-approximation and $f_0$ is a right
			$\mathscr F$-approximation. By Lemma~\ref{lemma 4.3}
			and its dual, these sequences satisfy the conditions in
			Proposition~\ref{A sufficient condition}, with $n=2$. This
			shows that $\mathscr F$ is a removable subcategory of $(\mathscr C,\mathscr X^{\mathrm{rel}})$. The quotient assertion
			now follows from Theorem~\ref{main thm1}.
		\end{proof}
		
			\begin{proposition}\label{relative criterion}
			Let $\mathscr C$ be a $2$-cluster tilting subcategory of $\ACM\X$. Then
			$\mathscr X\subset\mathscr X^{\mathrm{rel}}
			\subset\mathscr X^{\mathrm{vec}}.$ 
			Moreover, $\mathscr X^{\mathrm{rel}}=\mathscr X$ if and only if 
			$\min\{p_1,p_2,p_3\}=2$.
		\end{proposition}
		\begin{proof}
			Every line bundle is projective and injective for the exact structure $\mathscr E$ on $\ACM\X$.
			Hence the first and last differentials of any $\mathscr E$-acyclic
			complex with terms in $\mathscr C$ are $\mathscr L$-monic and
			$\mathscr L$-epic, respectively. These properties are preserved
			under homotopy equivalence. Indeed, applying $\Hom(-,L)$ and
			$\Hom(L,-)$ to homotopy equivalent complexes yields complexes
			with isomorphic cohomology for each $L\in\mathscr L$. Since $\mathscr{F}\subset\mathscr L$, the asserted inclusions follow.
		 
		 	Suppose first that $p_j=2$ for some $1\le j\le3$. We claim that every line		 	
		 	bundle $L$ satisfies 
		 	\begin{align}\label{L fading or Lw}
		 		L\in\mathscr F\quad \text{or}\quad L(\w)\in\mathscr F.
		 	\end{align}
		 	Write $L=\OO(\x)$, where
		 	$\x=\sum_{i=1}^3\lambda_i\x_i+a\x_4$ with
		 	$0\le\lambda_i<p_i$ and $a\in \Z$. Assume $L\notin\mathscr F$. Then $L$ is
		 	persistent, so $\lambda_j=0$. Since
		 	$\w=(p_4-1)\x_4-\sum_{i=1}^3\x_i$, the coefficient of $\x_j$
		 	in the normal form of $\x+\w$ is $p_j-1=1$. Hence 
		 	$L(\w)\in\mathscr F$.
		 
			By the homotopy invariance established above and the closure
			of $\mathscr X$ under homotopy equivalences, the proof of
			$\mathscr X^{\mathrm{rel}}\subset\mathscr X$ reduces to showing
			that every $\mathscr E^{\mathrm{vec}}$-acyclic complex
			\[
			X^{\mr}\colon\quad
			0\longrightarrow X^0\xrightarrow{d_X^0}X^1
			\xrightarrow{d_X^1}X^2\xrightarrow{d_X^2}X^3
			\longrightarrow0,
			\]
			with all terms in $\mathscr C$, $d_X^0$ $\mathscr F$-monic,
			and $d_X^2$ $\mathscr F$-epic, belongs to $\mathscr X$.
		 	 Put $K=\Im d_X^1$.   For any fading line bundle $F$, 
		 	  applying $\Hom_{\X}(-,F)$ and $\Hom_{\X}(F,-)$ to the two short exact sequences
		 	  \begin{align}\label{two short exact sequences}
		 	  	0\to X^0\xrightarrow{d_X^0}X^1 \to K \to0\quad\text{and}\quad0\to K\xrightarrow{}X^2 \xrightarrow{d_X^2} X^3 \to0,
		 	  \end{align}  this gives, respectively, two exact sequences
		 	\begin{align*}
		 		\Hom_{\X}(X^1,F)&\xrightarrow{(d_X^0)^\ast} \Hom_{\X}(X^0,F) \to \Ext_{\X}^1(K,F) \to 0,\\
		 		\Hom_{\X}(F,X^2)&\xrightarrow{(d_X^2)_\ast} \Hom_{\X}(F,X^3) \to \Ext_{\X}^1(F,K) \to 0.
		 	\end{align*}
		 Since $d_X^0$ is $\mathscr F$-monic
		 and $d_X^2$ is $\mathscr F$-epic, both the induced morphisms $(d_X^0)^\ast$ and $(d_X^2)_\ast$ are epimorphisms. Thus for any line bundle $L\in \mathscr{F}$, we have
		 \begin{align}\label{check ACM}
		 	\Ext_{\X}^1(K,L)=0=\Ext_{\X}^1(L,K).
		 \end{align}
		 
		 Let $L$ be a line bundle. If $L$ is fading, then $\Ext_{\X}^1(L,K)=0$ by \eqref{check ACM}. Otherwise $L(\w)$ is fading by \eqref{L fading or Lw}, and then by Auslander--Reiten--Serre duality, we have
		 \begin{align}\label{check ACM L}
		 	\Ext_{\X}^1(L,K)\simeq D\Ext_{\X}^1(K,L(\w))=0.
		 \end{align}
		 Thus $K$ is an ACM bundle by  \eqref{ACM bundle eq.}. Both the short exact sequences \eqref{two short exact sequences} belong to $\mathscr E$, so $X^{\mr}\in \mathscr X$.

		 Conversely, assume that $p_1,p_2,p_3\ge3$. We construct a
		 sequence in $\mathscr X^{\mathrm{rel}}$ which does not belong
		 to $\mathscr X$.	
		 	Put $Q:=R/(X_1,X_2,X_3)$ and
		 $\vec t:=\x_1+\x_2+\x_3$. Consider the Koszul complex
		 \begin{equation*}\label{relative Koszul}
		 			 	P^{\mr}\colon\quad
		 			 	0\rightarrow\OO\xrightarrow{d^0}\bigoplus_{1\le i\le 3}\OO(\x_i)
		 			 	\xrightarrow{d^1}\bigoplus_{1\le i<j\le 3}\OO(\x_i+\x_j)\xrightarrow{d^2}\OO(\x_1+\x_2+\x_3)
		 			 	\rightarrow0,
		 			 \end{equation*}
		 	in $\coh \X$, where
		 			 \[
		 			 d^0=
		 			 \begin{bmatrix}
		 				 	X_1\\
		 				 	X_2\\
		 				 	X_3
		 				 \end{bmatrix},
		 			 \quad
		 			 d^1=
		 			 \begin{bmatrix}
		 				 	X_2 & -X_1 & 0\\
		 				 	-X_3 & 0 & X_1\\
		 				 	0 & X_3 & -X_2
		 				 \end{bmatrix},
		 			 \quad
		 			 d^2=
		 			 \begin{bmatrix}
		 				 	X_3 & X_2 & X_1
		 				 \end{bmatrix}.
		 			 \]			 		 		 
		By Lemma \ref{kernels of epimorphisms in vect X}, we have $\ker d^2\in\vect\X$. 
		 Consequently, $P^{\mr}$ is $\mathscr E^{\mathrm{vec}}$-acyclic. 		 
	  For any $\vec a\in\L$, the identification
		 $\Hom_{\X}(\OO(\vec a),\OO(\vec b))=R_{\vec b-\vec a}$ by \eqref{extension spaces between line bundles}, 
		 gives
		 \begin{align*}
		 	\operatorname{Coker}\Hom_{\X}(d^0,\OO(\vec a))
		 	&\simeq
		 	R_{\vec a}\Big/\sum_{i=1}^3X_iR_{\vec a-\x_i}
		 	=Q_{\vec a},\\
		 	\operatorname{Coker}\Hom_{\X}(\OO(\vec a),d^2)
		 	&\simeq
		 	R_{\vec t-\vec a}\Big/
		 	\sum_{i=1}^3X_iR_{\vec t-\vec a-\x_i}
		 	=Q_{\vec t-\vec a}.
		 \end{align*}
		 The nonzero homogeneous components of $Q$ occur precisely
		 in degrees $r\x_4$ with $0\le r<p_4$. Thus the first
		 cokernel can be nonzero only when
		 $\vec a=r\x_4$, and the second only when
		 $\vec a=\vec t-r\x_4$.
		 The coefficients of $\x_1,\x_2,\x_3$ in the normal form of $r\x_4$ and $\vec t-r\x_4$ are $(0,0,0)$ and $(1,1,1)$, respectively.
		 Since $p_1,p_2,p_3\ge3$, both $\OO(r\x_4)$ and $\OO(\vec t-r\x_4)$ are
		 persistent. Both cokernels therefore vanish for every
		 fading line bundle. It follows that $d^0$ is
		 $\mathscr F$-monic and $d^2$ is $\mathscr F$-epic, so
		 $P^{\mr}\in\mathscr X^{\mathrm{rel}}$.
		 
		 On the other hand, taking $\vec a=\vec t$, we have
		 \[
		 \operatorname{Coker}
		 \Hom_{\X}(\OO(\vec t),d^2)
		 \simeq Q_0\simeq\mathbf k\ne0.
		 \]
		 Thus $d^2$ is not $\mathscr L$-epic.
		 As observed at the beginning of the proof, every sequence
		 in $\mathscr X$ has an $\mathscr L$-epic last differential.
		 Hence $P^{\mr}\notin\mathscr X$, and therefore
		 $\mathscr X\subsetneq\mathscr X^{\mathrm{rel}}$.
		\end{proof}
				
				In the rest, we show how some of the main constructions work out in some particular and simple examples.
	
		\begin{example}\label{Example 1} Let $R$ be an $\L$-graded BP singularity of type $(2,2,2,2)$, and 
			 $\X$ the corresponding GL projective plane. The Auslander--Reiten quiver of $\ACM \X$ is depicted below; see \cite[Example 4.47]{HIMO}. Here, we write $\x\in\L$ for the corresponding line bundle
			 $\OO(\x)$, and set 
			 $E_k:=E\langle0\rangle(k\x_4)$  for $k\in \Z$.
			
		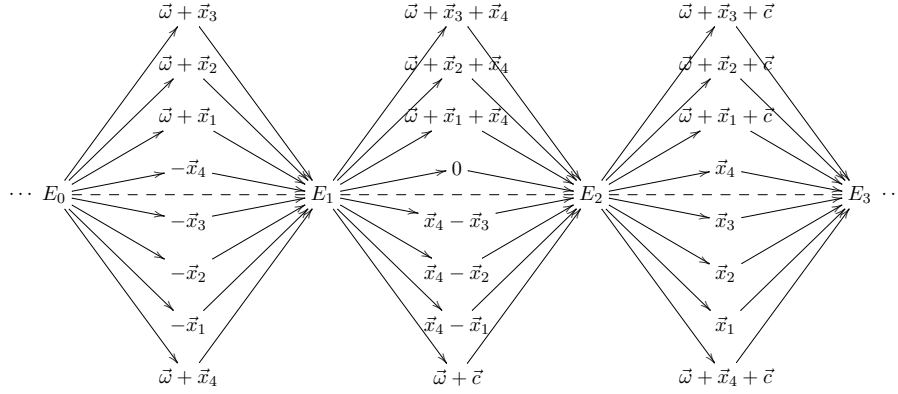
\begin{figure}[H]
			\centering
			\resizebox{0.95\textwidth}{!}{%
				\large
				\begin{xy}
					0;<4.5pt,0pt>:<0pt,14pt>::
					(-1,0)   *+{\cdots}             ="dl",
					(3,0)   *+{E_0} 					 ="e0",
					(35,0)  *+{E_1} 				 ="e1",
					(67,0)  *+{E_2}  			     ="e2",
					(99,0)  *+{E_3} 				 ="e3",
					(103,0) *+{\cdots}               ="dr",
					(19, 7) *+{\w+\x_3}              ="a1",
					(19, 5) *+{\w+\x_2}              ="a2",
					(19, 3) *+{\w+\x_1}              ="a3",
					(19, 1) *+{-\x_4}             	 ="a4",
					(19,-1) *+{-\x_3}           	 ="a5",
					(19,-3) *+{-\x_2}           	 ="a6",
					(19,-5) *+{-\x_1}          		 ="a7",
					(19,-7) *+{\w+\x_4}              ="a8",
					(51, 7) *+{\w+\x_3+\x_4}         ="b1",
					(51, 5) *+{\w+\x_2+\x_4}         ="b2",
					(51, 3) *+{\w+\x_1+\x_4}         ="b3",
					(51, 1) *+{0}                 	 ="b4",
					(51,-1) *+{\x_4-\x_3}       	 ="b5",
					(51,-3) *+{\x_4-\x_2}            ="b6",
					(51,-5) *+{\x_4-\x_1}            ="b7",
					(51,-7) *+{\w+\c}                ="b8",
					(83, 7) *+{\w+\x_3+\c}           ="c1",
					(83, 5) *+{\w+\x_2+\c}           ="c2",
					(83, 3) *+{\w+\x_1+\c}           ="c3",
					(83, 1) *+{\x_4}                 ="c4",
					(83,-1) *+{\x_3}                 ="c5",
					(83,-3) *+{\x_2}                 ="c6",
					(83,-5) *+{\x_1}                 ="c7",
					(83,-7) *+{\w+\x_4+\c}           ="c8",
					"e0", {\ar"a1"},
					"e0", {\ar"a2"},
					"e0", {\ar"a3"},
					"e0", {\ar"a4"},
					"e0", {\ar"a5"},
					"e0", {\ar"a6"},
					"e0", {\ar"a7"},
					"e0", {\ar"a8"},
					"a1", {\ar"e1"},
					"a2", {\ar"e1"},
					"a3", {\ar"e1"},
					"a4", {\ar"e1"},
					"a5", {\ar"e1"},
					"a6", {\ar"e1"},
					"a7", {\ar"e1"},
					"a8", {\ar"e1"},
					"e1", {\ar"b1"},
					"e1", {\ar"b2"},
					"e1", {\ar"b3"},
					"e1", {\ar"b4"},
					"e1", {\ar"b5"},
					"e1", {\ar"b6"},
					"e1", {\ar"b7"},
					"e1", {\ar"b8"},
					"b1", {\ar"e2"},
					"b2", {\ar"e2"},
					"b3", {\ar"e2"},
					"b4", {\ar"e2"},
					"b5", {\ar"e2"},
					"b6", {\ar"e2"},
					"b7", {\ar"e2"},
					"b8", {\ar"e2"},
					"e2", {\ar"c1"},
					"e2", {\ar"c2"},
					"e2", {\ar"c3"},
					"e2", {\ar"c4"},
					"e2", {\ar"c5"},
					"e2", {\ar"c6"},
					"e2", {\ar"c7"},
					"e2", {\ar"c8"},
					"c1", {\ar"e3"},
					"c2", {\ar"e3"},
					"c3", {\ar"e3"},
					"c4", {\ar"e3"},
					"c5", {\ar"e3"},
					"c6", {\ar"e3"},
					"c7", {\ar"e3"},
					"c8", {\ar"e3"},
					"e1", {\ar@{--} "e0"},
					"e2", {\ar@{--} "e1"},
					"e3", {\ar@{--} "e2"},
				\end{xy}%
			}\caption{The Auslander--Reiten quiver of $\ACM \X$ for type $(2,2,2,2)$}
		\end{figure}\noindent
			In this case, the persistent line bundles are precisely
		$\OO(a\x_4)$, $a\in\Z$, whereas
		\[
		\mathscr F
		=
		\add\{
		\OO(\sum_{i=1}^{3}\lambda_i\x_i+a\x_4)
		\ \mid \lambda_i\in\{0,1\},\
		(\lambda_1,\lambda_2,\lambda_3)\ne(0,0,0),\ a\in\Z\}.
		\]

		By \cite[Example~4.54]{HIMO}, the subcategory
		\[
		\mathscr C
		=
		\add\{E_{2k},\,\OO(\x)\mid k\in\Z,\ \x\in\L\}
		\]
		 is a $2$-cluster tilting subcategory of $\ACM\X$. 	One can compute the Auslander--Reiten quiver of $\mathscr{C}$  from that of  $\ACM\X$; see Figure \ref{The quiver of C}.
		\begin{figure}
			\centering
			\resizebox{0.95\textwidth}{!}{%
				\large
				\begin{xy}
					0;<4.5pt,0pt>:<0pt,15pt>::
					(-21, 7) *+{\w+\x_3-\x_4} ="p1",
					(-21, 5) *+{\w+\x_2-\x_4} ="p2",
					(-21, 3) *+{\w+\x_1-\x_4} ="p3",
					(-21, 1) *+{-\c}           ="p4",
					(-21,-1) *+{-\x_3-\x_4}   ="p5",
					(-21,-3) *+{-\x_2-\x_4}   ="p6",
					(-21,-5) *+{-\x_1-\x_4}   ="p7",
					(-21,-7) *+{\w}            ="p8",
					(-30,0) *+{\cdots} ="dl",
					(3,0)   *+{E_0}      ="e0",
					(99,0)  *+{E_2}    ="e2",
					(105,0) *+{\cdots} ="dr",
					(27, 7) *+{\w+\x_3} ="a1",
					(27, 5) *+{\w+\x_2} ="a2",
					(27, 3) *+{\w+\x_1} ="a3",
					(27, 1) *+{-\x_4}   ="a4",
					(27,-1) *+{-\x_3}   ="a5",
					(27,-3) *+{-\x_2}   ="a6",
					(27,-5) *+{-\x_1}   ="a7",
					(27,-7) *+{\w+\x_4} ="a8",
					(75, 7) *+{\w+\x_3+\x_4} ="b1",
					(75, 5) *+{\w+\x_2+\x_4} ="b2",
					(75, 3) *+{\w+\x_1+\x_4} ="b3",
					(75, 1) *+{0}            ="b4",
					(75,-1) *+{\x_4-\x_3}    ="b5",
					(75,-3) *+{\x_4-\x_2}    ="b6",
					(75,-5) *+{\x_4-\x_1}    ="b7",
					(75,-7) *+{\w+\c}        ="b8",
					"p1", {\ar"e0"},
					"p2", {\ar"e0"},
					"p3", {\ar"e0"},
					"p4", {\ar"e0"},
					"p5", {\ar"e0"},
					"p6", {\ar"e0"},
					"p7", {\ar"e0"},
					"p8", {\ar"e0"},
					"e0", {\ar"a1"},
					"e0", {\ar"a2"},
					"e0", {\ar"a3"},
					"e0", {\ar"a4"},
					"e0", {\ar"a5"},
					"e0", {\ar"a6"},
					"e0", {\ar"a7"},
					"e0", {\ar"a8"},
					"a1", {\ar"b1"},
					"a1", {\ar"b6"},
					"a1", {\ar"b7"},
					"a1", {\ar"b8"},
					"a2", {\ar"b2"},
					"a2", {\ar"b5"},
					"a2", {\ar"b7"},
					"a2", {\ar"b8"},
					"a3", {\ar"b3"},
					"a3", {\ar"b5"},
					"a3", {\ar"b6"},
					"a3", {\ar"b8"},
					"a4", {\ar"b4"},
					"a4", {\ar"b5"},
					"a4", {\ar"b6"},
					"a4", {\ar"b7"},
					"a5", {\ar"b2"},
					"a5", {\ar"b3"},
					"a5", {\ar"b4"},
					"a5", {\ar"b5"},
					"a6", {\ar"b1"},
					"a6", {\ar"b3"},
					"a6", {\ar"b4"},
					"a6", {\ar"b6"},
					"a7", {\ar"b1"},
					"a7", {\ar"b2"},
					"a7", {\ar"b4"},
					"a7", {\ar"b7"},
					"a8", {\ar"b1"},
					"a8", {\ar"b2"},
					"a8", {\ar"b3"},
					"a8", {\ar"b8"},
					"b1", {\ar"e2"},
					"b2", {\ar"e2"},
					"b3", {\ar"e2"},
					"b4", {\ar"e2"},
					"b5", {\ar"e2"},
					"b6", {\ar"e2"},
					"b7", {\ar"e2"},
					"b8", {\ar"e2"},
				\end{xy}%
			}
			\caption{The Auslander--Reiten quiver of  
				$\mathscr C$}
				\label{The quiver of C}
		\end{figure}
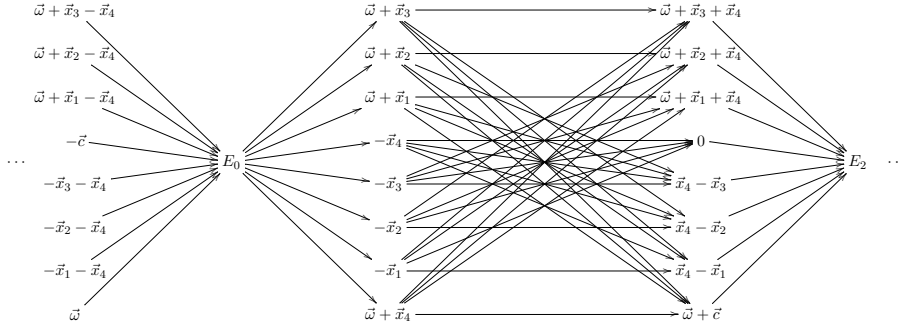	\noindent
			
		Let $\mathscr{E}$ be the natural exact structure on $\ACM\X$ inherited  from $\coh\X$, and $\mathscr{X}$ be the induced $2$-exact structure on $\mathscr{C}$ as defined in Theorem \ref{n-exact from n-CT}. 			
		By \cite[Theorem~4.4 and Proposition~4.6]{CRW},
		there are $\mathscr E$-conflations
		\begin{align}\label{E injective hull}
			0\longrightarrow E_k\longrightarrow I_k
			\longrightarrow E_{k+1}\longrightarrow0,
		\end{align}
		where
		\[
		I_k=
		\bigoplus_{1\le i\le 4}\OO(k\x_4-\x_i)
		\oplus
		\bigoplus_{1\le i\le 4}\OO(k\x_4+\w+\x_i).
		\]
		Consequently, we have
		$\mho_{\mathscr E}^{2}(E_{2k})=E_{2k+2}$.
		Since the remaining generators of $\mathscr C$ are
		projective and injective, it follows that
		$\mho_{\mathscr E}^{2}(\mathscr C)\subset\mathscr C.$ By Corollary~\ref{F removable of ACM X},
		$\mathscr F$ is removable of $(\ACM\X,\mathscr E)$.  		
		By Theorem~\ref{Compatibility with cluster tilting}, 
		$$
		\mathscr C/[\mathscr F]
		=
		\add\{{E}_{2k},\,
		\OO(a\x_4)
		\mid k,a\in\Z\}
		$$
		is a $2$-cluster tilting subcategory of the Frobenius exact
		category $\ACM\X/[\mathscr F]$.
		Its induced $2$-exact structure is $\mathscr X_{\mathscr F}$.
		Moreover, both $(\mathscr C,\mathscr X)$ and
		$(\mathscr C/[\mathscr F],\mathscr X_{\mathscr F})$
		are Frobenius $2$-exact categories, whose subcategories of projective-injective objects
		are $\mathscr L$ and
		$
		\mathscr L/[\mathscr F]
		=
		\add\{{\OO(a\x_4)}\mid a\in\Z\},
		$
		respectively.			
		\begin{figure}[H]
			\centering
			\resizebox{0.92\textwidth}{!}{%
				\large
				\begin{xy}
					0;<4.5pt,0pt>:<0pt,12pt>::
					(0,0)  *+{\cdots} ="l",
					(10,0)   *+{E_{-2}} ="em2",
					(20,4)  *+{-3\x_4} ="lm3",
					(30,4)  *+{-\c}    ="lm2",
					(40,0)  *+{E_0}      ="e0",
					(50,4)  *+{-\x_4}  ="lm1",
					(60,4)  *+{0}      ="l0",
					(70,0)  *+{E_2}    ="e2",
					(80,4)  *+{\x_4}   ="l1",
					(90,4)  *+{\c}     ="l2",
					(100,0) *+{E_4}    ="e4",
					(110,0) *+{\cdots} ="r",
					"l",   {\ar"em2"},
					"em2", {\ar"lm3"},
					"lm3", {\ar"lm2"},
					"lm2", {\ar"e0"},
					"e0",  {\ar"lm1"},
					"lm1", {\ar"l0"},
					"l0",  {\ar"e2"},
					"e2",  {\ar"l1"},
					"l1",  {\ar"l2"},
					"l2",  {\ar"e4"},
					"e4",  {\ar"r"},
				\end{xy}%
			}
			\caption{The Auslander--Reiten quiver of $\mathscr C/[\mathscr F]$.}
		\end{figure}
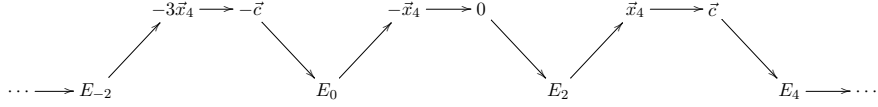			
			
		For a concrete $\mathscr X$-conflation, splice the two $\mathscr E$-conflations of the form \eqref{E injective hull}
		 with indices $2k$ and $2k+1$. This gives the
		$\mathscr X$-conflation
		$$
		E_{2k}\longrightarrow I_{2k}\longrightarrow I_{2k+1}
		\longrightarrow E_{2k+2}.
		$$
		The only persistent summand of $I_k$ is
		$\OO((k-1)\x_4)$. Hence its image in the quotient $\mathscr C/[\mathscr F]$ is
		isomorphic to the $\mathscr X_{\mathscr F}$-conflation
		\[
		{E}_{2k}
		\longrightarrow
		{\OO(k\c-\x_4)}
		\longrightarrow
		{\OO(k\c)}
		\longrightarrow
		{E}_{2k+2}.
		\]

		We note that $\mathscr X^{\mathrm{vec}}$ is strictly larger than $\mathscr X$.
		Indeed, the defining sequence
		\[
		0\longrightarrow\OO(\w)\xrightarrow{\alpha}
		E_0\longrightarrow
		\bigoplus_{1\le i\le 4}\OO(-\x_i)
		\longrightarrow\OO\longrightarrow0
		\]
		belongs to $\mathscr X^{\mathrm{vec}}$ but not to
		$\mathscr X$. If this sequence belonged to $\mathscr X$, injectivity of $\OO(\w)$ would force $\alpha$ to split, contradicting the indecomposability of $E_0$.
		\end{example}
		
		\begin{example}
			Let $\X$ be a GL projective plane of type $(2,2,2,4)$. The Auslander--Reiten quiver of $\ACM \X$ is depicted in Figure \ref{The quiver of ACM X(2,2,2,4)}; see \cite{CRW2,HIMO} for more details. 			
				In this case, the persistent line bundles are precisely
			$\OO(a\x_4)$, $a\in\Z$, whereas
			\[
			\mathscr F
			=
			\add\{
			\OO(\sum_{i=1}^{3}\lambda_i\x_i+a\x_4)
			\ \mid \lambda_i\in\{0,1\},\
			(\lambda_1,\lambda_2,\lambda_3)\ne(0,0,0),\ a\in\Z\}.
			\]		
			\begin{figure}
				\centering
				\resizebox{0.95\textwidth}{!}{%
					\begin{xy} 0;<18pt,0pt>:<0pt,20pt>::
						(15,4) *+{\dots},
						(16,8) *+{\bullet} ="42",
						(16,4) *+{\cdot} ="23",
						(18,6) *+{\langle 2,-3\rangle} ="33",
						(20,8) *+{(-\c)} ="43",
						(16,0) *+{\bullet} ="04",
						(18,2) *+{\langle 0,-2\rangle} ="14",
						(20,4) *+{\langle 1,-2\rangle} ="24",
						(22,6) *+{\langle 2,-2\rangle} ="34",
						(24,8) *+{(-3\x_4)} ="44",
						(20,0) *+{(\w-\x_4)} ="05",
						(22,2) *+{\langle 0,-1\rangle} ="15",
						(24,4) *+{\langle 1,-1\rangle} ="25",
						(26,6) *+{\langle 2,-1\rangle} ="35",
						(28,8) *+{(-2\x_4)} ="45",
						(24,0) *+{(\w)} ="06",
						(26,2) *+{\langle 0,0\rangle} ="16",
						(28,4) *+{\langle 1,0\rangle} ="26",
						(30,6) *+{\langle 2,0\rangle} ="36",
						(32,8) *+{(-\x_4)} ="46",
						(28,0) *+{(\w+\x_4)} ="07",
						(30,2) *+{\langle 0,1\rangle} ="17",
						(32,4) *+{\langle 1,1\rangle} ="27",
						(34,6) *+{\langle 2,1\rangle} ="37",
						(36,8) *+{(0)} ="47",
						(32,0) *+{(\w+2\x_4)} ="08",
						(34,2) *+{\langle 0,2\rangle} ="18",
						(36,4) *+{\langle 1,2\rangle} ="28",
						(38,6) *+{\langle 2,2\rangle} ="38",
						(40,8) *+{\bullet} ="48",
						(36,0) *+{(\w+3\x_4)} ="09",
						(38,2) *+{\langle 0,3\rangle} ="19",
						(40,4) *+{\cdot} ="29",
						(40,0) *+{\bullet} ="010",
						(41,4) *+{\dots},
						"42", {\ar "33"},
						"23", {\ar "33"},
						"33", {\ar "24"},
						"33", {\ar "43"},
						"23", {\ar "14"},
						"43", {\ar "34"},
						"04", {\ar "14"},
						"14", {\ar "24"},
						"34", {\ar "44"},
						"14", {\ar "05"},
						"24", {\ar "15"},
						"24", {\ar "34"},
						"34", {\ar "25"},
						"44", {\ar "35"},
						"05", {\ar "15"},
						"15", {\ar "25"},
						"35", {\ar "45"},
						"35", {\ar "26"},
						"15", {\ar "06"},
						"25", {\ar "16"},
						"25", {\ar "35"},
						"45", {\ar "36"},
						"06", {\ar "16"},
						"16", {\ar "26"},
						"36", {\ar "46"},
						"36", {\ar "27"},
						"16", {\ar "07"},
						"26", {\ar "17"},
						"26", {\ar "36"},
						"46", {\ar "37"},
						"07", {\ar "17"},
						"17", {\ar "27"},
						"37", {\ar "47"},
						"37", {\ar "28"},
						"27", {\ar "37"},
						"17", {\ar "08"},
						"27", {\ar "18"},
						"47", {\ar "38"},
						"08", {\ar "18"},
						"18", {\ar "28"},
						"28", {\ar "38"},
						"38", {\ar "48"},
						"38", {\ar "29"},
						"18", {\ar "09"},
						"28", {\ar "19"},
						"09", {\ar "19"},
						"19", {\ar "29"},
						"19", {\ar "010"},
						(20,8.8) *+{-\c},
						(20,9.6) *+{-\x_2-\x_3},
						(20,10.4) *+{-\x_1-\x_3},
						(20,11.2) *+{-\x_1-\x_2},
						(24,8.8) *+{-3\x_4},
						(24,9.6) *+{\x_2-\x_3-3\x_4},
						(24,10.4) *+{\x_1-\x_3-3\x_4},
						(24,11.2) *+{\x_1-\x_2-3\x_4},
						(28,8.8) *+{-2\x_4},
						(28,9.6) *+{\x_2-\x_3-2\x_4},
						(28,10.4) *+{\x_1-\x_3-2\x_4},
						(28,11.2) *+{\x_1-\x_2-2\x_4},
						(32,8.8) *+{-\x_4},
						(32,9.6) *+{\x_2-\x_3-\x_4},
						(32,10.4) *+{\x_1-\x_3-\x_4},
						(32,11.2) *+{\x_1-\x_2-\x_4},
						(36,8.8) *+{0},
						(36,9.6) *+{\x_2-\x_3},
						(36,10.4) *+{\x_1-\x_3},
						(36,11.2) *+{\x_1-\x_2},
						(20,-0.8) *+{-\x_3-2\x_4},
						(20,-1.6) *+{-\x_2-2\x_4},
						(20,-2.4) *+{-\x_1-2\x_4},
						(20,-3.2) *+{\w-\x_4},
						(24,-0.8) *+{-\x_3-\x_4},
						(24,-1.6) *+{-\x_2-\x_4},
						(24,-2.4) *+{-\x_1-\x_4},
						(24,-3.2) *+{\w},
						(28,-0.8) *+{-\x_3},
						(28,-1.6) *+{-\x_2},
						(28,-2.4) *+{-\x_1},
						(28,-3.2) *+{\w+\x_4},
						(32,-0.8) *+{-\x_3+\x_4},
						(32,-1.6) *+{-\x_2+\x_4},
						(32,-2.4) *+{-\x_1+\x_4},
						(32,-3.2) *+{\w+2\x_4},
						(36,-0.8) *+{-\x_3+2\x_4},
						(36,-1.6) *+{-\x_2+2\x_4},
						(36,-2.4) *+{-\x_1+2\x_4},
						(36,-3.2) *+{\w+3\x_4},
						"34", {\ar@{--} "33"},
						"35", {\ar@{--} "34"},
						"36", {\ar@{--} "35"},
						"37", {\ar@{--} "36"},
						"38", {\ar@{--} "37"},
						"24", {\ar@{--} "23"},
						"25", {\ar@{--} "24"},
						"26", {\ar@{--} "25"},
						"27", {\ar@{--} "26"},
						"28", {\ar@{--} "27"},
						"29", {\ar@{--} "28"},
						"15", {\ar@{--} "14"},
						"16", {\ar@{--} "15"},
						"17", {\ar@{--} "16"},
						"18", {\ar@{--} "17"},
						"19", {\ar@{--} "18"},
					\end{xy}%
				}
			\caption{The Auslander--Reiten quiver of $\ACM \X$ for type $(2,2,2,4)$. A vertex labeled $\langle i,j \rangle$ denotes the $2$-extension bundle  $E\langle i\x_4 \rangle(j\x_4)$ for $0\le i \le 2$ and $j\in\Z$.
				The four line bundles $\OO(\x)$ and $\OO(\x+\bar{x}_{ab})$,
				where $\bar{x}_{ab}:=\c-\x_a-\x_b$ for $1\le a<b\le3$,
				occupy the same position, labelled $(\x)$ or $\bullet$ when
				$\x$ is unspecified. Every arrow starting or ending at  vertex $(\x)$ or $\bullet$ represents four arrows in the actual Auslander--Reiten quiver (each one connected to one of the $\OO(\x)$ and $\OO(\x+\bar{x}_{ab})$) }
			\label{The quiver of ACM X(2,2,2,4)}
				\end{figure}
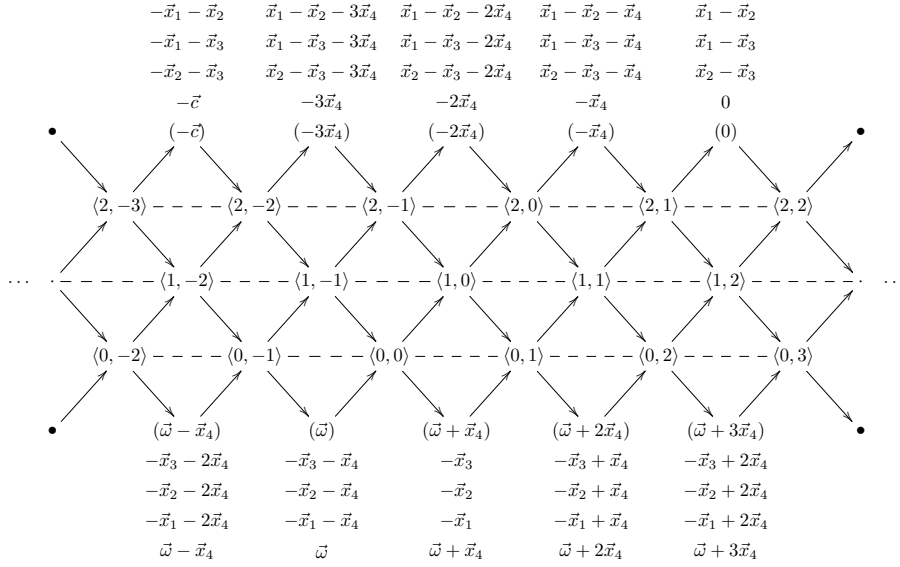	
	
			By \cite[Theorem~4.53]{HIMO}, the subcategory 
			\[
			\mathscr C'
			=
			\add\{\langle 0,2k \rangle, \langle 2,2k \rangle, \,\OO(\x)\mid k\in\Z,\ \x\in\L\}
			\]
			is a $2$-cluster tilting subcategory of $\ACM\X$. Indeed, by \cite[Theorem~6.2]{CRW},
			\[
			T=\langle0,0\rangle\oplus\langle2,0\rangle
			\oplus\langle0,2\rangle
			\]
			is a tilting object in $\underline{\ACM}\,\X$ with
			$\gl\underline{\End}(T)=2$. By Proposition \ref{susp.ext.}, we have
			$$\langle i,j \rangle(\w)=\langle 2-i,i+j-4 \rangle\quad\text{and}\quad\langle i,j \rangle(-\w)=\langle 2-i,i+j+2 \rangle.$$
			This gives
			$\mathscr C'=\add\{T(m\w),\ \OO(\x)
			\mid m\in\Z,\ \x\in\L\},$ so \cite[Theorem~4.53]{HIMO}
			 applies. Its Auslander--Reiten quiver is shown in Figure~\ref{The quiver of C_2} with			
			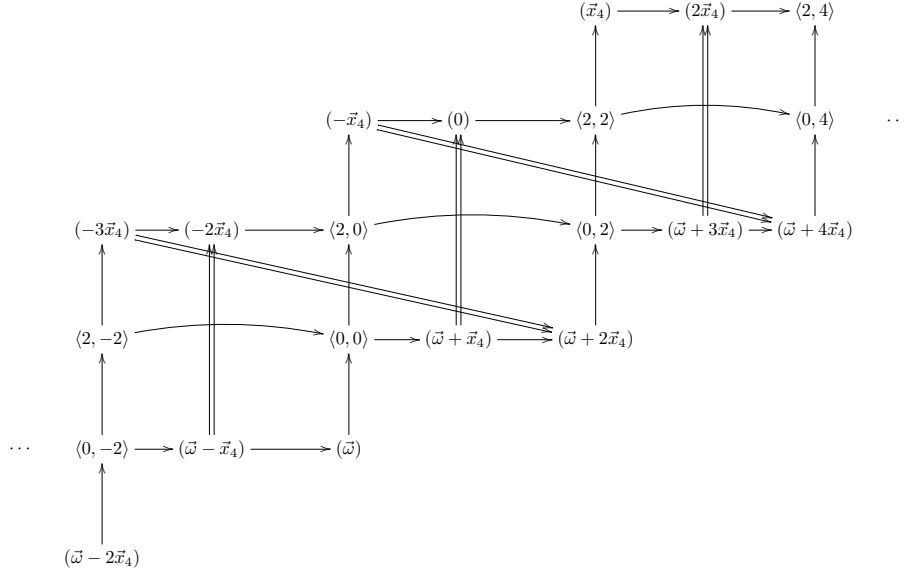
\begin{figure}
				\centering
				\resizebox{0.95\textwidth}{!}{%
					\begin{xy} 0;<0pt,30pt>:<30pt,0pt>::
						(2,-1.5) *+{\cdots},
						(8,14.5) *+{\cdots},
						(0,0) *+{(\w-2\x_4)} ="00",
						(2,0) *+{\langle 0,-2\rangle} ="10",
						(4,0) *+{\langle 2,-2\rangle} ="20",
						(6,0) *+{(-3\x_4)} ="30",
						(2,2) *+{(\w-\x_4)} ="11",
						(6,2) *+{(-2\x_4)} ="31",
						(2,4.5) *+{(\w)} ="02",
						(4,4.5) *+{\langle 0,0\rangle} ="12",
						(6,4.5) *+{\langle 2,0\rangle} ="22",
						(8,4.5) *+{(-\x_4)} ="32",
						(4,6.5) *+{(\w+\x_4)} ="13",
						(8,6.5) *+{(0)} ="33",
						(4,9) *+{(\w+2\x_4)} ="04",
						(6,9) *+{\langle 0,2\rangle} ="14",
						(8,9) *+{\langle 2,2\rangle} ="24",
						(10,9) *+{(\x_4)} ="34",
						(6,11) *+{(\w+3\x_4)} ="15",
						(10,11) *+{(2\x_4)} ="35",
						(6,13) *+{(\w+4\x_4)} ="06",
						(8,13) *+{\langle 0,4\rangle} ="16",
						(10,13) *+{\langle 2,4\rangle} ="26",
						"00", {\ar "10"},
						"10", {\ar "20"},
						"20", {\ar "30"},
						"20", {\ar@/^0.7pc/ "12"},
						"06", {\ar "16"},
						"16", {\ar "26"},
						"24", {\ar@/^0.7pc/ "16"},
						"35", {\ar "26"},
						"02", {\ar "12"},
						"12", {\ar "22"},
						"22", {\ar "32"},
						"04", {\ar "14"},
						"14", {\ar "24"},
						"24", {\ar "34"},
						"10", {\ar "11"},
						"30", {\ar "31"},
						"12", {\ar "13"},
						"22", {\ar@/^0.7pc/"14"},
						"32", {\ar "33"},
						"14", {\ar "15"},
						"15", {\ar "06"},
						"34", {\ar "35"},
						"13", {\ar@<-.3ex> "33"},
						"13", {\ar@<.3ex> "33"},
						"11", {\ar "02"},
						"31", {\ar "22"},
						"11", {\ar@<-.3ex> "31"},
						"11", {\ar@<.3ex> "31"},
						"13", {\ar "04"},
						"33", {\ar "24"},
						"15", {\ar@<-.3ex> "35"},
						"15", {\ar@<.3ex> "35"},
						"30", {\ar@<-.3ex> "04"},
						"30", {\ar@<.3ex> "04"},
						"32", {\ar@<-.3ex> "06"},
						"32", {\ar@<.3ex> "06"},
					\end{xy}%
				}
			\caption{The Auslander--Reiten quiver of $\mathscr C'$}
				\label{The quiver of C_2}
			\end{figure}	\noindent	
		the following conventions:
		for a $2$-extension bundle $U$, the situations				
		\[
		\begin{xy} 0;<50pt,0pt>:<0pt,-16pt>::
			(0,-1.25) *+{\OO(\x)} ="01",
			(0,0) *+{\OO(\x+\bar{x}_{12})} ="02",
			(0,1.25) *+{\OO(\x+\bar{x}_{13})} ="03",
			(0,2.5) *+{\OO(\x+\bar{x}_{23})} ="04",
			(1.2,0.75) *+{U} ="1",
			(3.4,-1.25) *+{\OO(\x)} ="11",
			(3.4,0) *+{\OO(\x+\bar{x}_{12})} ="12",
			(3.4,1.25) *+{\OO(\x+\bar{x}_{13})} ="13",
			(3.4,2.5) *+{\OO(\x+\bar{x}_{23})} ="14",
			(2.2,0.75) *+{U} ="2",
			"01", {\ar "1"},
			"02", {\ar "1"},
			"03", {\ar "1"},
			"04", {\ar "1"},
			"2", {\ar "11"},
			"2", {\ar "12"},
			"2", {\ar "13"},
			"2", {\ar "14"},
		\end{xy}
		\]				
		are written simply, respectively, as				
		\[
		\begin{xy} 0;<50pt,0pt>:<0pt,-16pt>::
			(4,0) *+{(\x)} ="4",
			(5,0) *+{U} ="5",
			(6.5,0) *+{U} ="6",
			(7.5,0) *+{(\x);} ="7",
			"4", {\ar "5"},
			"6", {\ar "7"},
		\end{xy}
		\]				
		for arrows between line bundles, the situations			
		\[
		\resizebox{0.98\textwidth}{!}{%
			\begin{xy} 0;<50pt,0pt>:<0pt,-16pt>::
				(4,-1.25) *+{\OO(\x)} ="21",
				(4,0) *+{\OO(\x+\bar{x}_{12})} ="22",
				(4,1.25) *+{\OO(\x+\bar{x}_{13})} ="23",
				(4,2.5) *+{\OO(\x+\bar{x}_{23})} ="24",
				(5.7,-1.25) *+{\OO(\x+\x_4)} ="31",
				(5.7,0) *+{\OO(\x+\bar{x}_{12}+\x_4)} ="32",
				(5.7,1.25) *+{\OO(\x+\bar{x}_{13}+\x_4)} ="33",
				(5.7,2.5) *+{\OO(\x+\bar{x}_{23}+\x_4)} ="34",
				(7.3,-1.25) *+{\OO(\x)} ="41",
				(7.3,0) *+{\OO(\x+\bar{x}_{12})} ="42",
				(7.3,1.25) *+{\OO(\x+\bar{x}_{13})} ="43",
				(7.3,2.5) *+{\OO(\x+\bar{x}_{23})} ="44",
				(9.5,-1.25) *+{\OO(\x+\x_3)} ="51",
				(9.5,0) *+{\OO(\x+\bar{x}_{12}+\x_3)} ="52",
				(9.5,1.25) *+{\OO(\x+\bar{x}_{13}+\x_3)} ="53",
				(9.5,2.5) *+{\OO(\x+\bar{x}_{23}+\x_3)} ="54",
				"21", {\ar "31"},
				"22", {\ar "32"},
				"23", {\ar "33"},
				"24", {\ar "34"},
				"41", {\ar "51"},
				"41", {\ar "53"},
				"41", {\ar "54"},
				"42", {\ar "52"},
				"42", {\ar "53"},
				"42", {\ar "54"},
				"43", {\ar "51"},
				"43", {\ar "52"},
				"43", {\ar "53"},
				"44", {\ar "51"},
				"44", {\ar "52"},
				"44", {\ar "54"},
			\end{xy}%
		}
		\]				
		are written simply, respectively, as				
		\[
		\begin{xy} 0;<50pt,0pt>:<0pt,-16pt>::
			(4,0) *+{(\x)} ="4",
			(5.3,0) *+{(\x+\x_4)} ="5",
			(6.7,0) *+{(\x)} ="6",
			(8,0) *+{(\x+\x_3).} ="7",
			"4", {\ar "5"},
			"6", {\ar@<-.2ex> "7"},
			"6", {\ar@<.2ex> "7"},
		\end{xy}
		\]
		
		Let $\mathscr X$ be the $2$-exact structure on $\mathscr C'$
		induced by $(\ACM\X,\mathscr E)$ as in Theorem \ref{n-exact from n-CT}. For any $0\le\ell\le2$ and
		$j\in\mathbb Z$,	 there are
		$\mathscr E$-conflations
		\[
		0\rightarrow\langle\ell,j\rangle
		\rightarrow I_{\ell,j}
		\rightarrow\langle2-\ell,j+\ell+1\rangle
		\rightarrow0
		\]
		by \cite[Theorem~4.4 and Proposition~4.6]{CRW}, where $I_{\ell,j}$ is the direct sum of line bundles represented by the two positions $((j-1)\x_4)$ and $(\w+(j+\ell+1)\x_4)$, that is,  
		\[
		\begin{aligned}
			I_{\ell,j}={}&\OO((j-1)\x_4)
			\oplus\bigoplus_{1\le i\le 3}\OO((j+\ell)\x_4-\x_i)\\
			&\oplus\bigoplus_{1\le i\le 3}\OO(\w+j\x_4+\x_i)
			\oplus\OO(\w+(j+\ell+1)\x_4).
		\end{aligned}
		\]
		Consequently, $\mho_{\mathscr E}^2(\langle\ell,2k\rangle)
		\simeq\langle\ell,2k+4\rangle$ for $\ell=0,2$ 
		and hence $\mho_{\mathscr E}^{2}(\mathscr C')\subset\mathscr C'$.
		By Theorem~\ref{n-exact from n-CT}, 
		$(\mathscr C',\mathscr X)$ is Frobenius $2$-exact, whose
		subcategory of projective-injective objects is $\mathscr L$.		
		By Corollary~\ref{F removable of ACM X} and Theorem \ref{Compatibility with cluster tilting}(a), $\mathscr F$
		is removable in $(\mathscr C',\mathscr X)$.  Thus	
		\[
		\mathscr C'/[\mathscr F]
		=\add\{\langle0,2k\rangle,\, \langle2,2k\rangle,\,
		\OO(a\x_4)\mid k,a\in\Z\}
		\]
		is  Frobenius
		$2$-exact with $2$-exact structure $\mathscr X_{\mathscr F}$ by Theorem~\ref{main thm2}. Its 	
	 subcategory of  projective-injective objects  is
		$
		\mathscr L/[\mathscr F]=\add\{\OO(a\x_4)\mid a\in\Z\}.
		$
		Its quiver is obtained by removing the vertices corresponding to fading
		line bundles and all arrows starting or ending at these vertices;
		see Figure~\ref{The quiver of  C_2F}.
		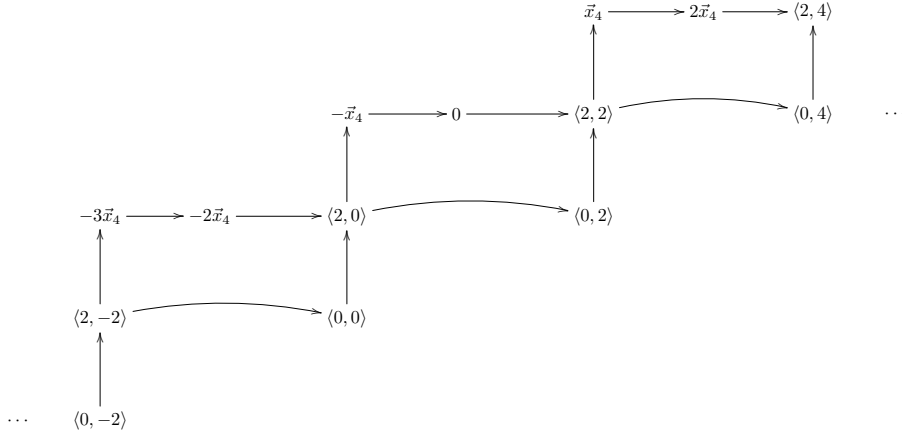
\begin{figure}[htbp]
			\centering
			\resizebox{0.95\textwidth}{!}{%
				\begin{xy} 0;<0pt,28pt>:<30pt,0pt>::
					(2,-1.5) *+{\cdots},
					(8,14.5) *+{\cdots},
					(2,0) *+{\langle 0,-2\rangle} ="10",
					(4,0) *+{\langle 2,-2\rangle} ="20",
					(6,0) *+{-3\x_4} ="30",
					(6,2) *+{-2\x_4} ="31",
					(4,4.5) *+{\langle 0,0\rangle} ="12",
					(6,4.5) *+{\langle 2,0\rangle} ="22",
					(8,4.5) *+{-\x_4} ="32",
					(8,6.5) *+{0} ="33",
					(6,9) *+{\langle 0,2\rangle} ="14",
					(8,9) *+{\langle 2,2\rangle} ="24",
					(10,9) *+{\x_4} ="34",
					(10,11) *+{2\x_4} ="35",
					(8,13) *+{\langle 0,4\rangle} ="16",
					(10,13) *+{\langle 2,4\rangle} ="26",
					"10", {\ar "20"},
					"20", {\ar "30"},
					"20", {\ar@/^0.7pc/ "12"},
					"16", {\ar "26"},
					"24", {\ar@/^0.7pc/ "16"},
					"35", {\ar "26"},
					"12", {\ar "22"},
					"22", {\ar "32"},
					"14", {\ar "24"},
					"24", {\ar "34"},
					"30", {\ar "31"},
					"22", {\ar@/^0.7pc/ "14"},
					"32", {\ar "33"},
					"34", {\ar "35"},
					"31", {\ar "22"},
					"33", {\ar "24"},
				\end{xy}%
			}
			\caption{The quiver of $\mathscr C'/[\mathscr F]$}
			\label{The quiver of  C_2F}
		\end{figure} 
		
		One can calculate from Figure \ref{The quiver of ACM X(2,2,2,4)} that  there are two $\mathscr{E}$-conflations
		\begin{align*}\label{two E-conflations}
		&	0\to \langle 0,0 \rangle \to \langle 2,0 \rangle \oplus F_0 \to \langle 1,1 \rangle \to 0,\\
		&	0\to \langle 1,1 \rangle \to \langle 0,2 \rangle \oplus \OO\oplus F_1 \to \langle 2,2 \rangle \to 0,
		\end{align*} 
		where $F_0$ is the direct sum of line bundles in the position $(\w+\x_4)$, while $\OO\oplus F_1$ is the direct sum of line bundles in the position $(0)$, that is, 
		$$F_0=\OO(\w+\x_4)\oplus\bigoplus_{1\le i\le 3}\OO(-\x_i)\quad\text{and}\quad 
		F_1=\bigoplus_{1\le i\le 3}\OO(\w+\x_4+\x_i).$$
		Splicing these two $\mathscr{E}$-conflations yields an
		$\mathscr E$-acyclic complex
		$$0\to\langle 0,0 \rangle \to \langle 2,0 \rangle \oplus F_0 \to \langle 0,2 \rangle \oplus \OO \oplus F_1 \to \langle 2,2 \rangle \to 0. $$ Since all its terms belong to $\mathscr C'$, it is an
		$\mathscr X$-conflation. Since $F_0,F_1\in \mathscr{F}$, its image in the quotient $\mathscr{C}'/[\mathscr{F}]$ is isomorphic to the $\mathscr{X}_{\mathscr{F}}$-conflation
		$$\langle 0,0 \rangle \to\langle 2,0 \rangle \to\langle 0,2 \rangle\oplus\OO \to\langle 2,2 \rangle.$$
			By Remark \ref{canonical equivalence}, the quotient functor $\pi_\mathscr{F}$ induces a canonical equivalence of 
		$4$-angulated categories
		$$\underline{\mathscr{C}'}=\mathscr{C}'/[\mathscr{L}]\simeq \underline{\mathscr{C}'/[\mathscr{F}]},$$
		see Figure~\ref{The quiver of  stable C_2}.
		\begin{figure}[htbp]
		\centering
		\resizebox{0.95\textwidth}{!}{%
			\begin{xy} 0;<0pt,28pt>:<30pt,0pt>::
				(2,-1.5) *+{\cdots},
				(8,14.5) *+{\cdots},
				(2,0) *+{\langle 0,-2\rangle} ="10",
				(4,0) *+{\langle 2,-2\rangle} ="20",
				(4,4.5) *+{\langle 0,0\rangle} ="12",
				(6,4.5) *+{\langle 2,0\rangle} ="22",
				(6,9) *+{\langle 0,2\rangle} ="14",
				(8,9) *+{\langle 2,2\rangle} ="24",
				(8,13) *+{\langle 0,4\rangle} ="16",
				(10,13) *+{\langle 2,4\rangle} ="26",
				"10", {\ar "20"},
				"20", {\ar@/^0.7pc/ "12"},
				"12", {\ar "22"},
				"22", {\ar@/^0.7pc/ "14"},
				"14", {\ar "24"},
				"24", {\ar@/^0.7pc/ "16"},
				"16", {\ar "26"},
			\end{xy}%
		}
		\caption{The quiver of the $4$-angulated category $\underline{\mathscr C'}$}
		\label{The quiver of  stable C_2}
		\end{figure}
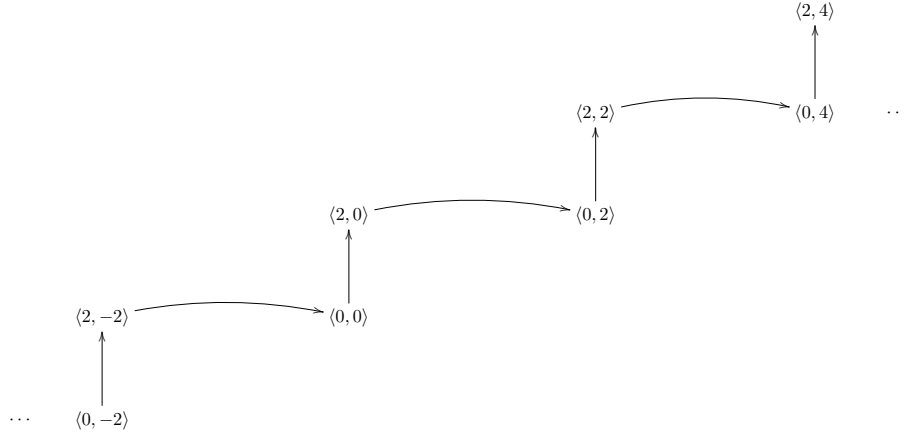
	 The suspension functor $\Sigma$ satisfies $\Sigma(\langle\ell,2k\rangle)
		\simeq\langle\ell,2k+4\rangle$ for $\ell=0,2$. Moreover, we have the $4$-angle
		\begin{equation}\label{4angle}
			\langle 0,0 \rangle \to\langle 2,0 \rangle \to\langle 0,2 \rangle \to\langle 2,2 \rangle\to\langle 0,4 \rangle\simeq \Sigma\langle 0,0 \rangle.
		\end{equation}		
		For each $2$-extension bundle $\langle\ell,2k\rangle$,
		with $\ell=0,2$ and $k\in\mathbb Z$,
		a $4$-angle with this first term is obtained by
		iterated left and right rotations of \eqref{4angle}.
		\end{example}

			\vskip 5pt
		\noindent {\scriptsize   \noindent  Weikang Weng\\
			School of Mathematical Sciences, \\
			Xiamen University, Xiamen, 361005, Fujian, PR China.\\
			E-mail: 
			wkweng@stu.xmu.edu.cn\\ }
		\vskip 3pt
		
	\end{document}